%% file: DJS_arXiv_final.tex
\documentclass[11pt,reqno,a4paper]{amsart}
\usepackage[margin=1in]{geometry}
\usepackage{amsmath, amsthm, amssymb}
\usepackage[colorlinks=true, pdfstartview=FitV, linkcolor=blue,citecolor=blue, urlcolor=blue]{hyperref}
\usepackage[abbrev,lite,nobysame]{amsrefs}
\usepackage{times}
\usepackage[usenames,dvipsnames]{color}
\usepackage{subcaption}
\usepackage{mathtools}
\usepackage{bbm}
\usepackage{esint}
\usepackage{xfrac}
\usepackage{enumitem}
\makeatletter
\newcommand{\mylabel}[2]{#2\def\@currentlabel{#2}\label{#1}}
\makeatother

\mathtoolsset{showonlyrefs=true}
\usepackage{dsfont}
\usepackage{tikz}
\usetikzlibrary{patterns}
\usetikzlibrary{fpu}
\usetikzlibrary{plotmarks}

\newcommand{\R}{\mathbb R}
\newcommand{\N}{\mathbb N}

\newcommand{\Z}{\mathbb Z}

\newcommand{\T}{\mathbb T}
\newcommand{\repsilon}{\epsilon}
\newcommand{\MAP}{F}
\DeclareMathOperator{\initial}{in}
\DeclareMathOperator{\adf}{}
\DeclareMathOperator{\adv}{adv}
\DeclareMathOperator{\heat}{heat}

 \input{macros}

\makeatletter
\def\@tocline#1#2#3#4#5#6#7{\relax
  \ifnum #1>\c@tocdepth 
  \else
    \par \addpenalty\@secpenalty\addvspace{#2}%
    \begingroup \hyphenpenalty\@M
    \@ifempty{#4}{%
      \@tempdima\csname r@tocindent\number#1\endcsname\relax
    }{%
      \@tempdima#4\relax
    }%
    \parindent\z@ \leftskip#3\relax \advance\leftskip\@tempdima\relax
    \rightskip\@pnumwidth plus4em \parfillskip-\@pnumwidth
    #5\leavevmode\hskip-\@tempdima
      \ifcase #1
       \or\or \hskip 1em \or \hskip 2em \else \hskip 3em \fi%
      #6\nobreak\relax
    \dotfill\hbox to\@pnumwidth{\@tocpagenum{#7}}\par
    \nobreak
    \endgroup
  \fi}
\makeatother

\author{Michele Dolce}
\address{Institute of Mathematics, EPFL, Station 8, 1015 Lausanne, Switzerland}
\email{michele.dolce@epfl.ch}

\author[Carl J. P. Johansson]{Carl Johan Peter Johansson}
\address{Institute of Mathematics, EPFL, Station 8, 1015 Lausanne, Switzerland}
\email{carl.johansson@epfl.ch}

\author{Massimo Sorella}
\address{Institute of Mathematics, EPFL, Station 8, 1015 Lausanne, Switzerland}
\email{massimo.sorella@epfl.ch}

\title[Dissipation enhancement for a class of Hamiltonian flows]{Dissipation enhancing properties for a class of Hamiltonian flows with closed streamlines}

\begin{document}
\maketitle
\begin{abstract}
We study the evolution of a passive scalar subject to molecular diffusion and advected by an incompressible velocity field on a 2D bounded domain. The velocity field is $u=\nabla^\perp H$, where $H$ is an autonomous Hamiltonian whose level sets are Jordan curves foliating the domain. We focus on the high Péclet number regime ($\mathrm{Pe}:=\nu^{-1} \gg 1$), where two distinct processes unfold on well separated time-scales: \textit{streamline averaging} and standard diffusion. For a specific class of Hamiltonians with one non-degenerate elliptic point (including perturbed radial flows), we prove exponential convergence of the solution to its streamline average on a subdiffusive time-scale $T_\nu \ll \nu^{-1}$, up to a small correction related to the shape of the streamlines. The time-scale $T_\nu$ is determined by the behavior of the period function around the elliptic point. To establish this result, we introduce a model problem arising naturally from the difference between the solution and its streamline average. We use pseudospectral estimates to infer decay in the model problem, and, in fact, this analysis extends to a broader class of Hamiltonian flows. Finally, we perform an asymptotic expansion of the full solution, revealing that the leading terms consist of the streamline average and the solution of the model problem.
\end{abstract}
\tableofcontents
\section{Introduction}
We consider the evolution of a passive tracer in a bounded domain which is advected by a regular Hamiltonian flow with closed streamlines in the large Péclet number $\mathrm{Pe}=LU/\kappa$ regime, where $L, U$ are the characteristic lenght and velocity scales and $\kappa$ is the diffusivity coefficient.  This is a paradigmatic example of a wide range of physical, chemical and biological processes and can be thought of as a toy model for some hydrodynamic stability problems around \textit{laminar flows}. For instance, in geophysical applications there are persistent lone eddies which are likely to homogenize tracers along their streamlines before standard diffusive processes take over \cite{rhines1983rapidly}. When $\mathrm{Pe}\gg 1$, it is often observed that the averaging along the streamlines of the flow happen on a time-scale $T_a$ much shorter than the diffusive time-scale $T_d=L^2/\nu$, a phenomenon which is linked to the \textit{effective diffusivity} studied in homogenization theory e.g. \cite{MR1265233}.  This separation in time-scales, also known as \textit{dissipation enhancement}, is related to the intricate interplay between diffusion and advection. In earlier investigations in the applied literature the precise scalings of the \textit{averaging time} $T_a$ in terms of $\mathrm{Pe}$ have been debated \cite{rhines1983rapidly}, but it is now clear that they heavily depend on specific properties of the advecting velocity field. From a mathematical point of view, the picture is well-understood for radial and shear flows. In these cases, $T_a$ is related to the behavior around critical points (or the absence of them) for the velocity profile. The great advantage of  radial and shear flows is that the average along the streamlines (average in the angular or shearing direction) commutes with the diffusion, meaning that the dynamics of the average is decoupled from the rest of the solution. Therefore, it is enough to quantify the decay rates of the solution minus its streamline average to estimate the separation of time-scales. This fact is crucially exploited in all the quantifications of $T_a$ for shear and radial flows, obtained through spectral methods \cite{Vukadinovic15,Gallay:2021aa,Colombo21,Feng23,He22}, \textit{hypocoercivity}  \cite{Bedrossian17,CZD20} and the quantitative version of H\"ormander's classical theory \cite{Albritton22}.

On the other hand, for more general Hamiltonian flows, this particular structure is lost and less is known.  An insightful quantitative homogenization-type result  was obtained in \cite{kumaresan2018advection,Vukadinovic21} for a class of Hamiltonians. They were able to show that there is a time $T_*$ with $T_a\ll T_*\ll T_d$ for which the tracer is indeed averaged along the streamlines in the $\mathrm{Pe}\to \infty$ limit, where the time $T_a$ is determined by the behavior around critical points of the \textit{period function} (the analogue of the velocity profile in shear and radial flows). This can be thought of as a form of dissipation enhancement, but the quantitative control they obtain is not uniform in time. We comment more about this result at the end of this introduction, where we will compare it with what we obtain in the present study.

The enhanced dissipation effect is much faster when the flow is \textit{turbulent} instead of laminar, but we keep the discussion short and somewhat informal given the purposes of this paper. In particular, we say that a flow is turbulent if it has `complicated trajectories', such as mixing flows \cite{constantin2008diffusion}. In these cases, there are results showing dissipation enhancement both in the deterministic \cite{elgindi2023optimal,CZDE20,constantin2008diffusion} and in the stochastic setting \cite{Bedrossina21Blum}. Finally,  also the \textit{anomalous dissipation} results  \cite{TDTEGIIJ22, CCS22, AV23, EL23, BBS23} obtained recently can be interpreted as an extreme form of dissipation enhancement.

\subsection{Main results} 

In this paper, we aim at giving a different perspective on Hamiltonian flows containing only elliptic points, naturally arising from perturbations of radial flows or in regions near elliptic points of general Hamiltonians. Let $M\subset \RR^2$ be a bounded domain with smooth boundary $\partial M$. Let $H \colon M \to \R$ be a $C^3$ function such that $H$ is constant on $\partial M$. Define $u = \nabla^{\perp} H = (-\de_{x_2},\de_{x_1}) H$ and let $\rho:[0,\infty)\times M\to \mathbb{R}$ be the solution to the advection diffusion equation
\begin{align}
\label{eq:advdiff}
\begin{cases}\de_t \rho+u\cdot\nabla \rho=\nu \Delta \rho, \qquad x \in M, \, t> 0,\\
\rho|_{t=0}=\rho^{in}, \qquad \rho|_{\partial M}=0 \text{ (or }\partial_n \rho|_{\partial M}=0).
\end{cases}
\end{align}
We consider either Dirichlet (or homogeneous Neumann) boundary conditions. Notice that $u$ is tangent to the boundary (since $H$ is constant on $\partial M$) and we have already properly adimensionalized the equations so that $\nu:=\mathrm{Pe}^{-1}$. We discuss more about precise hypotheses on $H$ in the sequel, but an essential assumption is that the level sets $\{H=h\}$ are Jordan curves foliating $M$. 

To understand the evolution of $\rho$, as we explained above, it is natural to introduce the average along the streamlines, corresponding to an integration over the level sets of $H$. For the class of Hamiltonians we consider here, as we show in Section \ref{sec:toolbox}, this is  equivalent to performing the orthogonal projection in $L^2(M)$ onto the kernel of the transport operator $u\cdot \nabla$. For instance,  all the functions of the type $F(H)$ are in $\mathrm{Ker}(u\cdot \nabla)$. Thus, we split a function $f\in L^2(M)$ as 
\begin{equation}
f:=P_0f+P_\perp f=f_0+f_\perp, \qquad P_0 f \in \mathrm{Ker}(u\cdot \nabla).
\end{equation}
This is also the natural splitting suggested by the inviscid problem, where $P_0 \rho^{in}$ is known to be, in the cases considered here, the weak limit in $L^2$ of the evolution of \eqref{eq:advdiff} with $\nu=0$ (quantitative decay rates of this convergence are related to  \textit{mixing} properties of $u$).
Projecting the equation \eqref{eq:advdiff}, one  has
\begin{align}
\label{eq:PperpP0}
\begin{cases}
\de_t\rho_\perp +u\cdot\nabla \rho_{\perp}=\nu P_\perp\Delta \rho_\perp+\nu P_\perp\Delta \rho_0,\\
\rho_\perp|_{t=0}=\rho_{\perp}^{in},
\end{cases} \qquad \qquad \begin{cases}
\de_t\rho_0 =\nu P_0\Delta \rho_0+\nu P_0 \Delta \rho_\perp,\\
\rho_0|_{t=0}=\rho_{0}^{in},
\end{cases}
\end{align}
with the same Dirichlet or Neumann boundary conditions and where we have used the fact that $P_0$ commutes with $u\cdot \nabla$, as is precisely shown in Section \ref{sec:toolbox}. In general, the commutator $[P_\perp,\Delta]=[\Delta,P_0]\neq0$ (which instead is zero for shear or radial flows) and therefore the dynamics of $\rho_\perp$ and $\rho_0$ are coupled. 

Our goal is to characterize the long-time behavior of $\rho_\perp$. A first guess, suggested by the behavior of shear and radial flows, would be that $\rho_\perp$ decays on a sub-diffusive time-scale depending on the properties of $u$. For instance
\begin{equation*}
\norm{\rho_{\perp}(t)}_{L^2}=\norm{(\rho - \rho_{0})(t)}_{L^2}\lesssim\e^{-\lambda_\nu t} \|\rho^{in}-\rho^{in}_0\|_{L^2}, \qquad \text{ with} \quad \nu/\lambda_\nu \to 0 \quad\text{as}\quad\nu \to 0,
\end{equation*}
where $\lambda_\nu$ is called \textit{enhanced dissipation rate}. To achieve this bound, a possibility is that the forcing term $\nu P_\perp\Delta\rho_0$ decays at least at the same rate, meaning that the interplay between the commutator $[P_\perp,\Delta]$ and the dynamics of $\rho_0$ result in some dissipation enhancing property of this forcing term, which can be quite challenging to quantify though. 
On the other hand, since for $\nu=0$ the $\rho_0$ is conserved in time, one can hope that  $\nu P_{\perp}\Delta \rho_0$ generates at worst a small term that one can subtract from the dynamics of $\rho_\perp$ and still get decay. This is still a delicate statement to formalize since higher order derivatives are involved, which in general can cost inverse powers of $\nu$ due to the formation of small scales. However, in account of the reasoning above, our first objective is to obtain the enhanced dissipation of $\rho_\perp $ up to a `small' correction $g_{\mathrm{corr}}$, namely
\begin{equation}
\label{eq:goal1}
\norm{(\rho - \rho_{0} - g_{\mathrm{corr}})(t)}_{L^2}\lesssim\e^{-\lambda_\nu t} \|\rho^{in}-\rho^{in}_0\|_{L^2}, \qquad \text{ with} \quad \nu/\lambda_\nu \to 0 \quad\text{as}\quad\nu \to 0.
\end{equation}
There is also a weaker result which would still be indicative of some dissipation enhancement generated by the flow. When $\rho_0^{in}=0$, all the errors  remain so small that at least on a sub-diffusive time $T_\nu=O(\lambda_{\nu}^{-1})$ one is able to prove 
\begin{equation}
\label{bd:disssub}
\norm{(\rho - \rho_{0})(T_\nu)}_{L^2} \leq (1-c_*)\|\rho^{in} \|_{L^2}, \qquad \text{ with } c_*\in (0,1).
\end{equation}
Namely, the $L^2$ norm is reduced at times $T_\nu\ll T_d=O(\nu^{-1})$, that is an instance of dissipation enhancement. In most of the cases the averaging time $T_a$ is in fact identified directly with $T_\nu$. For shear, radial and relaxation-enhancing flows\footnote{As defined in \cite{constantin2008diffusion}.} the bound \eqref{bd:disssub} would automatically imply the exponential decay  with rate $\lambda_\nu$ since $\rho_0(t)=0$ for all times if it is zero initially, see for instance \cite{Feng19,CZDE20}. Unfortunately, since $\rho_0$ is not conserved in our case, we cannot hope to iterate the estimate \eqref{bd:disssub} in a straightforward way, but still a bound like \eqref{bd:disssub} is an indication of dissipation enhancement.

Making the heuristic arguments above rigorous seems complicated in general, especially due to the lack of a precise understanding of the dynamics of $\rho_0$. In fact, we believe that studying quantitatively the evolution of $\rho_0$ is an interesting subject by itself. As an example, we also show that $\rho_0$ can be created and remain quantitatively lower bounded on a short-time interval even when starting from $\rho_0^{in}=0$. In Proposition \ref{thm:example1} we exhibit an example concentrated close to a hyperbolic point whereas in Proposition \ref{thm:example2} we exploit an elliptic point. 

Regarding our goals for $\rho_\perp$, in Corollary~\ref{cormain2} we show that bounds as in \eqref{bd:disssub} holds true for Hamiltonians in the class \ref{HamiltonianClassAm}, precisely defined in Definition~\ref{def:ClassHamiltoniansA}. To have $H$ in \ref{HamiltonianClassAm} we essentially require the following:
\begin{itemize}
\item Regularity, that is $H\in C^3$.
\item There is one non-degenerate elliptic critical point $x_0$, i.e. $\nabla H(x_0)=0$ and $\mathrm{det}(D^2 H)(x_0)>0$.
\end{itemize}
To prove the stronger result encoded in \eqref{eq:goal1}, we need to further assume a more precise control on $H$, roughly speaking:
\begin{itemize}
\item The modulus of the velocity at any point can be quantitatively controlled with its average along the streamline.
\end{itemize}
We will denote this class as \ref{HamiltonianClassPEps-m}  where $\eps$ essentially quantifies how close the Hamiltonian is to being a radial flow.
As we show in Section \ref{sec:examples},  any perturbation of a radial Hamiltonian of the form 
\begin{align} \label{intro:Hamiltonian}
H_\epsilon(r,\varphi)=H(r)+\epsilon f(r,\varphi), \qquad r,\varphi \text{ are the standard polar coordinates,}
\end{align}
is in the class \ref{HamiltonianClassPEps-m} with $\eps \simeq \epsilon$. Moreover, this behavior is typical of some non-degenerate Hamiltonians close to elliptic points.\footnote{For the cellular flow $H(x,y)=\sin(x)\sin(y)$, the parameter $\eps$ indicates how close one is to one of the elliptic points, say $(\pi/2,\pi/2)$. The streamlines are indeed smooth deformation of circles close to the elliptic point, but the behavior of the flow cannot be considered close to radial when one gets too close to the separatrices associated to the hyperbolic point.}

The index $m$ in both classes, is instead related to the enhanced dissipation rate $\lambda_\nu$, which takes into account the difference in velocity along different streamlines. This can be quantified by looking at the period and frequency functions defined respectively as
\begin{equation}
\label{def:period}
T(h):=\oint_{\{H=q(h)\}} \frac{\dd \ell}{|\nabla H|}, \qquad \Omega(h):=\frac{1}{T(h)}.
\end{equation}
where $q$ is some non-constant function related to the behavior of $H$ around its critical point (typically $q(h)=h^2$). Then,  an Hamiltonian in class \ref{HamiltonianClassAm}, which includes the class \ref{HamiltonianClassPEps-m}, is such that 
\begin{align*}
|\Omega'(h)|\simeq h^{m-1}, \qquad \text{for } \quad |h|\ll 1
\end{align*}
where we are assuming, without loss of generality, that the only critical point of $H$ is located at $h=0$.

In the class of Hamiltonians introduced above, by analogy with what is known for shear or radial flows, we expect that the (largest) averaging time $T_a$ is inversely proportional to 
\begin{equation}
\label{def:lambdanu}
\lambda_\nu:=\nu^{\frac{m}{m+2}}.
\end{equation}
Our first main result is the following. 
\begin{theorem}
\label{th:main}
Let $H$ be an Hamiltonian in the class \ref{HamiltonianClassPEps-m} as defined in Definition \ref{def:ClassHamiltoniansP} with $\eps\in [0,1/4)$, $m\geq1$ and let $\lambda_\nu$ be defined as in \eqref{def:lambdanu}. 
Then, there exists $c_0,\delta_*\in (0,1)$, $C>1$ and $g_{\mathrm{corr}}\in L^2(M)$ with $P_0 g_{\mathrm{corr}}=0$, such that the following bounds holds true:
\begin{align}
\label{bd:rhoperpmain}&\norm{(\rho - \rho_{0} -g_{\mathrm{corr}})(t)}_{L^2}\leq \e^{-\delta_*\lambda_\nu t+\pi/2}\| \rho^{in} - \rho_{0}^{in} \|_{L^2},\\
\label{bd:gcorr}&\norm{g_{\mathrm{corr}}(t)}_{L^2}\leq C\eps \e^{-c_0\nu t}\|\rho^{in}\|_{L^2},\\
 \label{bd:rho0main} &\norm{\rho_{0}(t)}_{L^2}\leq \e^{-c_0\nu t}\|\rho^{in}\|_{L^2}.
\end{align}
Moreover, if $\rho_0^{in}=0$, then the right-hand side of \eqref{bd:gcorr}-\eqref{bd:rho0main} is multiplied by an extra factor of $C\eps$.
\end{theorem} 
Let us give a few remarks about the results contained in the theorem above.
\begin{remark}[On the correction]
\label{rem:correction}
Looking at the equation satisfied by $\rho_\perp$ \eqref{eq:PperpP0},
 denoting $\cL_\perp=u\cdot \nabla -\nu P_\perp\Delta P_\perp$,  the natural way of defining the correction would be 
\begin{equation}
(\rho-\rho_0)(t)=\rho_{\perp}(t)=\e^{-t\cL_\perp}\rho_{\perp}^{in}+g_{\mathrm{corr}}(t), \qquad g_{\mathrm{corr}}(t)=\int_0^t\e^{-(t-s)\cL_\perp}(\nu P_{\perp}\Delta \rho_0)(s)\dd s.
\end{equation}
To prove \eqref{bd:rhoperpmain} from the formula above, on account of Wei's theorem \cite{WeiDiffusion19} (recalled in Theorem \ref{th:pseudoWei}) it is enough to understand \textit{pseudospectral properties} of $\cL_\perp$ and bounds on $P_\perp\Delta \rho_0$. As we explain below, the study of pseudospectral properties of $\cL_\perp$  is in fact a key point of our paper, and in particular estimates of the pseudospectral abscissa
\begin{equation}
\label{def:pseudointro}
\Psi(\cL_\perp)=\operatorname{inf}\{\|(\cL_\perp-i\lambda) f\|_{L^2_\perp} \, :\, f\in D(\cL_\perp), \, \lambda \in \mathbb{R}, \, \|f\|_{L^2_\perp}=1 \},
\end{equation}
where $L^2_{\perp}(M)$ is defined in \eqref{def:L2perp}. However, we are not able to provide good uniform bounds directly on $P_\perp\Delta \rho_0$. In fact, if $\norm{P_\perp\Delta \rho_0}_{L^2} \lesssim \| \rho^{in} \|_{H^2}$, thanks to Theorem \ref{th:main2}, we would be able to prove that 
\begin{equation}
\norm{g_{\mathrm{corr}}(t)}
_{L^2}\lesssim \nu \int_0^t \e^{-\delta_*\lambda_\nu(t-s)} \norm{P_\perp\Delta \rho_0(s)}_{L^2} \dd s\lesssim \nu^{\frac{2}{m+2}} \| \rho^{in} \|_{H^2}.
\end{equation}
Thus, the correction would go to zero as $\nu\to 0$ meaning that one would have a quantitative homogenization-type result as well. But obtaining uniform bounds (with respect to $\nu$) for $P_\perp\Delta \rho_0$ requires  a deep understanding of the dynamics of $\rho_0$ and, again, its interplay with $[P_0,\Delta]$. We overcome this obstacle by defining the correction through an \textit{asymptotic expansion} in powers of $\eps$, where the leading order term is determined by the `homogeneous' problem without the forcing term $P_\perp\Delta\rho_0$. We expand  in $\eps$ because it is the smallness parameter we gain in energy estimates from a commutator related to $[P_0,\nabla]$. See Section \ref{sec:proofmain} for details, where Theorem \ref{th:main} is proved.
\end{remark}
\begin{remark}[On initial data with zero streamlines average]
When $\rho_0^{in}=0$, we have an extra factor $\eps$ in the estimates \eqref{bd:gcorr}-\eqref{bd:rho0main}. This implies that, for $C_*>1$ sufficiently large and $\eps$ sufficiently small, at times $T_\nu=C_*\nu^{-\frac{m}{m+2}}$, we get 
\begin{equation}
\norm{\rho_\perp(T_\nu)}_{L^2}\leq \frac12\|\rho^{in}\|_{L^2}.
\end{equation}
Hence,  we also have the bound \eqref{bd:disssub}.
\end{remark}

\subsection{Model problems}
As mentioned in Remark \ref{rem:correction}, a key point in the proof of Theorem \ref{th:main} is the study of pseudospectral properties of the operator 
\begin{equation}
\label{def:Lperp}
\cL_\perp=u\cdot \nabla-\nu P_\perp\Delta P_\perp,
\end{equation}
with Dirichlet or Neumann boundary conditions, and $\cL_{\perp} : D(\cL_\perp)\subset L^2_\perp(M)\to L^2_\perp(M)$  with
\begin{equation}
\label{def:L2perp}
L^2_{\perp}(M):= \{ f\in L^2(M) \, : \, P_0 f=0\}.
\end{equation}
In particular, we aim at proving enhanced dissipation for the model problem  
\begin{equation}
\label{eq:model}
\begin{cases}
\de_tg +\cL_{\perp}g=0, \qquad x \in M, \, t> 0,\\
g|_{t=0}=\rho_{\perp}^{in}, \qquad g|_{\partial M}=0 \text{ (or }\partial_n g|_{\partial M}=0).
\end{cases} 
\end{equation}
Our second main result, which holds for the more general class \ref{HamiltonianClassAm}, is the following:
\begin{theorem}
\label{th:main2}
Let $H$ be an Hamiltonian in the class \ref{HamiltonianClassAm} as in Definition \ref{def:ClassHamiltoniansA} with  $m\geq1$ and let $\lambda_\nu$ be defined as in \eqref{def:lambdanu}.  Let $\rho_{\perp}^{in}\in L^2_\perp(M)$ and $g$ be the solution to \eqref{eq:model}.  Then, there exists $\delta_*\in (0,1)$ such that
\begin{align}
\label{bd:gmain} \norm{g(t)}_{L^2}\leq \e^{-\delta_* \lambda_{\nu} t + \pi/2}\|\rho_\perp^{in}\|_{L^2}.
\end{align}
\end{theorem} 
The precise decay estimate on $g$ follows by the lower bound on the pseudospectral abscissa that we prove in Proposition \ref{prop:main}, which is one of the main challenges of this paper. 
Moreover, the solution $g$ to the model problem \eqref{eq:model} is behind the definition of our correction $g_{\mathrm{corr}}$. Indeed, as we show in Section \ref{sec:proofmain}, we will define 
\begin{equation}
g_{\mathrm{corr}}(t):=(\rho-\rho_0-g)(t).
\end{equation}
A consequence of Theorem \ref{th:main2} is a  property analogous to \eqref{bd:disssub} for Hamiltonians in the class \ref{HamiltonianClassAm}.
\begin{corollary}
\label{cormain2}
Under the assumptions of Theorem \ref{th:main2}, let $\rho$ be the solution to \eqref{eq:advdiff} with $\rho^{in}=\rho_{\perp}^{in}$. Then, there exists two constants $c_*\in (0,1)$ and $C_*>1$ such that for $T_\nu=C_{*}\nu^{-\frac{m}{m+2}}$ the following holds true
\begin{equation}
\norm{(\rho - \rho_{0})(T_\nu)}_{L^2}\leq (1-c_*)\norm{\rho^{in} }_{L^2}.
\end{equation}
\end{corollary}
Naively, one would like to iterate the result above and conclude the dissipation enhancement property on the original solution for the larger class \ref{HamiltonianClassAm}. However, since $\rho_0$ is not conserved, an iteration would require a quantitative control of its evolution. Thus, we also need to introduce the natural model problem for $\rho_0$, which is given by
\begin{equation}
\label{eq:model00}
\begin{cases}
\de_t\eta=\nu P_0\Delta P_0 \eta:=-\cL_0\eta , \qquad x \in M, \, t> 0,\\
\eta|_{t=0}=\rho_{0}^{in}, \qquad \eta|_{\partial M}=0 \text{ (or }\partial_n \eta|_{\partial M}=0).
\end{cases} 
\end{equation}
This is the leading order approximation of $\rho_0$ in the asymptotic expansion we perform in the proof of Theorem \ref{th:main}, in the sense that $\|\rho_0-\eta\|_{L^2}\lesssim \eps \|\rho^{in}\|_{L^2}$. 
\begin{remark}
We believe that the model problem \eqref{eq:model00} effectively describes the solution to \eqref{eq:advdiff} on sub-diffusive time-scales after the averaging time $T_\nu$. For instance, in \cite{rhines1983rapidly} they assume that $\rho_0$ remains constant on sub-diffusive time-scales, whereas the equation \eqref{eq:model00} keeps into account a possible non-trivial evolution related to the geometry of the streamlines of the flow.
\end{remark}
Finally, by assuming a priori an upper bound on $\|\rho_0-\eta\|_{L^2}$, we are also able to prove the following conditional result.
\begin{proposition}
\label{prop:conditional}
Let $H$ be an Hamiltonian in the class \ref{HamiltonianClassAm} as in Definition \ref{def:ClassHamiltoniansA} with $m\geq1$ and let $\lambda_\nu$ be defined as in \eqref{def:lambdanu}. Let $\rho$ be the solution to \eqref{eq:advdiff} with initial data $\rho^{in}\in H^1$ such such that $\|\rho^{in}\|_{H^1}\leq C \|\rho^{in}\|_{L^2}$ for some constant $C>0$. Assume that there exists $\alpha \in (0,1)$ such that 
\begin{equation}
\label{eq:conditional}
\| (\rho_0-\eta)(t)\|_{L^2}\leq \nu^\alpha \|\rho^{in}\|_{L^2}, \qquad \text{for } t\in [0,\nu^{-1}).
\end{equation}
where $\eta$ is the solution to \eqref{eq:model00}. Then, there exist $\delta_*\in (0,1)$, $C_1>0$ such that 
\begin{equation}
\|(\rho-\eta)(t)\|_{L^2}\leq C_1\max\{\e^{-\delta_*\lambda_\nu t/2}, \nu^{\alpha/2}, \nu^{\frac{1}{m+2}}\} \|\rho^{in}\|_{L^2}, \qquad \text{for any } t\in [0,\nu^{-1}).
\end{equation}
As a consequence
\begin{equation}
\label{eq:homogenization}
\lim_{\nu \to 0} \sup_{t\in [\lambda_\nu^{-1^{-}},\nu^{-1})}\|(\rho-\eta)(t)\|_{L^2}=0.
\end{equation}
\end{proposition}
The result in \eqref{eq:homogenization} is a homogenization type result towards the effective dynamics dictated by the solution to our model problem \eqref{eq:model00}. Thus, we have a very precise control on the evolution of  $\rho$ once assumption \eqref{eq:conditional} is verified, which however we believe it can be quite challenging to prove. 
\begin{remark}
 Notice that the condition \eqref{eq:conditional} is much weaker than the one needed in Remark \ref{rem:correction}. Moreover, to prove Theorem \ref{th:main} we verify \eqref{eq:conditional} with $\nu^{\alpha}$ replaced by $\eps$. In particular, if we deform a radial flow with a perturbation of size $\nu^\alpha$, we can still recover the homogenization type result encoded in \eqref{prop:conditional}.
\end{remark}

\subsection{Main difficulties and discussion} Let us first highlight the main difficulties by comparison with the known case of radial flows. 
 For radial flows (included in the $\eps=0$ case), the pseudospectral abscissa of $\cL_\perp$ can be estimated by the one of the operators $ik \Omega(r)-\nu (\de_{rr}+r^{-1}\de_r-k^2/r^2)$ with $k\in \mathbb{Z}\setminus\{0\}$, see \cite{Gallay:2021aa}. This can be easily seen by passing to polar coordinates and taking the angular Fourier transform. The decoupling in $k$ is a property crucially related to the geometry of the streamlines, thanks to which the operator $\cL_\perp$ leaves invariant the space of $k$-fold symmetric functions (in the angular direction). 
 The decoupling in $k$ is fundamentally used to study pseudospectral bounds on a given finite number of  isolated thickened level sets, e.g. a $\delta$-neighbourhood of $\{|\Omega(r)-\lambda/k|\lesssim \delta^m\}$, and cleverly split the domain in two regions, in which the following happens:
\begin{enumerate}
\item In the thin sets, smallness can be gained via standard Poincar\'e type inequalities.
\item Outside of the thin sets, the transport operator can be cleverly used to get a good upper bound. 
\end{enumerate}
The combination of the  two points above lead to the sharp estimate on the pseudospectral abscissa after optimizing the size of the thin sets.

The natural connection between Hamiltonian and radial flows is given by an action-angle change of variables $(x,y)\to(h,\theta)$. This transforms the operator $u\cdot \nabla$  in $\Omega(h)\de_\theta$ and hence $ik\Omega(h)$ after a Fourier transform in the angle variable. The commutator $[P_k , \Delta]$, where $P_k$ denotes the projection on the $k$-th Fourier mode along the angular variable, is usually non zero. Therefore, the diffusion operator $ P_k \Delta$ does not decouple angular frequencies, meaning that the dynamics of the $k$-th Fourier mode is affected by the dynamics of all the other Fourier modes. As a consequence, a $k$-by-$k$ analysis is  challenging. 
However, the commutator  $[P_\perp, \Delta] $ can be uniformly controlled in terms of $\varepsilon$ (related to the $\epsilon$ in \eqref{intro:Hamiltonian} as explained above). 
Then, in our case we cannot think of $k$ as a simple rescaling of the spectral parameter $\lambda$ and instead we have to carefully redefine an arbitrary large number of thickened level sets (of cardinality approximately $|\lambda| \in \R_+$), avoiding their overlap as well. Here we crucially use the non-degeneracy assumption on the Hamiltonian, without which our sets are not mutually disjoint. 

We then need to quantify precisely the relation between thin sets in action-angle variables and their size in standard Cartesian coordinates,  where the geometry of the streamlines play a fundamental role. This is precisely done in Section \ref{sec:toolbox}, where we introduce our \textit{rescaled action-angle} variables that allow us to prove useful Poincar\'e type inequalities (that can be of independent interest). In particular, we scale properly our change of variables to avoid potential degeneracies when taking derivatives in the action direction. This is crucial to directly link the size of a thin sets in both set of coordinates and treat the action variable in a more uniform way. This idea can be technically useful in other contexts and we believe it is one of the main reasons why in the radial case it is more convenient to work in polar coordinates instead of a canonical action-angle change of variables (whose Jacobian is $1$ instead of $r$). 
Finally, to perform the desired estimates outside the thin sets, we use a Fourier multiplier method inspired by \cite{Gallay:2021aa, LiPseudo20, gallay2019estimations}, but here the multiplier is implicitly defined in rescaled action-angle variables. Again, the definition of the rescaled action-angle variables simplifies the estimate for derivatives of the Fourier multipliers.

\medskip

We conclude this introduction by comparing our result and the ones obtained by Vukadinovic and Kumaresan in \cite{kumaresan2018advection,Vukadinovic21}, which also establish dissipation enhancement (or homogenization) properties for more general Hamiltonians. In both cases, an approximating problem exhibiting dissipation enhancement is introduced and used to deduce quantitative estimates on \eqref{eq:advdiff}. The first main difference is the  operator replacing the full diffusion operator $\Delta$, which for us is $P_\perp\Delta P_\perp$. On the other hand, in \cite{kumaresan2018advection,Vukadinovic21} they introduce an `averaged' version of the Laplacian, denoted by $\langle\Delta\rangle$, which is related to the \textit{effective diffusion equation} found by Friedlin and Wentzell \cite{Freidlin94}. The way $\langle\Delta \rangle$ is precisely defined is  technical, but it can be easily explained informally. Indeed, let $(J,\theta)$ be the \textit{canonical} action-angle coordinates associated to the Hamiltonian $H$. Then, one has $\Delta \to A_H(J,\theta)\de_{JJ}+B_H(J,\theta)\de_{J\theta}+C_H(J,\theta)\de_{\theta\theta}$. Denoting with $\langle \cdot \rangle$ the averaging in $\theta$ (which is related to our $P_0$), the averaged diffusion operator is defined as $$\langle \Delta \rangle := \langle A_H\rangle(J)\de_{J\theta}+\langle B_H \rangle(J)\de_{J\theta}+\langle C_H \rangle(J)\de_{\theta\theta}.$$
The great advantage of the operator above is that it decouples angular frequencies, meaning that we can reduce everything to a one-dimensional problem as in the radial flow case discussed above. Moreover, the analysis of the model problem can be done for  a general class of Hamiltonians with  restrictive assumptions on the initial data (ruling out our initial data in the examples in Proposition \ref{thm:example1} and Proposition \ref{thm:example2}). For instance, the cellular flow in $\mathbb{T}^2$ defined as $H=\sin(x)\sin(y)$ is included in their analysis and a dissipation enhancement rate $\lambda_\nu=\nu^{1/2}$ can be proved for their model problem. 

The proof in \cite{kumaresan2018advection,Vukadinovic21} is based on the hypocoercivity method, which crucially requires a particular hypothesis on the initial data that exclude something concentrated close to the hyperbolic point. More precisely, the so-called `$\gamma$-terms' in the hypocoercive energy functional require that  
$$ \norm{\Omega' \rho^{in}}_{L^2}\leq C,$$
where $\Omega'$ is the derivative of the frequency function whereas $C$ is a constant independent of $\nu$. Now, close to an hyperbolic point at $h=0$, by the estimates in \cite{kumaresan2018advection,Vukadinovic21} we know that $|\Omega'|(h) \simeq 1/|h(\ln(h))^2|$. Therefore, the weight $\Omega'$ degenerates in presence of hyperbolic points and we cannot consider data really concentrated close to it. The initial data proposed in  Proposition \ref{thm:example1} does not satisfy the desired bound with $C$ independent of $\nu$,  where we show a pathologic behaviour not included in \cite{kumaresan2018advection,Vukadinovic21}.

Regarding the use of the model problem to deduce bounds for solutions to the advection--diffusion equation \eqref{eq:advdiff}, the authors in \cite{kumaresan2018advection,Vukadinovic21} 
assume higher order derivatives bounds independent on $\nu$ on the initial data (see for instance \cite[Equation (66)]{Vukadinovic21}) and
need to restrict themselves to the case $m\in \{0,1\}$\footnote{$m$ is the order of vanishing, see for instance \ref{item:ClassACondOnDerivative}.}. This is related to the fact that all the errors include derivatives whose bounds cannot be uniform in $\nu$. Thus, there is a delicate balance of powers of $\nu$. For instance, for $m=1$ (the case of the cellular flow) one error\footnote{See the bound for the term $c_{c1}^{(2)}$ in the proof of Theorem 22 in \cite{Vukadinovic21}.} is bounded by $\nu^{\frac{14}{30}}\sqrt{t}$, which is  large for any $t>\nu^{-\frac{28}{30}}$ (and exploding as $\nu\to 0$ in diffusive time scales). Therefore, one drawback of the approximation with $\langle \Delta\rangle$ is that it is not clear how to extend it to $m\geq 2$ and how to control the errors uniformly in time.

\subsection*{Plan of the paper}
In Section \ref{sec:toolbox} we introduce the main technical tools needed in the rest of the paper. Here, we introduce a version of action-angle coordinates and we present Poincaré type inequalities in thin sets adapted to the shape of the streamlines. In Section \ref{sec:proofmain}, by assuming the result in Theorem \ref{th:main2}, we prove Theorem \ref{th:main} and Corollary \ref{cormain2}. In Section \ref{sec:pseudo} we present the pseudospectral bounds for the operator $\cL_\perp$, and consequently prove Theorem \ref{th:main2}. In Section \ref{sec:examples} we present some examples, where we first show that perturbations of radial flows are in the class \ref{HamiltonianClassPEps-m}.  Then, we show the results related to the dynamics of $\rho_0$ we mentioned in the discussion above. We finally include some concluding remarks in Section \ref{sec:conclusion}.

\section{Toolbox and hypotheses on the Hamiltonians}
\label{sec:toolbox}
In this section, we first introduce the main technical tools and notation used in the rest of the paper. This include a coordinate change which simplifies computations in comparison to the standard action-angle coordinates when computing spectral estimates. With this notation at hand, we can finally state the main assumptions for the Hamiltonian $H$. Then, we prove useful results regarding Fourier multipliers which will be used in the pseudo-spectral estimates. Finally, we derive Poincar\'e inequalities in thin streamline sets.

\subsection{Rescaled action-angle variables}\label{sec:rescaled-action-angle}
When studying the problem \eqref{eq:advdiff}, the standard procedure is to consider the canonical action-angle variables, e.g. \cite{Vukadinovic21}, which is a  coordinate change (in general non-global) adapted to the level-curves of the Hamiltonian. For an Hamiltonian whose closed streamlines foliate the domain $M$ up to a zero measure set, it is well known that the action variable is the area enclosed within the streamline(s) $\{ H=s \}$\footnote{For instance, given a radial Hamiltonian $H=H(r)$, the canonical action-angle variables are the coordinates $x=\sqrt{2r}\cos(\theta), y=\sqrt{2r}\sin(\theta)$.}.
 For our purposes, we find it more convenient to work in a set of coordinates which resembles more the polar coordinates and are particularly adapted to the behavior near elliptic points of the Hamiltonian. We refer to this set of coordinates as \emph{rescaled action-angle variables} and the goal of this section is to introduce them. The technical advantage is to avoid carrying over the area function and it will be easier to quantify precisely derivatives in the new reference frame. 

To define rescaled action-angle variables we need a function $q$ defined on the image of $H$ which is well-adapted to the behaviour of the Hamiltonian around critical points. For the classes of Hamiltonians in this study, we select $q$ in the following manner. Let $M \subset \R^2$ be a bounded domain with smooth boundary. Let $H \in C^2(\overline{M})$ be an Hamiltonian with only one elliptic and non-degenerate critical point $x_0$. Set $h_0 = H(x_0)$ and $I = H(M \setminus \{ x_0 \} )$. We then define the function $q$ as
\[
 q(h) = (h - h_0)^2.
\]
We remark that the function $q$ is tailored to the Taylor expansion of the Hamiltonian around the elliptic point. In particular, since the elliptic point is non-degenerate, $q$ is selected to be quadratic. Other choices are possible and without doubt useful in different contexts. For degenerate elliptic points, selecting a higher order monomial depending on the degeneracy seems suit the problem while for Hamiltonians with several critical points a polynomial is required.

\begin{definition}[Period and frequency function]\label{def:PeriodAndFrequency}
Under the assumptions above, we define the period $T \colon I \to \R^+$ and frequency $\Omega \colon I \to \R^+$ functions as
\begin{equation}\label{eq:PeriodFunctionAndFrequencyFunction}
   T(h) = \oint_{\{ H = q(h) \}} \dfrac{1}{|\nabla H|} \, d \mathcal{H}^1, \quad \text{and} \quad \Omega(h) = \dfrac{1}{T(h)}.
\end{equation}
\end{definition}
 
\begin{definition}[rescaled action-angle change of coordinates]
 Let $M \subset \R^2$ be a bounded domain with smooth boundary. Let $H \in C^2(\overline{M})$ be an Hamiltonian with only one elliptic and non-degenerate critical point $x_0$.
Then, a bijective mapping $\Phi \colon \T \times I \to M \setminus \{ x_0 \}$ is called a \textit{rescaled action-angle} change of coordinates w.r.t. to $H$ if it satisfies the following three conditions:
 \begin{enumerate}
  \item $H(\Phi(\theta, h)) = q(h)$ for all $(\theta, h) \in \T \times I$;
  \item $\partial_{\theta} \Phi(\theta,h) = T(h) \nabla^{\perp} H (\Phi(\theta,h))$  for all $(\theta, h) \in \T \times I$ where $T \colon I \to \R$ is defined in \eqref{eq:PeriodFunctionAndFrequencyFunction};
  \item $\partial_h \Phi \in L^{\infty}(\T \times I)$.
 \end{enumerate}
\end{definition}

\begin{remark}
 The inverse of a rescaled action-angle change of coordinates can be written as 
 $$M \setminus \{ x_0 \} \ni x \mapsto (\theta(x), h(x)) \in \T \times I$$
 where $x_0$ is the elliptic point. However, to avoid confusing notation at a later stage, we will denote  $\Phi^{-1}:=\InverseMap = (\InverseMap_{1}, \InverseMap_{2}) \colon M \setminus \{ x_0 \} \to \T \times I$.
\end{remark}
Let us now show that this change of coordinates exists under appropriate conditions on the period function.
\begin{lemma}
 Let $M$ be a bounded domain with smooth boundary. Let $H \in C^2(\overline{M})$ be an Hamiltonian with only one elliptic and non-degenerate critical point $x_0$. Additionally,  assume that
\begin{equation}\label{eq:AssumptionPeriod}
 0 < \inf_{h \in I} T(h) \leq \sup_{h \in I} T(h) < \infty \quad \text{and} \quad \sup_{h \in I} T^{\prime}(h) < \infty.
\end{equation} 
Then, there exists a rescaled action-angle change of coordinates w.r.t. to $H$.
\end{lemma}
\begin{proof}
Let $\tilde{z} \in M$ be an arbitrary point which is not a critical point of $H$. Take $\tilde{h}$ such that $$q(\tilde{h}) = H(\tilde{z})$$ and define $z \colon I \to M$ as the maximally extended solution of
\begin{align}
\begin{cases}
 \dfrac{d}{dh}{z}(h) = q^{\prime}(h) \dfrac{\nabla H}{|\nabla H|^2}(z(h)); \\
 z(\tilde{h}) = \tilde{z}.
\end{cases}
\end{align}
 Observe that 
 \begin{equation}\label{eq:H-and-z-are-consistent}
 H(z(h)) = q(h).
 \end{equation}
 Indeed
 \[
  \frac{d}{dh} \left[ H(z(h)) - q(h) \right] = \dot{z}(h)\cdot \nabla H (z(h))  - q^{\prime}(h) = 0.
 \]
 As a consequence, 
 \[
  H(z(h)) - q(h) = H(z(\tilde{h})) - q(\tilde{h}) = 0.
 \]
 Moreover, we see that
 \begin{equation}\label{eq:DerivativeOfZBounded}
  \sup_{h \in I} \dot{z}(h) < \infty
 \end{equation}
 Now, let $X \colon [0, \infty) \times M \to M$ be the flow of associated to the velocity field $u=\nabla^{\perp} H$, which solves
 \begin{align}
\begin{cases}
 \partial_t {X}(t,x) = \nabla^{\perp} H (X(t,x)); \\
 X(t,x) = x.
\end{cases}
\end{align}
The period function $T \colon I \to \R^{+}$ is then given by
\begin{equation}
\label{def:period1}
 T(h) = \inf \{ t \in (0, \infty) : X(t,z(h)) = X(0,z(h)) \}.
\end{equation}
Indeed,
\[
 \oint_{\{ H = q(h) \}} \dfrac{1}{|\nabla H|} \, d \mathcal{H}^1 = \int_{0}^{T(h)} \dfrac{1}{|\nabla H (X(t, z(h)))|} |\partial_t X(t, z(h))| \, dt = T(h).
\]
Define $\Phi \colon \T \times I \to M$ as 
\begin{equation}
\label{def:Phi}
\Phi(\theta, h) = X(\theta T(h), z(h)).
\end{equation}
The fact that $\Phi$ is bijective is clear.
Now, we observe that 
\begin{align}
 \partial_{\theta} \Phi(\theta, h) &= \nabla^{\perp} H (\Phi(\theta, h)) T(h); \label{eq:DerivativeWRTTheta} \\
 \partial_h \Phi(\theta, h) &= \theta T^{\prime}(h) \nabla^{\perp} H (\Phi(\theta, h)) + D_x X (\theta T(h), z(h)) \dot{z}(h). \label{eq:DerivativeWRTc}
\end{align}
Hence point (1) is already satisfied. Let us now verify that (2) is fulfilled.
Since $H(\Phi(0, h)) = H(z(h)) = q(h)$ and 
\[
 \dfrac{\partial}{\partial \theta} \Big[ H(\Phi(\theta, h)) \Big] = \nabla H(\Phi(\theta, h)) \cdot \partial_{\theta} \Phi(\theta, h) = 0,
\]
we deduce that 
\begin{equation}
H(\Phi(\theta, h)) = H(z(h)) = q(h) \qquad \text{for all } (\theta, h) \in \T \times I.
\end{equation}
Hence point (2) is satisfied. We now prove (3). 
A standard Gr\"onwall estimate yields
\[
 |D_x X(t,x)| \leq \e^{\| D^2 H \|_{C^0} t}.
\]
Combining this with the fact that both $T$ and $T^{\prime}$ are bounded (see Equation \eqref{eq:AssumptionPeriod}) gives us (3) thanks to \eqref{eq:DerivativeOfZBounded} and \eqref{eq:DerivativeWRTc}.
\end{proof}

\begin{remark}
 We note that although rescaled action-angle variables always exist and can be constructed using the technique in the proof above, they are not unique. 
 In fact, when considering perturbations of radial flows as in Section \ref{sec:Radial}, it will be more convenient to find rescaled action-angle variables as a perturbation of the classical radial coordinates rather than applying the method above. Moreover, the hypotheses on the period could be relaxed.
\end{remark}
The rescaled action-angle change of coordinates satisfies some useful properties we state below.
\begin{lemma}
 Let $M$ be a bounded domain with smooth boundary. Let $H \in C^2(\overline{M})$ be an Hamiltonian with only one elliptic and non-degenerate critical point. Then for any rescaled action-angle change of coordinates w.r.t. $H$ denoted by $\Phi \colon \T \times I \to M \setminus \{ x_0\}$ and its inverse $\InverseMap \colon M \setminus \{ x_0\} \to \T \times I$ we have
 \begin{align}
  \det J_{\Phi}(\theta, h) &= - q^{\prime}(h) T(h); \label{eq:JacobianLemmaOne} \\
   \nabla \InverseMap_1 (\Phi(\theta,h)) &= 
 \dfrac{1}{q^{\prime}(h) T(h)}
 \begin{pmatrix}
  - \partial_{h} \Phi^2(\theta, h) \\
  \partial_{h} \Phi^1(\theta, h) \\
 \end{pmatrix} 
 =  \dfrac{1}{q^{\prime}(h) T(h)} (\partial_{h} \Phi)^{\perp}; \label{eq:JacobianLemmaTwo} \\
  \nabla \InverseMap_2 (\Phi(\theta,h)) &= 
 \dfrac{1}{q^{\prime}(h)} \nabla H (\Phi(\theta, h)). \label{eq:JacobianLemmaThree}
 \end{align}
\end{lemma}
\begin{proof}
First, note that from (1) it follows that
\begin{equation}\label{eq:DerivativeOfHOfHPhi}
 \dfrac{\partial}{\partial h} \Big[ H(\Phi(\theta, h)) \Big] = \nabla H (\Phi(\theta, h)) \cdot \partial_h \Phi(\theta, h)= q^{\prime}(h)
\end{equation}
 The Jacobian matrix of $\Phi$ is 
\[
 J_{\Phi}(\theta, h) = 
 \begin{pmatrix}
  - T(h) \partial_y H(\Phi(\theta, h)) & \partial_{h} \Phi^1(\theta, h) \\ 
  T(h) \partial_x H(\Phi(\theta, h)) & \partial_{h} \Phi^2(\theta, h) \\
 \end{pmatrix}
\]
The determinant of this matrix is given by
\[
 \det J_{\Phi}(\theta, h) = - T(h) \partial_h \Phi(\theta, h) \cdot \nabla H (\Phi(\theta, h)) \stackrel{\eqref{eq:DerivativeOfHOfHPhi}}{=} - q^{\prime}(h) T(h).
\]
This proves \eqref{eq:JacobianLemmaOne}.
Finally, we observe that
 \[
 J_{\InverseMap}(\Phi(\theta, h)) = \Big( J_{\Phi}(\theta, h) \Big)^{-1} = - \dfrac{1}{q^{\prime}(h) T(h)}
 \begin{pmatrix}
  \partial_{h} \Phi^2(\theta, h)  & - \partial_{h} \Phi^1(\theta, h) \\ 
  - T(h) \partial_x H(\Phi(\theta, h)) & - T(h) \partial_y H(\Phi(\theta, h)) \\
 \end{pmatrix}
\]
which implies \eqref{eq:JacobianLemmaTwo} and \eqref{eq:JacobianLemmaThree}
\end{proof}

For the purpose of stating the assumptions on the Hamiltonian, we also introduce the notion of streamline average.

\begin{definition}[Streamline average]
We define a streamline average denoted by $\langle \cdot \rangle$ as follows: Let $f \colon M \to \R^m$ be a function, then we define $\langle f \rangle \colon M \to \R^m$ in the following manner. For any $x \in M$, let $h \in I$ be such that $x \in \{ H = q(h) \}$ and then set
\[
 \langle f \rangle (x) = \int_{\T} (f \circ \Phi)(\theta, h) \, d \theta.
\]
\end{definition}

Note that, by definition $\langle f \rangle$ is constant along streamlines and hence independent on the choice of rescaled action-angle variables.

\subsection{Fourier series in angle variable} \label{sec:Fourier}
One of the main advantages of having an angle coordinate is the possibility to perform the Fourier transform along the streamlines. Indeed,  recall that the mapping $\Phi$ defined in the previous subsection provides a change of variables from $\T \times I$ to $M$. We denote with $\mathcal{F}$ the Fourier transform on $\T$ and, given $f \in L^2(M)$, $h\in I$, and $k\in \ZZ$, we introduce the following notation 
\[
 \widehat{f \circ \Phi}(k,h) \coloneqq \mathcal{F}[(f \circ \Phi)(\cdot, h)](k)=\int_{\TT}\e^{2 \pi ik \theta}(f \circ \Phi)(\theta, h) \dd \theta.
\]
With this notation, Parseval's identity reads as
\begin{equation}\label{eq:ParsevalsIdentity}
 \int_{\T} (f_1 \circ \Phi)(\theta, h) \overline{(f_2 \circ \Phi)(\theta, h)} \, d \theta = \sum_{k \in \Z} \widehat{(f_1 \circ \Phi)}(k, h) \overline{\widehat{(f_2 \circ \Phi)}(k, h)}
\end{equation}
for all functions $f_1, f_2 \in L^2(M)$ and a.e. $h \in I$.
Then, in account of \eqref{eq:DerivativeWRTTheta}, notice that 
\begin{equation}
(u\cdot \nabla f)\circ \Phi= (\nabla^\perp H\cdot \nabla f)\circ \Phi=\frac{1}{T(h)} \de_\theta (f\circ \Phi).
\end{equation}
Thus, recalling the definition of the frequency function (see Definition~\ref{def:PeriodAndFrequency}), we get
\begin{equation}
\mathcal{F}((u\cdot \nabla f)\circ \Phi)(k,h)= \frac{ik}{T(h)} \widehat{(f\circ \Phi)}(k,h)=ik\Omega(h) \widehat{(f\circ \Phi)}(k,h).
\end{equation}
In other words, the operator $u\cdot \nabla$ becomes  $\Omega \de_\theta$ in the new reference frame. The projection onto the $0$-th angular mode, defined as
 \begin{align}
 (P_0 f \circ \Phi)(\theta, h) = \int_\T (f \circ \Phi)(\theta, h) \dd \theta
\end{align}
is equivalent to the projection onto the kernel of $u\cdot\nabla $ and average along the streamlines. 
In terms of the Fourier series notation introduced above, we have 
\begin{align}
\label{def:P0Fourier}
P_0(\widehat{ f \circ \Phi})(k,h)=  \begin{cases}
  \widehat{f \circ \Phi}(0,h) &\text{if $k = 0$,} \\
  0 &\text{if $k \neq 0$,}
 \end{cases}
 \qquad 
 P_\perp =I-P_0.
\end{align}
Thanks to the Fourier series characterization, we see that $P_0$, and therefore $P_\perp$, commutes with $u\cdot \nabla $.

We also need to use more general Fourier multipliers, for which we use the following notation. We say that $\cW \colon L^2(M) \to L^2(M)$ is a Fourier multiplier with with Fourier symbol $w(k,h):\ZZ\times I\to \RR$ if
 \begin{equation}
 \label{def:W}
 \widehat{((\cW f) \circ \Phi)}(k, h) = w(k, h) \widehat{f \circ \Phi}(k, h).
 \end{equation}
For instance, $P_0$ is a Fourier multiplier with symbol $\mathbbm{1}_{k=0}(h)$.

\subsection{Hypotheses on the Hamiltonian}\label{sec:HypHamiltonians}
We are now ready to introduce the two classes of Hamiltonians for which we prove Theorems~\ref{th:main} and \ref{th:main2} and Corollary~\ref{cormain2}.
We start with the least restrictive class for which we can prove Theorem~\ref{th:main2} and Corollary~\ref{cormain2}.

\begin{definition}[The Classes \mylabel{HamiltonianClassA}{\normalfont{($\mathrm{A}$)}} and \mylabel{HamiltonianClassAm}{($\mathrm{A}^m$)}]\label{def:ClassHamiltoniansA}
A Hamiltonian $H \colon M \to \R$ is said to belong to the class ($\mathrm{A}$) if the following holds:
\begin{enumerate}
 \item[\mylabel{item:ClassACond1}{(A-1)}] (Regularity) $H \in C^3(\overline{M})$ 
 \item[\mylabel{item:ClassACond2}{(A-2)}] (Non-degenerate elliptic point) The Hamiltonian $H$ has one unique critical point which is non-degenerate and is elliptic, i.e. there exists a $x_0 \in M$ such that
 \[
  \nabla H(x_0) = 0 \quad \text{and} \quad \det D^2 H(x_0) > 0.
 \] 
 \item[\mylabel{item:ClassACond3}{(A-3)}] (Properties of the period) The function $\Omega \colon I \to \R^+$ belongs to $C^1(I)$ and is bounded by strictly positive constants from below and above and $\Omega^{\prime}$ is bounded from above. 
 \item[\mylabel{item:ClassACond4}{(A-4)}] (Weak uniformity) There exists a constant $C$ such that for all streamlines $\{ H = s \}$ such that $x_0 \not \in \{ H = s \}$ we have
 \[
  \dfrac{\sup_{x \in \{ H = s \}} |\nabla H(x)|}{\inf_{x \in \{ H = s \}} |\nabla H(x)|} \leq C.
 \] 
\end{enumerate}
An Hamiltonian $H$ belonging to ($\mathrm{A}$) is said to belong to the class ($\mathrm{A}^m$) where $m \geq 1$ is an integer if 
\begin{enumerate}
\item[\mylabel{item:ClassACondOnDerivative}{($\textup{A}^{\prime}_m$)}] $\Omega^{\prime}$ vanishes only at the elliptic point and the order of vanishing is $m$. In other words, there exists a constant $c > 0$ such that
 \[
  |\Omega^{\prime}(h)| \geq c |h|^{m-1} \quad \forall h \in I.
 \]
\end{enumerate} 
\end{definition}

For Hamiltonians belonging to this class, it is not hard to show the following result using the fact that the elliptic point is non-degenerate. The proof simply follows by Taylor expansion around the elliptic point.
\begin{lemma}
\label{lem:Hamiltonian}
 Let $H \colon M \to \R$ be an Hamiltonian of class \ref{HamiltonianClassA}. Then, there exists $R > 0$ and $\Lambda > 0$ such that for all $0 < r < R$
 \begin{align}
  \| \nabla H \|_{L^{\infty}(B_r(x_0))} &\leq \Lambda r; \\
  \{ x \in M : H(x) - H(x_0) < r^2 \} &\subseteq B_{\Lambda r} (0). 
 \end{align}
\end{lemma}

\begin{definition}\label{def:ClassHamiltoniansP}
Let $\eps \in (0,1)$. An Hamiltonian $H \colon M \to \R$ is said to belong to the class \mylabel{HamiltonianClassPEps}{$(\mathrm{P}_\eps)$} if \ref{item:ClassACond1}, \ref{item:ClassACond2}, \ref{item:ClassACond3} are fulfilled and the following holds:
\begin{enumerate}
 \item[\mylabel{item:ClassPCond4}{(P-4)}] (Uniformity) For any $x \in M$, we have
 $$ \left| |\nabla H|^2(x) - \left\langle |\nabla H|^2 \right\rangle(x) \right| \leq \eps |\nabla H|^2(x). $$ 
 \item[\mylabel{item:ClassPCond5}{(P-5)}] (Transversality) There exists an action angle change of variables such that
 \begin{equation}\label{eq:transversality}
  |\partial_{\theta} \Phi(\theta, h) \cdot \partial_{h} \Phi(\theta, h)| \leq \eps |\partial_{\theta} \Phi(\theta, h)| |\partial_{h} \Phi(\theta, h)|.
  \end{equation} 
\end{enumerate}
Moreover, an Hamiltonian $H$ belonging to \ref{HamiltonianClassPEps} is said to belong to \mylabel{HamiltonianClassPEps-m}{$(\mathrm{P}_\eps^m)$} where $m \geq 1$ is an integer if it satisfies \ref{item:ClassACondOnDerivative}.
\end{definition}

\begin{remark}
We see that any radial flow with only one critical elliptic non-degenerate point belongs to $(\mathrm{P}_0)$.
\end{remark}

\begin{remark}
We note that \ref{item:ClassPCond4} implies \ref{item:ClassACond4}. Hence, any Hamiltonian belonging to \ref{HamiltonianClassPEps} belongs to \ref{HamiltonianClassA}.
\end{remark}

\subsection{Fourier multipliers and derivatives}
In the sequel, we need a precise control of the interplay between derivatives and the Fourier multipliers defined in \eqref{def:W}. Indeed, the operators $\nabla$ and $\cW$ do not commute for a general Hamiltonian (as opposed to the shear or radial flow case). However, for Hamiltonians in the classes introduced above we are able to prove the following. 
  \begin{lemma} \label{lemma:operator-estimates}
 Let $H$ be an Hamiltonian belonging to \ref{HamiltonianClassA}. Let $\cW$ be a Fourier multiplier with symbol $w$ as defined in \eqref{def:W}. Let $C,N>0$ be two fixed constants. Then, there exists a constant $C^{\prime}>0$ depending only on $C,N$ and $H$ such that the following hold true:
  \begin{itemize}
  \item[a)] If $\norm{w}_{L^\infty(\mathbb{Z}\times I)} \leq C$ and $\norm{\de_h w}_{L^\infty(\mathbb{Z}\times I)}\leq N$, then
  \begin{equation}
 \label{bd:nablaWf}
 \| \nabla 
 (\cW f) \|_{L^2(M) } \leq C^{\prime} \| \nabla f \|_{L^2(M)} + C^{\prime} N \| f \|_{L^2(M) } \quad \forall f \in L^2(M). 
 \end{equation}
   \item[b)] If $\norm{w}_{L^\infty(\mathbb{Z}\times I)} \leq C$ and $\norm{\de_h w(k)}_{L^\infty(I)} \leq C |k|$, then
  \begin{equation}
 \label{bd:nablaWf2}
 \| \nabla 
 (\cW f) \|_{L^2(M) } \leq C^{\prime} \| \nabla f \|_{L^2(M)}  \quad \forall f \in L^2(M).
 \end{equation}

\end{itemize}
 \end{lemma} 

  \begin{proof}[Proof of Lemma \ref{lemma:operator-estimates}]
By the definition \eqref{def:W}, we know that the Fourier multiplier $\cW$ is adapted to the coordinates defined through the map $\Phi$. One possibility to understand the operator $\nabla \cW$ is to pass from the standard $(\de_{x_1},\de_{x_2})$ derivatives to derivatives in the directions $(\de_{\theta}\Phi/|\de_\theta \Phi|,\de_h\Phi/|\de_h \Phi|)$. 
To do so, we notice that
\begin{equation}\label{eq:DetermiantEstimateFromBelow}
\left|\operatorname{det} \left (\frac{\de_\theta\Phi}{|\de_\theta\Phi|}, \frac{\de_h \Phi}{|\de_h \Phi|} \right )\right| = \dfrac{q^{\prime}(h)T(h)}{|\de_h\Phi| |\nabla H\circ \Phi|} = \frac{q^{\prime}(h)}{|\de_h\Phi| |\nabla H\circ \Phi|} > c
\end{equation}
for some constant $c > 0$ (which depends on the Hamiltonian $H$).
Note that we have crucially used the assumption that $H$ belongs to ($\mathrm{A}$).  
Thus, we can write 
\begin{equation}
\label{eq:graddehphi}
\nabla (\cW f)\circ\Phi=\left (\frac{\de_\theta\Phi}{|\de_\theta\Phi|}, \frac{\de_h \Phi}{|\de_h \Phi|} \right )(\nabla (\cW f)\circ\Phi).
\end{equation}
Observing that 
\begin{equation}
\de_h\Phi \cdot (\nabla(\cW f)\circ \Phi)=\de_h(\cW f\circ \Phi), \qquad \de_\theta\Phi \cdot (\nabla(\cW f)\circ \Phi)=\de_\theta(\cW f\circ \Phi),
\end{equation}
and combining it with \eqref{eq:DetermiantEstimateFromBelow} we arrive at 
\begin{align}
\norm{\nabla (\cW f)}^2_{L^2(M)}&=\int_{I}\int_\TT |\nabla(\cW f)\circ \Phi|^2(\theta,h)q^{\prime}(h)T(h)\dd \theta \dd h\\
&\leq c^{\prime} \underbrace{\int_{I}\int_\TT \left(\frac{1}{|\de_h\Phi|^2} |\de_h (\cW f\circ \Phi)|^2\right)(\theta,h) q^{\prime}(h)T(h)\dd \theta \dd h}_{= \mathcal{I}_h}\\
&\qquad + c^{\prime} \underbrace{\int_{I}\int_\TT \left(\frac{1}{|\de_\theta\Phi|^2} |\de_\theta(\cW f\circ \Phi)|^2 \right)(\theta,h) q^{\prime}(h) T(h)\dd \theta \dd h}_{= \mathcal{I}_\theta}.
\end{align}
for some constant $c^{\prime}$.
We start by estimating $\cI_\theta$. We note that
\begin{equation}
|\de_\theta \Phi(\theta, h)| = T(h) |\nabla H (\Phi(\theta, h) ) | \geq c^{\prime \prime} T(h) q^{\prime}(h)
\end{equation}
for some constant $c^{\prime \prime}$. 
Therefore, using the bound above, Plancherel's Theorem and the fact that $\norm{w}_{L^\infty}\lesssim 1$, we obtain
\begin{align}
\cI_\theta&\leq \dfrac{1}{(c^{\prime \prime})^2} \int_{I}\int_\TT \left(\frac{1}{(T(h)q^{\prime}(h))^2} |\de_\theta(\cW f\circ \Phi)|^2 \right)(\theta,h)q^{\prime}(h)T(h)\dd \theta \dd h\\
&\leq \dfrac{1}{(c^{\prime \prime})^2} \int_I \sum_{k\neq 0} \frac{|k|^2}{(T(h)q^{\prime}(h))^2}|(w \widehat{f\circ \Phi})|^2(k,h) q^{\prime}(h)T(h)\dd h\\
&\leq \dfrac{C^2}{(c^{\prime \prime})^2} \int_I \sum_{k\neq 0} \frac{1}{(T(h)q^{\prime}(h))^2}|\widehat{(\de_\theta f\circ \Phi)}|^2(k,h) q^{\prime}(h)T(h)\dd h\\
&\leq \dfrac{C^2}{(c^{\prime \prime})^2} \int_I \int_\TT\left(\frac{|\de_\theta\Phi(\theta, h)|}{T(h)q^{\prime}(h)}\right)^2|\nabla f\circ \Phi|^2(\theta,h) q^{\prime}(h)T(h)\dd \theta \dd h.
\end{align}
However, since 
\begin{equation}
\sup_{\theta \in \T, h \in I} \frac{|\de_\theta\Phi(\theta, h)|}{T(h)q^{\prime}(h)} = \sup_{\theta \in \T, h \in I} \frac{|\nabla H(\Phi(\theta, h))|}{q^{\prime}(h)} < \infty,
\end{equation}
we conclude that 
\begin{equation}
\label{bd:Itheta}
\cI_\theta \leq C^{\prime}_1 \norm{\nabla f}^2.
\end{equation}
for some $C^{\prime}_1$.
Turning our attention to $\cI_h$, thanks to \eqref{eq:DerivativeOfHOfHPhi}, we have 
\begin{equation}
q^{\prime}(h) \leq |\de_h \Phi(\theta, h)||\nabla H(\Phi(\theta, h))| \quad \text{and hence} \quad |\de_h\Phi(\theta, h)| \geq \frac{q^{\prime}(h)}{|\nabla H(\Phi(\theta, h))|} \geq c^{\prime \prime}.
\end{equation}
Thus, using this bound and proceeding similarly to what we have done for $\cI_\theta$, we get
\begin{align}
\label{bd:Ih1}
\cI_h &\leq \dfrac{1}{(c^{\prime \prime})^2} \int_I \sum_{k \in \Z} |\de_h(w \widehat{f\circ \Phi})|^2(k,h) q^{\prime}(h)T(h)\dd h\\
\label{bd:Ih2}
&\leq \dfrac{1}{(c^{\prime \prime})^2} \underbrace{\int_I \sum_{k \in \Z} |\de_h w|^2|\widehat{f\circ \Phi}|^2(k,h) q^{\prime}(h)T(h)\dd h}_{=\cI_{h,1}} + \dfrac{1}{(c^{\prime \prime})^2} \underbrace{\int_I \sum_{k \in \Z} |\de_h(\widehat{f\circ \Phi})|^2(k,h) q^{\prime}(h)T(h)\dd h}_{=\cI_{h,2}}.
\end{align}
For $\cI_{h,2}$, using Plancherel's Theorem and exploiting the fact that $\partial_h \Phi \in L^{\infty}(\T \times I)$, we have 
\begin{align}
\label{bd:Ih20}
\cI_{h,2} = \int_I \int_\TT|\de_h\Phi|^2|(\nabla f)\circ \Phi)|^2(\theta,h) q^{\prime}(h)T(h)\dd \theta \dd h \leq \| \partial_h \Phi \|_{L^{\infty}(\T \times I)}^2 \norm{\nabla f}^2_{L^2(M)}.
\end{align}
Finally, to complete the proof, we control the term $\cI_{h,1}$. Here we split the argument depending on whether we are considering the case (a) or (b). In the case (a), we obtain
\begin{align}
\label{bd:Ih20}
\cI_{h,1} \leq N^2 \int_I \int_\TT| \widehat{f\circ \Phi}|^2(\theta,h) q^{\prime}(h)T(h)\dd \theta \dd h\lesssim N^2 \norm{f}^2_{L^2(M)}.
\end{align}
In the case (b), we use the fact that $|\de_\theta \Phi| \leq \| \nabla H \|_{L^{\infty}(M)} \cdot \sup_{h \in I} T(h) $ to obtain
\begin{align}
\label{bd:Ih10}
\cI_{h,1}&\leq C^2 \int_I \sum_{k\neq 0} |k|^2 |\widehat{f\circ \Phi}|^2(k,h) q^{\prime}(h)T(h)\dd h\\
&\leq C^2 \int_I \int_{\mathbb{T}}|\de_\theta(f\circ \Phi)|^2(\theta,h) q^{\prime}(h)T(h)\dd \theta\dd h\\
&\leq C^2 \int_I \int_{\mathbb{T}}|\de_\theta\Phi|^2|(\nabla f)\circ \phi)|^2(\theta,h) q^{\prime}(h)T(h)\dd \theta\dd h \leq C^{\prime}_{2} \norm{\nabla f}^2_{L^2(M)},
\end{align}
for some $C^{\prime}_{2}$.
This concludes the proof.
\end{proof}

\subsection{Poincar\'e-type inequalities on thin sets} 
One of the key ingredients  for the subsequent pseudospectral estimates, are Poincaré-type inequalities on thin neighbourhoods of streamlines. 
 These are needed to quantitatively estimate where one can gain something from the transport operator, as we explain more precisely in Section \ref{sec:pseudo}. The prototypical bound we aim at proving is the following, which is what one needs in the case of a radial Hamiltonian:
\begin{equation}
\int_{\{R_1\leq |x|\leq R_2\}}|f|^2(x)\dd x\leq 2(R_2-R_1)\norm{f}_{L^2}\norm{\nabla f}.
\end{equation}
A proof of this bound can be found in \cite{Gallay:2021aa}, and notice that one can effectively gain smallness parameters whenever $|R_2-R_1|\ll 1$.
 However, to the best of our knowledge, there are no such results that can be applied as a black box in our context. Therefore, in this section we prove in detail the estimates we need, some of which might be of independent interest.

With the notation introduced in Section \ref{sec:toolbox},  we denote a rectangle in $\TT\times I$ as
\begin{equation}
\label{def:Box}
 Q_{\eta, \gamma}(\overline{\theta}, \overline{h}) \coloneqq \{ (\theta, h) \in \T \times I : |\theta - \overline{\theta}| < \eta, |h - \overline{h}| < \gamma \} \subset \T \times I.
\end{equation}
In account of the periodicity of the domain, when $\eta=1$ there is no need to specify $\overline{\theta}$ and we simply denote 
\begin{equation}
\label{def:Box1}
 Q_{ \gamma}(\overline{h})= Q_{1, \gamma}(0, \overline{h})
 \end{equation}
Thanks to the map $\Phi$ in \eqref{def:Phi}, we define the \textit{stream-adapted} rectangles as
\begin{equation}
\label{def:cBox}
\mathcal{Q}_{\eta,\gamma}(\bar{\theta},\bar{h}):=\Phi(Q_{\eta, \gamma}(\overline{\theta}, \overline{h})). 
\end{equation}
Notice that $\mathcal{Q}_{\cdot}(\cdot)$ is a collection of `curly' rectangles around the streamlines $\{H= q(\bar{h}) \}$. 
In all the definitions above, we refer to $(\overline{\theta}, \overline{h})$ as the \emph{center} of the rectangle or stream-adapted rectangle. Whenever it is clear from the context, we will choose not to write $(\overline{\theta}, \overline{h})$ to lighten the notation.
\begin{remark}
When $\eta=1$, $\mathcal{Q}_{\gamma}(\bar{h})$ is a tubular neighbourood of the streamlines $\{H= q(\bar{h})\}$, where we do not need to specify $\bar{\theta}$ since we are considering the whole $\TT$ in the angular direction. The width of the set $\mathcal{Q}_{\gamma}(\bar{h})$ is  related to $\gamma$ through the map $\Phi$. For instance, one might have Hamiltonians where the width ranges from $\gamma$ to $\gamma^2$, as it happens in the cellular flow $H=\sin(x)\sin(y)$ for streamlines close to the axis.
\end{remark}
The following result is the main result of this section.
\begin{lemma}
\label{lem:Poincare}
Let $(\overline{\theta}, \overline{h}) \in \T \times I$, $\eta, \gamma >0$ and  $g \in H^1(M)$. Then, the following inequalities hold true for stream-adapted rectangles centered in $(\overline{\theta}, \overline{h})$:
\begin{enumerate}
\item For any integer $K$ such that $1 < K \ll \gamma^{-1}$, one has
\begin{align}
 \begin{split}
   \| f \|_{L^2(\mathcal{Q}_{\eta, \gamma})}^2 &\leq \frac{8}{K} \| f \|_{L^2(\mathcal{Q}_{\eta, K \gamma})}^2 + 4 \gamma \| f \|_{L^2(\mathcal{Q}_{\eta, K \gamma})}\| \nabla f \|_{L^2(\mathcal{Q}_{\eta, K \gamma})} \left\| \partial_h \Phi \right\|_{L^{\infty}(Q_{\eta, K \gamma})} \\
  &\qquad + 2 \gamma \| f \|_{L^2(\mathcal{Q}_{\eta, K \gamma})}^2 \left\| [\log(q^{\prime} T)]^{\prime} \right\|_{L^{\infty}(Q_{\eta, K \gamma})}. \label{eq:SecondPoincare}
  \end{split}
\end{align}

\item For any integer $N$ such that $1 < N \ll \eta^{-1}$, one has
\begin{align}
\begin{split}
\| f \|_{L^2(\mathcal{Q}_{\eta, \gamma})}^2 &\leq \frac{8}{N} \| f \|^2_{L^2(\mathcal{Q}_{N \eta, \gamma})} + 4 \eta \| f \|_{L^2(\mathcal{Q}_{N \eta, \gamma})} \| \nabla f \|_{L^2(\mathcal{Q}_{N \eta, \gamma})} \| (\nabla^{\perp} H \circ \Phi) T \|_{L^{\infty}(Q_{N \eta, \gamma})} . \label{eq:FirstPoincare}
\end{split}
\end{align}

\item If $g$ satisfies the additional property that $P_0 f = 0$, and  $\eta = 1$, one has
\begin{equation}
   \| f \|_{L^2(\mathcal{Q}_{\gamma}( \overline{h}))} \leq \| \nabla f \|_{L^2(\mathcal{Q}_{ \gamma}(\overline{h}))} \| T (\nabla H \circ \Phi) \|_{L^{\infty}(Q_{\gamma}( \overline{h}))}. \label{eq:ThirdPoincare}
\end{equation}
\end{enumerate}
\end{lemma}
The inequality \eqref{eq:FirstPoincare} quantitatively exploits the smallness in the streamwise (or angular) direction, as can be seen by the factor $\eta$ multiplying the second term in \eqref{eq:FirstPoincare}. Similarly, the inequality \eqref{eq:SecondPoincare} uses the thinness in the \textit{streakwise} (or action) direction. The first term in both inequalities containing the factor $1/N$ is related to the need of removing the average over a given direction in the domain, and in fact in \eqref{eq:ThirdPoincare} is not present. This will still be sufficient for our purposes because we can use $1/N$ as a smallness parameter to absorb that error with the $L^2$ norm in the full domain.

\begin{proof}
 Using the rescaled action-angle variables introduced in Section~\ref{sec:toolbox},  we have
 \[
  \| f \|_{L^2(\mathcal{Q}_{\eta, \gamma}(\overline{\theta}, \overline{h}))}^2 = \int_{Q_{\eta,\gamma}(\overline{\theta}, \overline{h})} |f \circ \Phi|^2(\theta, h) q^{\prime}(h) T(h) \, d\theta dh.
 \]
 \\
 \textbf{Proof of \eqref{eq:SecondPoincare}:}
  Let $\chi_r^K \in C^{\infty}(I)$ be a radial cut-off function such that
  \begin{itemize}
   \item $\chi_r^K(h) = 1$ if $|h-\overline{h}| < \gamma$;
   \item $\chi_r^K(h) = 0$ if $|h-\overline{h}| \geq K \gamma$;
   \item $\| \chi_r^K \|_{L^{\infty}(I)} \leq 1$ and $\| (\chi_r^K)^{\prime} \|_{L^{\infty}(I)} \leq \dfrac{2}{K \gamma}$.
  \end{itemize}
  We define an auxiliary function $F_r^K \colon \T \times I \to \R$ as
  \[
   F_r^K(\theta, h) = \chi_r^K(h) (f \circ \Phi)(\theta, h) \sqrt{q^{\prime}(h) T(h)}.
  \]
  Then
  \[
   \| f \|_{L^2(\mathcal{Q}_{\eta, \gamma}(\overline{\theta}, \overline{h}))}^2 = \int_{{Q}_{\eta, \gamma}(\overline{\theta}, \overline{h})} |F_r^K(\theta, h)|^2 \, d \theta dh = \int_{\overline{h} - \gamma}^{\overline{h} + \gamma} \int_{\overline{\theta} - \eta}^{\overline{\theta} + \eta} |F_r^K(\theta, h)|^2 \, d \theta dh.
  \]
  We observe that for all $h \in [\overline{h} - K \gamma, \overline{h} + K \gamma]$
  \[
   |F_r^K(\theta, h)|^2 = \int_{\overline{h} - K \gamma}^h \partial_h \left( |F_r^K(\theta, \tilde{h})|^2 \right) \, d \tilde{h} = 2 \int_{\overline{h} - K \gamma}^h F_r^K(\theta, \tilde{h})  \partial_h F_r^K(\theta, \tilde{h}) \, d \tilde{h}.  \]
  Integrating this over $[\overline{\theta} - \eta, \overline{\theta} + \eta]$ yields
  \begin{align}
  \begin{split}\label{eq:PoincareBigFIntegratedInTheta}
   \int_{\overline{\theta} - \eta}^{\overline{\theta} + \eta} |F_r^K(\theta, h)|^2 \, d \theta &= 2 \int_{\overline{\theta} - \eta}^{\overline{\theta} + \eta} \int_{\overline{h} - K \gamma}^h F_r^K(\theta, \tilde{h})  \partial_h F_r^K(\theta, \tilde{h}) \, d \tilde{h} \\
   &\leq 2 \| F_r^K \|_{L^2(Q_{\eta, K \gamma}(\overline{\theta}, \overline{h}))} \| \partial_h F_r^K \|_{L^2(Q_{\eta, K \gamma}(\overline{\theta}, \overline{h}))}
   \end{split}
  \end{align}
  where we used the Cauchy-Schwarz inequality together with the fact that $[\overline{\theta} - \eta, \overline{\theta} + \eta] \times [\overline{h} - K \gamma, h] \subseteq Q_{\eta, K \gamma}(\overline{\theta}, \overline{h})$ for all $h \in [\overline{h} - K \gamma, \overline{h} + K \gamma]$. Then, we notice that
  \begin{equation}\label{eq:PoincareBoundOnBigF}
   \| F_r^K \|_{L^2(Q_{\eta, K \gamma}(\overline{\theta}, \overline{h}))} \leq \| f \|_{L^2(\mathcal{Q}_{\eta, K \gamma}(\overline{\theta}, \overline{h}))}.
  \end{equation}
  Moreover,
  \begin{align*}
   \partial_h F_r^K(\theta, h) &= (\chi_r^K)^{\prime}(h) (f \circ \Phi)(\theta, h) \sqrt{q^{\prime}(h) T(h)} + \chi_r^K(h) (\nabla f \circ \Phi)(\theta, h) \partial_h \Phi(\theta, h) \sqrt{q^{\prime}(h) T(h)} \\
   &\qquad + \chi_r^K(h) (f \circ \Phi)(\theta, h) \partial_h(\sqrt{q^{\prime}(h) T(h)})
  \end{align*}
  Therefore,
  \begin{align*}
   \| \partial_h F_r^K \|_{L^2(Q_{\eta, K \gamma}(\overline{\theta}, \overline{h}))} &\leq \| (\chi_r^K)^{\prime} \|_{L^{\infty}(Q_{\eta, K \gamma}(\overline{\theta}, \overline{h}))} \| (f \circ \Phi) \sqrt{q^{\prime} T} \|_{L^{2}(Q_{\eta, K \gamma}(\overline{\theta}, \overline{h}))} \\
   &\quad + \| \chi_r^K \|_{L^{\infty}(Q_{\eta, K \gamma}(\overline{\theta}, \overline{h}))} \| (\nabla f \circ \Phi) \sqrt{q^{\prime} T} \|_{L^{2}(Q_{\eta, K \gamma}(\overline{\theta}, \overline{h}))}  \| \partial_h \Phi \|_{L^{\infty}(Q_{\eta, K \gamma}(\overline{\theta}, \overline{h}))} \\
   &\hspace{-2cm} + \dfrac{1}{2} \| \chi_r^K \|_{L^{\infty}(Q_{\eta, K \gamma}(\overline{\theta}, \overline{h}))} \| (f \circ \Phi) \sqrt{q^{\prime} T} \|_{L^{2}(Q_{\eta, K \gamma}(\overline{\theta}, \overline{h}))}  \left\| \dfrac{q^{\prime \prime} T + q^{\prime} T^{\prime}}{q^{\prime} T} \right\|_{L^{\infty}(Q_{\eta, K \gamma}(\overline{\theta}, \overline{h}))} \\
   &\leq \dfrac{2}{K \gamma} \| f \|_{L^{2}(\mathcal{Q}_{\eta, K \gamma}(\overline{\theta}, \overline{h}))} + \| \nabla f \|_{L^{2}(\mathcal{Q}_{\eta, K \gamma}(\overline{\theta}, \overline{h}))} \| \partial_h \Phi \|_{L^{\infty}(Q_{\eta, K \gamma}(\overline{\theta}, \overline{h}))} \\
   &\quad + \dfrac{1}{2} \| f \|_{L^{2}(\mathcal{Q}_{\eta, K \gamma}(\overline{\theta}, \overline{h}))}  \left\| [\log(q^{\prime} T)]^{\prime} \right\|_{L^{\infty}(Q_{\eta, K \gamma}(\overline{\theta}, \overline{h}))}
  \end{align*}
  Combining this last observation with \eqref{eq:PoincareBigFIntegratedInTheta} and \eqref{eq:PoincareBoundOnBigF}, we deduce that 
  \begin{align*}
   \int_{\overline{\theta} - \eta}^{\overline{\theta} + \eta} |F_r^K(\theta, h)|^2 \, d \theta &\leq \dfrac{4}{K \gamma} \| f \|_{L^{2}(\mathcal{Q}_{\eta, K \gamma}(\overline{\theta}, \overline{h}))}^2 + 2 \| f \|_{L^{2}(\mathcal{Q}_{\eta, K \gamma}(\overline{\theta}, \overline{h}))} \| \nabla f \|_{L^{2}(\mathcal{Q}_{\eta, K \gamma}(\overline{\theta}, \overline{h}))} \| \partial_h \Phi \|_{L^{\infty}(Q_{\eta, K \gamma}(\overline{\theta}, \overline{h}))} \\
   &\quad + \| f \|_{L^{2}(\mathcal{Q}_{\eta, K \gamma}(\overline{\theta}, \overline{h}))}^2  \left\| [\log(q^{\prime} T)]^{\prime} \right\|_{L^{\infty}(Q_{\eta, K \gamma}(\overline{\theta}, \overline{h}))}.
  \end{align*}
  Therefore,
  \begin{align*}
  \| f \|_{L^2(\mathcal{Q}_{\eta, \gamma}(\overline{\theta}, \overline{h}))}^2 &= \int_{\overline{h} - \gamma}^{\overline{h} + \gamma} \int_{\overline{\theta} - \eta}^{\overline{\theta} + \eta} |F_r^K(\theta, h)|^2 \, d \theta dh \\
  &\leq \dfrac{8}{K} \| f \|_{L^{2}(\mathcal{Q}_{\eta, K \gamma}(\overline{\theta}, \overline{h}))}^2 + 4 \gamma \| f \|_{L^{2}(\mathcal{Q}_{\eta, K \gamma}(\overline{\theta}, \overline{h}))} \| \nabla f \|_{L^{2}(\mathcal{Q}_{\eta, K \gamma}(\overline{\theta}, \overline{h}))} \| \partial_h \Phi \|_{L^{\infty}(Q_{\eta, K \gamma}(\overline{\theta}, \overline{h}))} \\
   &\quad + 2 \gamma \| f \|_{L^{2}(\mathcal{Q}_{\eta, K \gamma}(\overline{\theta}, \overline{h}))}^2  \left\| [\log(q^{\prime} T)]^{\prime} \right\|_{L^{\infty}(Q_{\eta, K \gamma}(\overline{\theta}, \overline{h}))}.
  \end{align*}
  Thanks to \eqref{eq:DerivativeWRTc}, the proof is finished.
  
  \medskip
 \textbf{Proof of \eqref{eq:FirstPoincare}:} The proof follows the exact same strategy as the previous point. The only difference is that one defines a cutoff in $\T$ instead of a radial one. 

\medskip

 \textbf{Proof of \eqref{eq:ThirdPoincare}: }
 Since $P_0 f = 0$, we have
 \[
  \int_{\T} (f \circ \Phi)(\theta, h) \, d \theta = 0 \quad \text{for a.e. $h \in I$.}
 \]
 Thus for a.e. $h \in I$ there exists $\theta_0^h \in \T$ such that $(f \circ \Phi)(\theta_0^h, h) = 0$. Hence, we have
 \begin{align*}
  |f \circ \Phi(\theta, h)|^2 &= \int_{\theta_0^h}^{\theta} \partial_{\theta}(|f \circ \Phi|^2)(\tilde{\theta}, h) \, d \tilde{\theta} \\
  &= 2\int_{\theta_0^h}^{\theta} \Big( (f \circ \Phi) \partial_{\theta} \Phi \cdot (\nabla f \circ \Phi) \Big) (\tilde{\theta}, h) \, d \tilde{\theta}. 
 \end{align*}
 From this identity, H\"older's inequality and \eqref{eq:DerivativeWRTTheta}, we deduce that
  \begin{align*}
  \int_{|h - \overline{h}| < \gamma} |(f \circ \Phi)(\theta, h)|^2 q^{\prime}(h) T(h) \, dh \leq \| f \|_{L^2(\mathcal{Q}_{1, \gamma}(\overline{\theta}, \overline{h}))} \| \nabla f \|_{L^2(\mathcal{Q}_{\gamma}(\overline{h}))} \| T (\nabla H \circ \Phi) \|_{L^{\infty}(Q_{  \gamma}( \overline{h}))}. 
 \end{align*}
 Integrating in the angular variable over $\TT$, we finally prove \eqref{eq:ThirdPoincare}.
\end{proof}

\section{The model problem as the leading order approximation for $\rho_\perp$}
\label{sec:proofmain}
In this section, we aim at showing that our model problem \eqref{eq:model}, namely, 
\begin{equation*}
\begin{cases}
\de_tg +u\cdot \nabla g=\nu P_\perp\Delta P_\perp g, \qquad x \in M, \, t> 0,\\
g|_{t=0}=\rho_{\perp}^{in} \qquad g|_{\partial M}=0, \text{ or }\partial_n g|_{\partial M}=0,
\end{cases} 
\end{equation*}
can be thought as the leading order approximation describing the dynamics of $\rho_\perp$, which solves \eqref{eq:PperpP0}.  In this section we assume the result in Theorem \ref{th:main2}, namely we need the bound \eqref{bd:gmain} on $g$. With this bound at hand, we first prove first prove Theorem \ref{th:main} and then Corollary \ref{cormain2} and Proposition \ref{prop:conditional}. 

\subsection{Proof of Theorem \ref{th:main}}
First of all, we observe that since $\nabla \cdot u=0$ and we have Dirichlet or Neumann boundary conditions on $\rho$, a standard energy estimate for \eqref{eq:advdiff} give us
\begin{equation}
\frac12\frac{\dd }{\dd t}\norm{\rho}^2_{L^2}=-\nu \norm{\nabla \rho}^2_{L^2}\leq -\frac{\nu}{C_P} \norm{\rho}^2,
\end{equation}
where we used the Poincaré inequality (which holds true since we have $\int_M\rho^{in}=0$ and the domain $M$ is bounded). With the notation in Section \ref{sec:Fourier}, from the inequality above we get
\begin{equation}
\label{bd:heatdecay}
\norm{\rho(t)}_{L^2}^2=\norm{\rho_0(t)}^2_{L^2}+\norm{\rho_\perp(t)}^2_{L^2}\leq \e^{-2(\nu/C_P ) t}\norm{\rho^{in}}_{L^2}^2.
\end{equation}
When $\rho_{0}^{in}\neq 0$, the bound above readily implies the decay on time-scales $O(\nu^{-1})$ for $\rho_0$ stated in \eqref{bd:rho0main}. Notice that when $\rho_0^{in}=0$ we have to gain an extra factor of $\eps$, that does not follow by the standard energy estimate above.

To improve the bounds on $\rho_\perp$ (and $\rho_0$ when $\rho_0^{in}=0$), we use that $H$ belongs to the class \ref{HamiltonianClassPEps-m} to perform an asymptotic expansion in $\eps$. We recall that $\eps$ is the smallness parameter quantifying how far we are from being ``radial", more precisely we need the hypotheses \ref{item:ClassPCond4} and \ref{item:ClassPCond5} in Definition \ref{def:ClassHamiltoniansP}. We want to construct solutions to \eqref{eq:PperpP0} as  
\begin{equation}
\label{eq:expansion}
\rho_\perp=:g+\sum_{n=1}^{+\infty}\eps^{n}\rho^{(n)}_{\perp}=:g+g_{\mathrm{corr}}, \qquad \rho_0=\eta+\sum_{n=1}^{+\infty}\eps^{n}\rho^{(n)}_{0},
\end{equation}
where $g$ solves \eqref{eq:model}, $\eta$ is the solution to \eqref{eq:model00}. We define the approximations as follows\footnote{without specifying the domain and boundary conditions, which are the same ones of $g$}: denote $\rho_\perp^{(0)}:=g$ and $\rho_0^{(0)}=\eta$. 
\begin{itemize}
\item  For $n\geq 1$, we define $\rho^{(n)}_\perp$ to be the solution of 
\begin{equation}
\label{pb:rhoperpn}
\begin{cases}
\displaystyle \de_t\rho^{(n)}_\perp +u\cdot\nabla \rho^{(n)}_\perp=\nu P_\perp\Delta \rho^{(n)}_\perp+\frac{\nu}{\eps}P_\perp\Delta  \rho^{(n-1)}_0,\\
\rho^{(n)}_\perp|_{t=0}=0.
\end{cases} 
\end{equation}
\item For $n\geq 1$ we define $ \rho^{(n)}_0$ to be the solution of
\begin{equation}
\label{pb:rho0n}
 \begin{cases}
\displaystyle \de_t\rho^{(n)}_0 =\nu P_0\Delta  \rho^{(n)}_0+\frac{\nu}{\eps}P_0\Delta  \rho^{(n-1)}_\perp,\\
\rho^{(n)}_0|_{t=0}=0,
\end{cases}
\end{equation}
\end{itemize}
With the definitions above, if the series in \eqref{eq:expansion} is well defined, it is not difficult to see that indeed $\rho_\perp,\rho_0$ solve \eqref{eq:PperpP0}. 
In particular, we aim at proving that  $\rho^{(n)}_{\iota}$, with $\iota\in\{\perp,0\}$, remain sufficiently small in a suitable functional space. This would imply the desired convergence of the series and smallness of $g_{\mathrm{corr}}$. Hence, the decay on the time-scale $O(\lambda_\nu)$ in \eqref{bd:rhoperpmain} is coming from $g$ and Theorem \ref{th:main2}, whereas the other error is related to the bounds on $\rho^{(n)}_\iota$. Moreover, we also need to show that $g_{\mathrm{corr}}$ decays at least on $O(\nu^{-1})$ time-scales, but this will be a simple consequence of the Poincaré inequality.

To obtain the desired bounds on $\rho_\iota^{(n)}$ we heavily rely on the assumptions in Definition \ref{def:ClassHamiltoniansP}. We are able to prove the following.
\begin{proposition}
\label{prop:expansion}
Let $\rho_{\perp}^{(n)},\rho_0^{(n)}$ be the solutions of \eqref{pb:rhoperpn} and \eqref{pb:rho0n} respectively. Assume that the Hamiltonian $H \colon \overline{M} \to \R$ belongs to the class \ref{HamiltonianClassPEps-m} with parameter $\eps \in (0,1)$ and $m \geq 1$. Then, for all $n\geq 0$ the following bounds holds true
\begin{equation}
\label{bd:proprhon}
\|\rho_{\iota}^{(n)}(t)\|_{L^2}+\left(\frac{\nu}{2}\int_{0}^t\|\nabla\rho_{\iota}^{(n)}(\tau)\|_{L^2}^2\dd \tau\right)^{\frac12}\leq \e^{-c_0\nu t}\left( \sqrt{\dfrac{ 8}{1 - \eps}} \right)^{n} \norm{\rho^{in}}_{L^2}, \quad \text{for } \iota\in\{\perp,0\},
\end{equation}
where $c_0=1/(2C_P)$ with $C_P$ being the Poincaré constant of the domain $M$.

\noindent The series in \eqref{eq:expansion} are well defined in the $L^\infty_t L^2_x\cap L^\infty_t H^1_x$ sense whenever $\eps\sqrt{8/(1-\eps)}<1$.

\noindent Moreover, if $\|\rho_0^{in}\|_{L^2}=0$ then $\|\rho_0^{(2n)}\|_{L^2}=\|\rho_\perp^{(2n+1)}\|_{L^2}=0$ for all $n\geq0$.
\end{proposition}
Before proving this proposition, let us show that Theorem \ref{th:main} is a direct consequence of Proposition \ref{prop:expansion} and Theorem \ref{th:main2}.
\begin{proof}[Proof of Theorem \ref{th:main}]
By \eqref{eq:expansion}, we know that $\rho_\perp-g_{\mathrm{corr}}=g$ where $g$ solves \eqref{eq:model}. Therefore, since the class \ref{HamiltonianClassPEps-m} is contained in the class $(\mathrm{A}^m)$, we can apply Theorem \ref{th:main2} and readily deduce the bound \eqref{bd:rhoperpmain}. 

To prove \eqref{bd:gcorr}, we directly apply the bound \eqref{bd:proprhon} with $n= 1,2\dots$ and $\iota=\perp$. Similarly, when $\rho_0^{in}\neq 0$ we apply the bound \eqref{bd:proprhon} with $n=0,1,\dots$ and prove \eqref{bd:rho0main}. In the case $\rho_0^{in}=0$, we know that $\rho_0^{(0)}=\rho_\perp^{(1)}=0$ and therefore the series for $g_{\mathrm{corr}}$ starts at $n=2$ whereas the series for $\rho_0$ starts at $n=1$, thus giving us the extra factor $\eps$ claimed in Theorem \ref{th:main}.
\end{proof}
It thus remain to prove Proposition \ref{prop:expansion}.
\begin{proof}[Proof of Proposition \ref{prop:expansion}] 
To control $g_{\mathrm{corr}}$ and $\rho_0$ we rely on standard energy estimates for the equations \eqref{eq:model}, \eqref{pb:rhoperpn} and \eqref{pb:rho0n} where a quantitative bound on the commutator $[P_0,\nabla]$ is needed to absorb the factor $\eps^{-1}$ in the forcing terms in \eqref{pb:rhoperpn} and \eqref{pb:rho0n}. 
To propagate the exponential decay on the time scale $O(\nu^{-1})$, we define 
\begin{equation}
\label{def:tilderho}
\tilde{\rho}^{(n)}_\iota= \e^{c_0\nu t} \rho^{(n)}_\iota, \qquad c_0=\frac{1}{2C_p}, \qquad \iota\in \{0,\perp\}, \qquad n=0,1,\dots.
\end{equation}
Notice that $\tilde{\rho}^{(n)}_\perp, \tilde{\rho}^{(n)}_0,$ satisfy
\begin{align}
&\de_t \tilde{\rho}^{(n)}_\perp+u\cdot \nabla \tilde{\rho}^{(n)}_\perp=\nu P_\perp \Delta \tilde{\rho}^{(n)}_\perp+\frac{\nu}{\eps}P_\perp\Delta \tilde{\rho}_{0}^{(n-1)}+c_0\nu \tilde{\rho}^{(n)}_\perp,\\
&\de_t \tilde{\rho}^{(n)}_0=\nu P_0 \Delta \tilde{\rho}^{(n)}_0+\frac{\nu}{\eps}P_0\Delta \tilde{\rho}_{\perp}^{(n-1)}+c_0\nu \tilde{\rho}^{(n)}_0,
\end{align}
with the same initial and boundary conditions of $\rho^{(n)}_\iota$ for $\iota \in \{0,\perp\}$.

We first deal with the case where $\|\rho_0^{in}\|_{L^2}\neq0$. For $n = 0$, recalling that $\e^{c_0\nu t} g=\tilde{\rho}^{(0)}_\perp$, standard energy estimates yield
\begin{align}
&\frac{\dd }{\dd t}\|{\tilde{\rho}^{(0)}_\perp}\|_{L^2}^2+2\nu \|{\nabla \tilde{\rho}^{(0)}_\perp}\|^2_{L^2}=2c_0\nu \|{\tilde{\rho}^{(0)}_\perp}\|_{L^2}^2\leq \nu\|{\nabla\tilde{\rho}^{(0)}_\perp}\|_{L^2}^2,\\
&\frac{\dd }{\dd t}\|{\tilde{\rho}^{(0)}_0}\|_{L^2}^2+2\nu \|{\nabla \tilde{\rho}^{(0)}_0}\|^2_{L^2}=2c_0\nu \|{\tilde{\rho}^{(0)}_0 }\|_{L^2}^2\leq \nu\|{\nabla\tilde{\rho}^{(0)}_0}\|_{L^2}^2.
\end{align}
where in the last line we used  the Poincaré inequality and $c_0=1/(2C_P)$. Integrating in time, we obtain
\begin{align}
\label{eq:eng}
&\|{\tilde{\rho}^{(0)}_\perp(t)}\|_{L^2}^2+\nu \int_0^t \|{\nabla\tilde{\rho}^{(0)}_\perp(\tau)}\|^2_{L^2}\dd \tau\leq \|{\rho_\perp^{in}}\|_{L^2}^2,\\
\label{eq:enrho0}
&\|{\tilde{\rho}_0^{(0)}(t)}\|_{L^2}^2+\nu \int_0^t \|{\nabla\tilde{\rho}_0^{(0)}(\tau)}\|^2_{L^2}\dd \tau\leq \|{\rho_0^{in}}\|_{L^2}^2.
\end{align}
For $n\geq 1$, since $\tilde{\rho}^{(n)}_\iota|_{t=0}=0$ for $\iota \in \{0,\perp\}$, performing analogous energy estimates we have
\begin{align}
\label{eq:enrhoperpn}
\|{\tilde{\rho}_\perp^{(n)}(t)}\|_{L^2}^2+\nu \int_0^t \|{\nabla \tilde{\rho}_\perp^{(n)}(\tau)}\|^2_{L^2}\dd \tau\leq-2\frac{\nu}{\eps}\int_0^t\jap{\nabla \tilde{\rho}_0^{(n-1)},\nabla \tilde{\rho}_\perp^{(n)}}_{L^2}(\tau)\dd \tau.
\end{align}
and
\begin{align}
\label{eq:enrho0n}
\|{\tilde{\rho}_0^{(n)}(t)}\|_{L^2}^2+\nu \int_0^t \|{\nabla \tilde{\rho}_0^{(n)}(\tau)}\|^2_{L^2}\dd \tau\leq-2\frac{\nu}{\eps}\int_0^t\jap{\nabla \tilde{\rho}_\perp^{(n-1)},\nabla \tilde{\rho}_0^{(n)}}_{L^2}(\tau)\dd \tau.
\end{align}
Notice that, in both energy inequalities \eqref{eq:enrhoperpn} and \eqref{eq:enrho0n}, the commutator between $P_0$ and $\nabla$ is involved. Indeed, if $[P_0,\nabla]=0$ the right-hand side in \eqref{eq:enrhoperpn} and \eqref{eq:enrho0n} vanish. To extract information from the commutator, we need to control  terms of the form
\begin{equation}
\mathcal{J}[f_0,b_\perp]:=\jap{\nabla f_0,\nabla b_\perp}_{L^2(M)} = \int_I \int_{\T} ((\nabla f_0 \circ \Phi) (\nabla b_{\perp} \circ \Phi))(\theta,h) q^{\prime}(h)T(h) \, d \theta d h,
\end{equation}
for two given functions $f,b\in H^1(M)$ and where $(\theta,h)$ and $\Phi$ are the variables and the map defining the change of coordinates in Section \ref{sec:rescaled-action-angle}.  We denote 
\begin{equation}
F_0( h):=(f_0\circ \Phi)( h), \qquad B_\perp(\theta,h):=(b_\perp \circ \Phi)(\theta,h).
\end{equation}
Note that $F_0$ does not depend on $\theta$ since $f_0$ is a streamline-average. We recall that $\Psi:=\Phi^{-1}$ and in particular $\Psi_1(x)=h(x)$ and $\Psi_2(x)=\theta(x)$, but we do not use $h,\theta$ to avoid confusion.
With the notation from Subsection~\ref{sec:rescaled-action-angle} , we have the following identities
\begin{align}
&f_0(x)=(F_0 \circ \Psi)(x) = F_0(\Psi_2(x)); \\
&b_\perp(x)=(B_\perp \circ \Psi)(x); \\
\label{eq:nablaf0}&\nabla f_0= (\de_h F_0 \circ \Psi ) \nabla \Psi_2; \\
\label{eq:nablabperp}& \nabla b_\perp= (\de_\theta B_\perp \circ \Psi)\nabla \Psi_1 + (\de_h B_\perp \circ \Psi )\nabla \Psi_2.
\end{align}
Therefore, we have
\begin{align*}
\mathcal{J}[f_0,b_\perp] &= \int_I \int_{\T}\left (\left(\left( (\nabla \Psi_1 \cdot \nabla \Psi_2) \circ \Phi \right) \partial_{\theta} B_{\perp} \partial_h F_0+ \left( |\nabla \Psi_2 \circ \Phi|^2 \right)\partial_{h} B_{\perp} \partial_h F_0\right) q' T\right)(\theta,h) \, d \theta dh 
\end{align*}
Since 
\[
 \int_{\T} B_{\perp}(\theta, h) \, d \theta = 0 \text{ for a.e. $h \in I$,}
\]
we observe that for any function $\Gamma_0 \colon I \to \R$
\begin{align}
 \int_I \int_{\T} \Gamma_0(h)(\partial_{h} B_{\perp} \partial_h F_0)(\theta, h) \, d \theta dh = 0. \label{eq:StreamlineAverageIntegratedZero2}
\end{align}
Thus, to obtain bounds for $\mathcal{J}[f_0,b_\perp]$, we can remove the streamline average of the coefficients related to the change of coordinates if needed. In particular, thanks the assumptions \ref{item:ClassPCond4}-\ref{item:ClassPCond5} in Definition \ref{def:ClassHamiltoniansP}, it will be enough to control
\begin{equation}
\mathcal{C}_1=|\nabla \Psi_2 \circ \Phi|^2-\langle|\nabla \Psi_2 \circ \Phi|^2\rangle, \qquad \mathcal{C}_2=(\nabla \Psi_1 \cdot \nabla \Psi_2) \circ \Phi.
\end{equation}
In fact, as we mention in Remark \ref{rem:deBperp}, one could slightly improve the bounds for the term involving $\de_\theta B_{\perp}$.
In account of \eqref{eq:JacobianLemmaThree}, by using \ref{item:ClassPCond4} we have
\begin{align*}
 |\mathcal{C}_1|&=  \big| |\nabla \Psi_2  \circ \Phi|^2 - \int_{\T}  |\nabla \Psi_2  \circ \Phi|^2\, d \theta \big|= \dfrac{\big||\nabla H (\Phi(\theta,h))|^2 - \int_{\T} |\nabla H (\Phi(\theta,h))|^2 \, d \theta\big|}{q^{\prime}(h)^2} \\
 &\leq \dfrac{\eps |\nabla H (\Phi(\theta, h))|^2}{q^{\prime}(h)^2} = \eps |\nabla \Psi_2(\Phi(\theta, h))|^2.
\end{align*}
Exploiting now \eqref{eq:JacobianLemmaTwo} and \ref{item:ClassPCond5}, we obtain
\begin{equation}
|\mathcal{C}_2|\leq \frac{1}{(q' T)^2}|\de_h\Phi\cdot \de_\theta \Phi|\leq  \eps|\nabla \Psi_1(\Phi(\theta, h))||\nabla \Psi_2(\Phi(\theta, h))|.
\end{equation}
Using the last two bounds above, we get 
\begin{align}
\label{bd:I0perp1}
|\mathcal{J}[f_0,b_\perp]|\leq \eps  \int_I \int_{\T} \left(\big(|(\nabla \Psi_1 \circ \Phi)\partial_{\theta} B_{\perp}|+|(\nabla \Psi_2 \circ \Phi)\partial_{h} B_{\perp}|\big)| (\nabla \Psi_2\circ\Phi)  \partial_h F_0| q' T\right)(\theta,h) \, d \theta dh. 
\end{align}
Moreover, in view of \ref{item:ClassPCond5}, notice that 
\begin{align*}
&|(\nabla \Psi_1 \circ \Phi)\partial_{\theta} B_{\perp}+(\nabla \Psi_2 \circ \Phi)\partial_{h} B_{\perp}|^2\\
&\geq |(\nabla \Psi_1 \circ \Phi)\partial_{\theta} B_{\perp}|^2+|(\nabla \Psi_2 \circ \Phi)\partial_{h} B_{\perp}|^2-2|(\nabla \Psi_1 \circ \Phi)\cdot(\nabla \Psi_2 \circ \Phi)||\partial_{\theta} B_{\perp}\partial_{h} B_{\perp}|\\
&\geq (1-\eps)(|(\nabla \Psi_1 \circ \Phi)\partial_{\theta} B_{\perp}|^2+|(\nabla \Psi_2 \circ \Phi)\partial_{h} B_{\perp}|^2).
\end{align*}
Consequently, by using the simple inequality $|a|+|b|\leq \sqrt{2(a^2+b^2)}$ and the identity \eqref{eq:nablabperp}, we deduce that
\begin{align}
|(\nabla \Psi_1 \circ \Phi)\partial_{\theta} B_{\perp}|+|(\nabla \Psi_2 \circ \Phi)\partial_{h} B_{\perp}|\leq\sqrt{\frac{2}{1-\eps}}|\nabla b_\perp \circ \Phi|.
\end{align}
Using the inequality above in the bound \eqref{bd:I0perp1}, and appealing to  \eqref{eq:nablaf0}, we have
\begin{align}
\label{bd:Jfobperp}
|\mathcal{J}[f_0,b_\perp]|&\leq \eps\sqrt{\frac{2}{1-\eps}}\int_I \int_{\T}(|\nabla b_\perp \circ \Phi||\nabla f_0\circ \Phi|q'T)(\theta,h)\dd \theta\dd h\\
&\leq \eps\sqrt{\frac{2}{1-\eps}} \norm{\nabla b_\perp}_{L^2}\norm{\nabla f_0}_{L^2},
\end{align}
where in the last inequality we also changed variables to go back to standard Cartesian coordinates.
We now apply the bound \eqref{bd:Jfobperp} with $f_0=\tilde{\rho}_0^{(j)}$ and $b_\perp=\tilde{\rho}_\perp^{(j)}$ with $j\in \{n,n-1\}$ to the energy estimates \eqref{eq:enrhoperpn} and \eqref{eq:enrho0n}.   For \eqref{eq:enrhoperpn}, this yields
\begin{align}
&\|{\tilde{\rho}_\perp^{(n)}(t)}\|_{L^2}^2+\nu \int_0^t \|{\nabla \tilde{\rho}_\perp^{(n)}(\tau)}\|^2_{L^2}\dd \tau\leq2\nu \sqrt{\frac{2}{1-\eps}}  \int_0^t\|{\nabla \tilde{\rho}_0^{(n-1)}(\tau)}\|_{L^2}\|{\nabla \tilde{\rho}_\perp^{(n)}(\tau)}\|_{L^2}\dd \tau.
\end{align}
Applying  Young's inequality, we get
\begin{align}
&\|{\tilde{\rho}_\perp^{(n)}(t)}\|_{L^2}^2+\frac{\nu}{2} \int_0^t \|{\nabla \tilde{\rho}_\perp^{(n)}(\tau)}\|^2_{L^2}\dd \tau \leq  \frac{8}{1-\eps} \left( \frac{\nu}{2} \int_0^t\|{\nabla \tilde{\rho}_0^{(n-1)}(\tau)}\|_{L^2}^2\dd\tau\right).
\end{align}
Proceeding analogously in \eqref{eq:enrho0n}, we find that for  $\tilde{\rho}_0^{(n)}$ we have exactly the same inequality above by exchanging each occurence of $\rho_\perp$ with $\rho_0$ and viceversa.

In the case $n=1$, thanks to \eqref{eq:eng} and \eqref{eq:enrho0}, these estimates reduce to
\begin{align}
\label{bd:rhoperp1}&\|{\tilde{\rho}_\perp^{(1)}(t)}\|_{L^2}^2+\frac{\nu}{2} \int_0^t \|{\nabla \tilde{\rho}_\perp^{(1)}(\tau)}\|^2_{L^2}\dd \tau\leq  \frac{4}{1-\eps} \|{\rho_0^{in}}\|^2_{L^2}\leq \frac{8}{1-\eps}\left(\frac12\norm{\rho^{in}}^2_{L^2}\right),\\
&\|{\tilde{\rho}_0^{(1)}(t)}\|_{L^2}^2+\frac{\nu}{2} \int_0^t \|{\nabla \tilde{\rho}_0^{(1)}(\tau)}\|^2_{L^2}\dd \tau\leq \frac{4}{1-\eps}\|{\rho_\perp^{in}}\|^2_{L^2}\leq\frac{8}{1-\eps}\left(\frac12\|{\rho^{in}}\|^2_{L^2}\right).
\end{align}
Therefore, proceeding by induction we deduce that 
\begin{align}
&\|{\tilde{\rho}_\iota^{(n)}(t)}\|_{L^2}^2+\frac{\nu}{2} \int_0^t \|{\nabla \tilde{\rho}_\iota^{(n)}(\tau)}\|^2_{L^2}\dd \tau \leq  \left( \dfrac{8}{1 - \eps} \right)^{n} \frac12\norm{\rho^{in}}^2_{L^2}, \qquad \iota\in \{0,\perp\},
\end{align}
thus proving the desired bound \eqref{bd:proprhon} after recalling the definition in \eqref{def:tilderho}.

We now turn our attention to the case $\|\rho_0^{in}\|_{L^2}=0$. Thanks to \eqref{eq:enrho0} we deduce that $\|\tilde{\rho}_0^{(0)}\|_{L^2}=0$. Moreover, thanks to the first inequality in \eqref{bd:rhoperp1}, we also have $\|\tilde{\rho}_{\perp}^{(1)}\|_{L^2}=0$. Iterating this reasoning, it is not hard to conclude that $\|\tilde{\rho}_0^{(2n)}\|_{L^2}=\|\tilde{\rho}_\perp^{(2n+1)}\|_{L^2}=0$ for all $n\geq0$ thus concluding the proof of the proposition.
\end{proof}
\begin{remark}\label{rem:deBperp}
In the definition of $\mathcal{J}[f_0,b_\perp]$, notice that we have the term 
\begin{align}
&\int_I\int_{\mathbb{T}}(((\nabla \Psi_1 \cdot \nabla \Psi_2)\circ \Phi)\de_\theta B_\perp \de_h F_0)(q'T)(\theta,h)\dd \theta \dd h\\
&=-\int_I\int_{\mathbb{T}}\left((\de_\theta((\nabla \Psi_1 \cdot \nabla \Psi_2)\circ \Phi))) (B_\perp \de_h F_0)q'T\right)(\theta,h)\dd \theta \dd h,
\end{align}
where in the last identity we integrated by parts in $\theta$ and used that $\de_\theta F_0=0$. With this identity, it would be enough to assume 
\begin{equation}
|\de_\theta((\nabla \Psi_1 \cdot \nabla \Psi_2)\circ \Phi)) |\leq c |\nabla \Psi_2\circ \Phi|,
\end{equation}
where $c$ need not to be small in principle. Indeed, we would not need to use a full gradient on $\rho^{(j)}_\perp$ when doing the energy estimates, and this would imply getting worst universal constants in the final estimates. However, when $c>\eps$ we need to assume $\eps$ to be small compared to $c$, and therefore it does not seem to be a significant improvement.
\end{remark}

\subsection{Proof of Corollary \ref{cormain2}}
The proof of this corollary is based on a standard comparison argument as done, for instance, in \cite{CZDE20}. Indeed,  let $T_\nu=C_* \nu^{-\frac{m}{m+2}}$ with $C_*$ to be specified later. Then, let $\tilde{c}_*\in (0,1)$ be a fixed constant. If
\begin{equation}
2\nu \int_0^{T_\nu}\norm{\nabla \rho(\tau)}^2_{L^2}\dd \tau\geq \tilde{c}_*\|{\rho^{in}_\perp}\|^2_{L^2}, 
\end{equation}
by the standard energy estimate on \eqref{eq:advdiff}, since $\rho^{in}=\rho^{in}_\perp$, we get
\begin{align}
\norm{\rho(T_\nu)}^2_{L^2}=\|{\rho^{in}_\perp}\|^2_{L^2}-2\nu \int_0^{T_\nu}\norm{\nabla \rho(\tau)}^2_{L^2}\dd \tau\leq (1-\tilde{c}_*)\|{\rho^{in}_\perp}\|^2_{L^2}.
\end{align}
By the properties of $P_0$ , we have 
\begin{equation}
\|\rho_\perp(T_\nu)\|^2_{L^2}\leq\norm{\rho(T_\nu)}^2_{L^2}\leq (1-\tilde{c}_*)\|{\rho^{in}_\perp}\|^2_{L^2},
\end{equation}
thus proving the desired result in this case. On the other hand, consider now the case 
\begin{equation}
\label{bd:hypcstar}
2\nu \int_0^{T_\nu}\norm{\nabla \rho(\tau)}^2_{L^2}\dd \tau< \tilde{c}_*\|{\rho^{in}_\perp}\|^2_{L^2}. 
\end{equation}
Let $g$ be the solution to \eqref{eq:model}. Then, by the linearity of the equations considered, we notice that $\rho_\perp-g$ solves 
\begin{equation}
\de_t(\rho_\perp-g)+u\cdot \nabla(\rho_\perp-g)=\nu P_\perp\Delta (\rho_\perp-g) +\nu P_\perp\Delta  P_0\rho,
\end{equation}
with zero initial data. Therefore, multiplying by $\rho_\perp-g$ and integrating in space, we deduce that 
\begin{equation}
\frac12 \frac{\dd }{\dd t}\|\rho_\perp-g\|^2_{L^2}+\nu \|\nabla (\rho_\perp-g)\|^2_{L^2}=-\nu \jap{\nabla (P_0\rho), \nabla (\rho_\perp-g)}_{L^2}.
\end{equation}
Applying the Cauchy-Schwarz and Young's inequality, and integrating in time, we get
\begin{equation}
\label{bd:energy}
\|(\rho_\perp-g)(T_\nu)\|^2_{L^2}+\nu \int_0^{T_\nu}\|\nabla (\rho_\perp-g)(\tau)\|^2_{L^2}\dd \tau\leq \nu \int_0^{T_\nu}\|\nabla (P_0 \rho)(\tau)\|^2_{L^2}\dd \tau.
\end{equation}
Applying Lemma \ref{lemma:operator-estimates} and in particular \eqref{bd:nablaWf} with $w=\mathbbm{1}_{k=0}$, we know that for an Hamiltonian in class  \ref{HamiltonianClassAm} we have
\begin{align}
\|\nabla (P_0 \rho)(\tau)\|^2_{L^2}\leq C_H \|\nabla \rho(\tau)\|_{L^2}^2
\end{align}
for a suitable constant $C_H>0$ depending only the Hamiltonian $H$. Combining the bound above with the energy estimate \eqref{bd:energy} and the condition \eqref{bd:hypcstar}, we see that
\begin{align}
\|(\rho_\perp-g)(T_\nu)\|^2_{L^2}\leq \tilde{c}_* C_H\|\rho^{in}\|^2_{L^2}.
\end{align}
Finally, thanks to the bounds on $g$ in \eqref{bd:gmain} we see that 
\begin{align}
\|\rho_\perp(T_\nu)\|^2_{L^2}\leq2( \|g(T_\nu)\|^2_{L^2}+\|(\rho_\perp-g)(T_\nu)\|^2_{L^2})\leq 2(\e^{-\delta_* C_*+\pi/2}+\tilde{c}_* C_H)\|\rho^{in}\|^2_{L^2}.
\end{align}
Choosing $C_*$ sufficiently large and $\tilde{c}_*$ sufficiently small, we get
\begin{align}
\|\rho_\perp(T_\nu)\|^2_{L^2}\leq \frac12\|\rho^{in}\|^2_{L^2}.
\end{align}
Therefore, we conclude the proof of the corollary by choosing $c_*$ such that $1-c_*=\max\{ 1/2, 1-\tilde{c}_*\}$. \qed 

\begin{remark}
\label{rem:rho0}
We stress that we cannot directly iterate the proof of Corollary \ref{cormain2} to deduce the exponential decay of $\rho_{\perp}$. Indeed, the major obstacle is that $\rho_0$ is not conserved and therefore $\|\rho(T_\nu)\|_{L^2}\neq \|\rho_\perp(T_\nu)\|_{L^2}$ in general. Notice that having $\rho_0$ constant is a necessary condition in the proof of the result in \cite{CZDE20}. In particular, one cannot apply the result in \cite{CZDE20} to deduce enhanced dissipation from  mixing estimates on the advection problem ($\nu=0$) when $[P_0,\Delta]\neq 0$ (for example the cellular flow studied in detail in \cite{Brue:2022aa}). To iterate the bounds proved above, it would be enough to propagate quantitative smallness of $\rho_0$ (with respect to $\nu$) on suitable sub-diffusive time-scales, thus proving smallness of $\rho_\perp$ for longer time-scales.
\end{remark}

\subsection{Proof of Proposition~\ref{prop:conditional}}
The proof is based on iteration of the previous corollary, using crucially the assumption on the uniform estimate of $\| (\rho_0 - \eta) (t) \|_{L^2}\leq \nu^\alpha \|\rho^{in}\|_{L^2}$.
Let $T_{\nu}=C_* \nu^{-\frac{m}{m+2}}$ with $C_*$ sufficiently large and $\tilde{c}_*\in (0,1/2)$ sufficiently small as in the proof of Corollary \ref{cormain2}. In the iteration,  we assume $(n+1) T_{\nu} \leq \nu^{-1}$, where the case $n=0$ is the previous corollary, and divide the proof into two distinct cases. 

\medskip 
\noindent
\textbf{Case 1:} Assume

\begin{equation}\label{eq:CondCase1}
2\nu \int_{n T_{\nu}}^{(n+1) T_{\nu}}\norm{\nabla \rho(\tau)}^2_{L^2}\dd \tau\geq \tilde{c}_* \| \rho_\perp (n T_{\nu}  )\|^2_{L^2}.
\end{equation}
By properties of the projection $P_0$ and the fact that $P_\perp \eta=0$, we get
\begin{align}
 \| (\rho  &  - \eta) ((n+1) T_{\nu})\|^2_{L^2}  =\| \rho_\perp ((n+1) T_{\nu}) \|^2_{L^2} 
 + \norm{ (\rho_0 - \eta )((n+1)T_{\nu})}_{L^2}^2
 \\
\label{id:pop37} & =  \| \rho ((n+1 )T_{\nu}) \|_{L^2}^2 - \norm{ \rho_0 ((n+1)T_{\nu})}_{L^2}^2 
  +  \norm{ (\rho_0 - \eta) ((n+1)T_{\nu})}_{L^2}^2.
\end{align}
Using the assumption, we can bound the third term above by $\nu^{2\alpha} \| \rho^{in}\|_{L^2}^2$. By standard energy estimates on \eqref{eq:advdiff}, thanks to \eqref{eq:CondCase1}, the first term can be bounded by
\begin{align*}
 \| \rho (n T_{\nu}) \|_{L^2}^2 - \tilde{c}_*\| \rho_\perp  (n T_{\nu}) \|_{L^2}^2 &= (1- \tilde{c}_*) \| \rho_\perp  (n T_{\nu}) \|_{L^2}^2 + \| \rho_0 (n T_{\nu}) \|_{L^2}^2 \\
 &= (1- \tilde{c}_*) \| (\rho- \rho_0) (n T_{\nu}) \|_{L^2}^2 + \| \rho_0 (n T_{\nu}) \|_{L^2}^2 \\
 &\leq  (1- \tilde{c}_*) \| (\rho - \eta) (n T_{\nu}) \|_{L^2}^2 +\| \eta (n T_{\nu}) \|_{L^2}^2   
 \\
 & \quad 
 + 2  \| (\eta - \rho_0) (n T_{\nu})  \|_{L^2}^2 
 \\
 & \quad + 4 \| (\eta - \rho_0)(n T_{\nu}) \|_{L^2} ( \| \eta (n T_\nu) \|_{L^2} + \| (\rho - \eta) ( n T_\nu) \|_{L^2} )
 \end{align*}
where in the last inequality we sum and subtract $\eta$ in both terms and perform trivial bounds by expanding the squares. Having that both $\rho$ and $\eta$ are decreasing and thus bounded by the intial data, in view of the assumption on $\rho_0-\eta$ we get 
\begin{align*}
 \| \rho ((n+1 )T_{\nu}) \|_{L^2}^2 & \leq  (1- \tilde{c}_*) \| (\rho- \eta )(n T_{\nu}) \|_{L^2}^2 +\| \eta (n T_{\nu}) \|_{L^2}^2   + 10 \nu^\alpha \| \rho^{in} \|_{L^2}^2.
\end{align*}
Analogously, for the term involving $\norm{ \rho_0 ((n+1)T_{\nu})}_{L^2}^2 $ we get
\begin{align*}
\norm{ \rho_0 ((n+1)T_{\nu})}_{L^2}^2 &  \geq \norm{ \eta ((n+1) T_\nu)}_{L^2}^2 - 2 \| (\rho_0 - \eta)((n+1)T_{\nu}) \|_{L^2} \| \eta ((n+1)T_{\nu}) \|_{L^2}
\\
& \geq \norm{ \eta ((n+1) T_\nu)}_{L^2}^2 - 2 \nu^\alpha \| \rho^{in} \|_{L^2}^2.
\end{align*} 
Thus, all in all, we deduce that
\begin{align*}
\| (\rho  & - \eta) ((n+1) T_{\nu})\|^2_{L^2} \\
 & \leq (1 - \tilde{c}_* ) \| (\rho  - \eta) (n T_{\nu}) \|_{L^2}^2 +  \norm{  \eta ( nT_{\nu})}_{L^2}^2  - \norm{ \eta ((n+1)T_{\nu})}_{L^2}^2  + 10 \nu^\alpha \| \rho^{in}\|_{L^2}^2.
\end{align*}
Then, since $\tau \mapsto \| \nabla \eta(\tau) \|_{L^2}$ is decreasing (this can be proved by testing \eqref{eq:model00} by $P_0 \Delta P_0 \eta$ and using that $P_0 \eta = \eta$),   we notice that
\begin{align*}
\| \eta ((n+1) T_{\nu}) -  \eta ( n T_{\nu}) \|_{L^2}^2 & \leq \nu \int_{nT_{\nu}}^{(n+1) T_{\nu}} \| \nabla  \eta (n T_{\nu})  \|_{L^2}^2 ds 
\\
& \leq  C_* \nu^{\frac{2}{m+2}} \| \nabla \rho^{in} \|_{L^2}^2 \leq C  C_* \nu^{ \frac{2}{m+2}} \| \rho^{in} \|_{L^2}^2  \,
\end{align*}
Therefore, from the previous bound we conclude that  there exists a constant $C_0$ such that 
\begin{align} \label{eq:bound:iterate}
\| (\rho    - \eta) ((n+1) T_{\nu})\|^2_{L^2}  \leq  (1 - \tilde{c}_* ) \| (\rho  - \eta) (n T_{\nu}) \|_{L^2}^2  +( C_0 \nu^{\frac{2}{m+2}} + 10 \nu^\alpha) \| \rho^{in} \|_{L^2}^2 \,.
\end{align} \\
\textbf{Case 2: } Assume
$$2\nu \int_{n T_{\nu}}^{(n+1) T_{\nu}}\norm{\nabla \rho(\tau)}^2_{L^2}\dd \tau < \tilde{c}_* \| \rho_\perp (n T_{\nu}  )\|^2_{L^2}.$$
We get
 \begin{align*}
  \|(\rho_\perp-g)((n+1)T_{\nu})\|^2_{L^2} & \leq \nu \int_{n T_{\nu}}^{(n+1)T_{\nu}}\|\nabla (P_0\rho)(\tau)\|^2_{L^2}\dd \tau
 \\
 &  \leq C_H \nu \int_{n T_{\nu}}^{(n+1)T_{\nu}} \|\nabla  \rho(\tau)\|^2_{L^2}\dd \tau
 \\
 &  \leq C_H \tilde{c}_* \| \rho_\perp (n T_{\nu}) \|_{L^2}^2
\end{align*}
where where $g$ is the solution to the model problem \eqref{eq:model} initialised at time $n T_{\nu}$ with initial data $\rho_{\perp}(n T_{\nu})$.
From this, we conclude that
 \begin{align*}
\|\rho_\perp((n+1)T_{\nu})\|^2_{L^2}& \leq2( \|g((n+1)T_{\nu})\|^2_{L^2}+\|(\rho_\perp-g)((n+1)T_{\nu})\|^2_{L^2})
\\
& \leq 2(\e^{-\delta_* C_*+\pi/2}+\tilde{c}_* C_H)\|(\rho  - \rho_0) (n T_{\nu})  \|^2_{L^2}.
\end{align*}
and
 by choosing $C_*$ large enough  and $c_*$ sufficiently small (exactly as the ones in the proof of Corollary \ref{cormain2}), we have 
 $$\| (\rho  - \rho_0 )((n +1)T_{\nu}) \|_{L^2}^2 \leq \frac{1}{2} \|(\rho - \rho_0) (n T_{\nu})  \|^2_{L^2}\,. $$
Therefore, using the assumption, summing and subtracting $\eta$ we get 
$$ \| (\rho  -  \eta) ((n +1)T_{\nu})  \|_{L^2}^2 \leq \frac{1}{2} \| (\rho   -  \eta) (n  T_{\nu})  \|_{L^2}^2 + 6 \nu^\alpha \| \rho^{in} \|_{L^2}^2  \,.$$

\medskip 
\noindent
Finally, in both cases, we are able to prove the bound \eqref{eq:bound:iterate}. 
Iterating this bound we get 
\begin{align}
\label{bd:iterativefinal}
\| (\rho    - \eta) ((n+1) T_{\nu})\|^2_{L^2} & \leq (1- \tilde{c}_*)^{n+1} \| \rho^{in}_{\perp} \|_{L^2}^2 + ( C_0 \nu^{\frac{2}{m+2}} + 10 \nu^\alpha) \| \rho^{in} \|_{L^2}^2 \sum_{j=0}^{n} (1- \tilde{c}_*)^{j}.
\end{align}
To conclude the bound, we need to control $(\rho-\eta)(t)$ for $t\in (nT_\nu,(n+1)T_\nu)$, where it is not hard to see that it can be controlled by $\|(\rho-\eta)(nT_\nu)\|^2$ up to an error of size $\max\{\nu^{\frac{2}{m+2}},\nu^\alpha\}\|\rho^{in}\|^2$. To see this, we can repeat computations similar to the Case 1. The starting identity \eqref{id:pop37} is replaced by 
\begin{align}
\notag \| (\rho    - \eta) (t)\|^2_{L^2}  =\| \rho (t) \|_{L^2}^2 - \norm{ \rho_0 (t)}_{L^2}^2 
  +  \norm{ (\rho_0 - \eta) (t)}_{L^2}^2.
\end{align}
We can then bound the first term by $\|\rho(nT_\nu)\|_{L^2}^2$ and repeat the computations in Case 1 with $\tilde{c}_*=0$ to get
\begin{align*}
 \| \rho (t) \|_{L^2}^2 & \leq   \| (\rho- \eta )(n T_{\nu}) \|_{L^2}^2 +\| \eta (n T_{\nu}) \|_{L^2}^2   + 10 \nu^\alpha \| \rho^{in} \|_{L^2}^2.
\end{align*}
For the remaining terms we argue exactly as in Case 1 and, using  \eqref{bd:iterativefinal}, we finally see that there exists $C_1>0$ such that 
$$ \| (\rho  - \eta) (t )\|^2_{L^2} \leq  (1- \tilde{c}_*)^{n} \| \rho^{in} \|_{L^2}^2 + C_1 (  \nu^{\frac{2}{m+2}} +  \nu^\alpha ) \| \rho^{in} \|_{L^2}^2$$
 for any $t\in (nT_{\nu},(n+1)T_{\nu})$. Therefore, recalling that $T_{\nu} = C_* \nu^{- \frac{m}{m+2}} = C_* \lambda_\nu^{-1}$ there exists 
$\delta_* $ such that 
$$ \| (\rho  - \eta) (t )\|^2_{L^2} \lesssim \left ( \exp (- \delta_* \lambda_\nu t)  +  \nu^{\frac{2}{m+2}} + \nu^\alpha   \right ) \| \rho^{in} \|_{L^2}^2 $$
for any $t \leq \nu^{-1}$.
\qed

\section{Pseudo spectral bounds on the model problem}
\label{sec:pseudo}
In this section, we prove Theorem \ref{th:main2}, which quantifies the enhanced dissipation effect for the model problem \eqref{eq:model}. The proof of the bound \eqref{bd:gmain} is in fact a direct consequence of the Gearhart-Pr\"uss type theorem proved by  Wei in \cite{WeiDiffusion19} (see also  \cite{Helffer:2010aa,HelfferImproving21,Gallagher09}), which we state here for completness.
\begin{theorem}[\cite{WeiDiffusion19}]
\label{th:pseudoWei}
Let $\cA:D(\cA)\subset X\to X$ be a maximally accretive operator\footnote{An operator $\cA$ is maximally accretive if $\|(a I+\cA)f\|_X\geq a \|f\|_X$ for all $a\geq 0$ and $a_0I+\cA$ is surjective for some $a_0>0$.} acting on a Hilbert space $X$. Then, for all $f\in D(\cA)$, 
\begin{equation}
\label{bd:Wei}
\norm{\e^{-t\cA} f}_X\leq \e^{-t \Psi(\cA)+\frac{\pi}{2}}\norm{f}_X,
\end{equation}
where  $\Psi(\cA)$ is the pseudo spectral abscissa defined as
\begin{equation}
\label{def:pseudo}
\Psi(\cA)=\operatorname{inf}\{\|(\cA-i\lambda) f\|_X \, :\, f\in D(\cA), \, \lambda \in \mathbb{R}, \, \|f\|_X=1 \}.
\end{equation}
\end{theorem}
Therefore, the proof of Theorem \ref{th:main2} is reduced to an estimate of the pseudo spectral abscissa of the operator 
\begin{equation}
\label{def:Lperp}
\cL_{\perp}=u\cdot \nabla -\nu P_\perp \Delta P_\perp,
\end{equation}
when $X=L^2_\perp(M)$. Depending on the boundary conditions we have, the domain of this operator will be 
\begin{equation}
D(\cL_\perp)=\{f \in H^2(M) \, ;\, \de_n f|_{\partial M} =0\} \qquad \text{ or } \qquad D(\cL_\perp)=\{f \in H^2(M) \, ;\, f|_{\partial M} =0\}.
\end{equation}
First of all, we need the following
\begin{lemma}
\label{lem:macc}
The operator $\cL_\perp: D(\cL_\perp)\subset L^2_\perp(M)\to L^2_{\perp}(M)$ defined in \eqref{def:Lperp} is m-accretive.
\end{lemma}
\begin{proof}
The domain $D(\cL_\perp)=H^2(M)$  is clearly dense on $L^2_{\perp}(M)$.
To check the resolvent bound, we notice that for $f\in D(\cL_\perp)$ such that $\norm{f}_{L^2_\perp}=1$ we have
\begin{equation} \label{eq:coercivity-accretive}
\norm{(a I+\cL_\perp)f}_{L^2_\perp}^2=a+\norm{\cL_\perp f}^2_{L^2_\perp}+2a \jap{f,\cL_\perp f}_{L^2_{\perp}}\geq a+2 a\nu\norm{\nabla f}_{L^2_\perp}^2\geq a. 
\end{equation}
Finally, defining the bilinear form $B : H^1 \cap L^2_\perp \times H^1 \cap L^2_\perp \to \R $ as 
$$B(\varphi, \psi) =a  \langle \varphi, \psi \rangle_{L^2_\perp} + \langle u \cdot \nabla \varphi , \psi  \rangle_{L^2_\perp} + \nu \langle \nabla \varphi , \nabla \psi  \rangle_{L^2_\perp}  $$
we deduce from Lax Milgram that for any $a>0$ and $\nu >0$ the operator $a I+\cL_\perp$ is surjective.
\end{proof}
We are now in the position of applying Theorem \ref{th:pseudoWei} to the operator $\cL_\perp$, meaning that we need to estimate the pseudo spectral abscissa defined in \eqref{def:pseudo}. The main result in this section is the following.
\begin{proposition} \label{prop:main}
Let $H$ be a Hamiltonian in the class $(\mathrm{A}^m)$ as defined in Definition \ref{def:ClassHamiltoniansA} and let $\cL_\perp: D(\cL_\perp)\subset L^2_\perp(M)\to L^2_{\perp}(M)$  be the operator defined in \eqref{def:Lperp}. Then, there exists a constant $\delta_*>0$ such that 
\begin{equation}
\label{bd:pseudoLperp}
\Psi(\cL_\perp)\geq \delta_*\lambda_\nu, \qquad \lambda_\nu :=\nu^{\frac{m}{m+2}}
\end{equation}
where $\Psi(\cdot)$ is defined in \eqref{def:pseudo}.
\end{proposition}
The proof strategy of this proposition is inspired by what is done for shear and radial flows \cite{Gallay:2021aa}. However, here there are major differences that give rise to several technical challenges to overcome. One major obstacle is the fact that we cannot perform a $k$-by-$k$ analysis since the operator $P_\perp\Delta P_\perp$ does not decouple in frequencies. The decoupling in frequencies is helpful since one can replace $\lambda$ by $\lambda/k$ and study a simpler one dimensional problem. On the other hand, here we need to be much more careful since we will need to sum over all $k$'s, and to ensure summability we have to exploit specific properties of the Hamiltonians we consider. 

The rest of this section is dedicated to the proof of Proposition \ref{prop:main}, which we have to split in different cases that needs to be treated separately. 
\subsection{Proof of Proposition \ref{prop:main}}
By the definition \eqref{def:pseudo}, the goal is to prove 
\begin{equation}
\label{bd:pseudolambda}
\norm{(\cL_\perp-i\lambda)f}_{L^2_\perp}\geq \delta_* \lambda_\nu\norm{f}_{L^2_\perp}, \qquad \text{ for all } \lambda \in \RR.
\end{equation}
In the sequel, we will often omit the $\perp$ and $L^2_\perp$ subscript to ease the notation, and we denote 
\begin{align}
&\cL_\lambda:=\cL_\perp-i\lambda=(u\cdot \nabla -i\lambda)-\nu P_\perp\Delta P_\perp,\\
&\jap{f_1,f_2}:=\jap{f_1,f_2}_{L^2(M)}=\int_M f_1 f_2 \dd x, \qquad \text{for all } f_1,f_2\in L^2_\perp(M).
\end{align}

A first rough upper bound on the pseudo spectral abscissa is readily obtained from the dissipative part of the operator. Namely, since $P_\perp f=f$ for all $f\in L^2_{\perp}$, integrating by parts we get
\begin{equation}
\Re(\jap{\cL_{\lambda} f,f})=-\nu \jap{\Delta P_\perp f, P_\perp f}=-\nu \jap{\Delta f,  f}=\nu \norm{\nabla f}^2
\end{equation}
Thus, from the Cauchy-Schwarz inequality we have
\begin{equation}
\label{bd:ReL0}
\nu \norm{\nabla f}^2=|\Re(\jap{\cL_{\lambda} f,f})|\leq \norm{\cL_{\lambda} f}\norm{f}.
\end{equation}
Since $P_0 f=0$, we can apply Poincaré's inequality (see also \eqref{eq:ThirdPoincare}) above and obtain 
\begin{equation}
\label{bd:ReL}
\norm{f}^2\leq C_P\norm{\nabla f}^2\qquad \Longrightarrow \qquad 
\norm{\cL_{\lambda} f}\geq \frac{\nu}{C_P} \norm{f},
\end{equation}
which is the standard spectral gap one expects from the dissipative operator $-\nu P_\perp \Delta P_\perp$. To improve this bound, we have to crucially exploit the advection, which in the estimate above has not been used. With the notation introduced in Section \ref{sec:Fourier}, we know that 
\begin{equation}
\mathcal{F}(u\cdot \nabla -i \lambda)=ik (\Omega(h)-\lambda/k).
\end{equation}
Since we are working on $L^2_\perp$, we have $k\neq 0$. Thus, we can hope to gain information from this part of the operator $\cL_\lambda$ whenever $\Omega(h)-\lambda/k$ is bounded from below. We need to introduce some sets and distinguish different cases to account for this. The general idea is to find a thin set $\cE_{\lambda,\delta}$ concentrated where $|\Omega(h)-\lambda/k|\ll1$ for all $k$ and split 
\begin{equation}
\label{eq:split}
\int_M |f|^2 \dd x=\int_{M\setminus \cE_{\lambda,\delta}} |f|^2 \dd x+\int_{\cE_{\lambda,\delta}} |f|^2 \dd x.
\end{equation}
For the integral concentrated on $\cE_{\lambda,\delta}$, we  exploit the ``thinness/smallness" of this set through the inequalities in Lemma \ref{lem:Poincare}. On the set $M\setminus \cE_{\lambda,\delta}$ we exploit the advection via the Fourier multiplier method used in \cite{WeiDiffusion19,Gallay:2021aa,LiPseudo20} for instance, suitably adapted to our case where we cannot handle each angular mode separately.

In the following, we recall that $m\in \N$ is the order of the critical point of the period function $\Omega(h)$, see Definition \ref{def:ClassHamiltoniansA}. Then, we fix two constants 
\begin{equation}
0<\delta=\delta(\nu)\ll 1, \qquad 0<\gamma_0\ll 1
\end{equation}
which will be specified later in the proof.

 We divide the proof on two cases, depending on how large $\lambda$ is. 
 
 \subsubsection{\textbf{Case}  $\boldsymbol{| \lambda | \leq \gamma_0\delta^{-1}}$} To perform the splitting argument in \eqref{eq:split}, we have to first introduce the thin sets and the main Fourier multiplier which will allows us to exploit properties of the advection operator. After proving some key properties of these sets, we can proceed with the pseudospectral bounds.
  
  \medskip
  \noindent
 $\circ$ \textbf{Thin sets.}
  Given $\lambda$, $k\neq0$ and $0<\delta\ll 1$, we define the thickened level sets
\begin{equation}
\label{def:Eklambda}
E_{k, \lambda, \delta} \coloneqq \left\{ (\theta, h) \in \T \times I : \left|\Omega(h) - \frac{\lambda}{k} \right| < \delta^m \right\}, \qquad \cE_{k,\lambda,\delta}=\Phi( E_{k, \lambda, \delta}),
\end{equation}
 where we are following the notation introduced in Section \ref{sec:toolbox}. 
  We use bold symbols to denote a $\delta$-neighbourhood of the sets above, namely 
   \begin{equation}
   \label{def:Eklambdathick}
 \boldsymbol{E}_{k, \lambda, \delta} =\{h \in I \, : \, \mathrm{dist}(h,E_{k,\lambda,\delta})<\delta\}, \qquad \boldsymbol{\mathcal{E}}_{k, \lambda, \delta}=\Phi( \boldsymbol{E}_{k, \lambda, \delta}).
 \end{equation}
 Notice that, with the stream-adapted rectangles defined in \eqref{def:Box1}, we also know   
 \begin{equation}
 \boldsymbol{E}_{k, \lambda, \delta}=\bigcup_{\{h'\in E_{k,\lambda,\delta}\}} Q_{\delta}(h'),
 \end{equation}
 where clearly one can extract a finite covering with rectangles thanks to the compactness of the set $E_{k,\lambda,\delta}$. For technical reasons, it is convenient to directly isolate also a set around the critical point without having any dependence on $k$. Hence, since we are assuming that the critical point is at $h=0$, following the notation \eqref{def:Box1}, given $\mathsf{C}_0$ sufficiently large to be specified afterwards, we define \footnote{The definition of this set is related to the specific choice of $q(h)$, namely one needs $|q'(h)|\lesssim\delta$. In the non-degenerate setting, we have $q(h)=h^2$ and therefore the sets are equivalent (up to constants).}
 \begin{equation}
 \label{def:Bdelta}
 Q_{\mathsf{C}_0\delta} =Q_\delta(0)= \{ (\theta, h) \in \T \times I : |h| \leq \mathsf{C}_0 \delta  \}, \qquad \mathcal{Q}_{\mathsf{C}_0\delta}=\Phi( Q_{\mathsf{C}_0\delta}).
 \end{equation}
 Finally, we choose the set for the splitting argument in \eqref{eq:split} as  
 \begin{equation}
 \label{def:Elambda}
 \mathcal{E}_{\lambda,\delta}  \coloneqq  \mathcal{Q}_{\mathsf{C}_0\delta} \cup \left (   \bigcup_{k \in \Z \setminus \{ 0\}}  \boldsymbol{\mathcal{E}}_{k, \lambda, \delta} \cap \mathcal{Q}_{\mathsf{C}_0\delta}^c  \right ), \qquad E_{\lambda,\delta}=\Phi^{-1}(\mathcal{E}_{\lambda,\delta}).
 \end{equation} 
These sets satisfies the following properties.
 \begin{lemma} \label{lemma:subset}
Let $|\lambda|\leq \gamma_0\delta^{-1}$, $k\neq 0$ be such that $E_{k,\lambda,\delta}\neq \emptyset$. Then, there exists a constant $C_*>0$, independent of $\lambda,k$, such that for any point $(h_{k,\lambda,\delta},\bar{\theta})\in E_{k,\lambda,\delta}\cap Q_{\delta}^c$, one has
\begin{equation}
\label{eq:incl}
\boldsymbol{E}_{k,\lambda,\delta}\cap Q^c_{\delta}\subset Q_{C_*\delta}(h_{k,\lambda,\delta}).
\end{equation}
Moreover, for any $k\neq k'$, 
\begin{equation}
\label{eq:disjoint}
 Q_{32C_*\delta}(h_{k,\lambda,\delta})\cap Q_{32C_*\delta}(h_{k',\lambda,\delta})=\emptyset.
\end{equation}
\end{lemma}
 \begin{proof}[Proof of Lemma \ref{lemma:subset}]
 Thanks to the assumption \ref{item:ClassACond3} in Definition \ref{def:ClassHamiltoniansA}, we have $\underline{c}\leq\Omega(h)\leq \overline{C}$. Therefore 
 \begin{equation}
 E_{k,\lambda,\delta}\neq \emptyset \iff \underline{c}-\delta^m\leq \frac{\lambda}{k}\leq \overline{C}+\delta^m.
 \end{equation}
 Thus, for $\delta$ sufficiently small compared to $\underline{c}$, we know that 
 \begin{equation}
 \label{eq:lambdak}
 |\lambda|\simeq |k|.
 \end{equation}
 In view of the assumption \ref{item:ClassACondOnDerivative}, we get
\begin{equation}
\inf_{h \in  Q_{\delta}^c} |\Omega'(h)|\geq \tilde{c}_* \delta^{m-1},
\end{equation}
for a constant $\tilde{c}_*$ not depending on $\lambda,k$. 
Then, let $h,h'\in E_{k,\lambda,\delta}\cap Q_{\delta}^c$. By the mean value theorem we know that 
\begin{equation}
|\Omega(h)-\Omega(h')|=|\Omega'(h_*)||h-h'|\geq  \tilde{c}_*\delta^{m-1}|h-h'|.
\end{equation}
On the other hand, since $h,h'\in E_{k,\lambda,\delta}$, by the triangle inequality we have 
\begin{equation}
|\Omega(h)-\Omega(h')|\leq|\Omega(h)-\lambda/k|+|\lambda/k-\Omega(h')|\leq 2\delta^m.
\end{equation}
Combining the two inequalities above, denoting  $\tilde{C}_*=2\tilde{c}_*^{-1}$, we get
\begin{equation}
|h-h'|\leq \tilde{C}_*\delta \qquad \text{for all } h,h'\in E_{k,\lambda,\delta}\cap Q_{\delta}^c.
\end{equation}
Since $\boldsymbol{E}_{k,\lambda,\delta}$ is a $\delta$-neighbourhood of $E_{k,\lambda,\delta}$, and the inequality above tells us that $E_{k,\lambda,\delta}\cap Q_{\delta}^c\subset Q_{\delta\tilde{C}_*}(h')$, we conclude that \eqref{eq:incl} holds true with $C_*=\tilde{C}_*+1$.

To prove the disjointness of the rectangles stated in \eqref{eq:disjoint}, let $h_{j,\lambda,\delta}\in E_{j,\lambda,\delta}$ with $j\in \{k,k'\}$ and $k\neq k'$. Denoting $h_j:=h_{j,\lambda,\delta}$, by the mean value theorem we have 
\begin{equation}
\label{eq:Omkk'}
|\Omega(h_k)-\Omega(h_{k'})|= |\Omega'(h_*)||h_k-h_{k'}| \leq C_1 |h_k-h_{k'}|,
\end{equation}
for some constant $C_1\geq 1$. On the other hand, since $h_k\in E_{k,\lambda,\delta}$ and $h_{k'}\in E_{k',\lambda,\delta}$, by the reverse triangle inequality, we get 
\begin{align}
\label{bd:lk0}
|\Omega(h_k)-\Omega(h_{k'})|&\geq \left|\frac{\lambda}{k}-\frac{\lambda}{k'}\right| -\left|\Omega(h_k)-\frac{\lambda}{k}\right|-\left|\Omega(h_{k'})-\frac{\lambda}{k'}\right| \geq \frac{|\lambda|}{|k|}\frac{|k-k'|}{|k'|}-2\delta^m.
\end{align}
Combining \eqref{eq:lambdak} (valid for both $k$ and $k'$) with the fact that $|k-k'|\geq 1$, we have 
\begin{equation}
\label{bd:lk1}
\frac{|\lambda|}{|k|}\frac{|k-k'|}{|k'|}\gtrsim |\lambda|^{-1}.
\end{equation}
Hence, since we are considering the case $|\lambda| \leq \gamma_0\delta^{-1}$, we know that there is constant $c_1>0$  such that 
\begin{equation}
|\Omega(h_k)-\Omega(h_{k'})|\geq \gamma_0^{-1} c_1\delta-2\delta^m.
\end{equation}
Choosing $\gamma_0= (100c_1 C_1 C_*)^{-1}$, from the inequality above and \eqref{eq:Omkk'} we get 
\begin{equation}
|h_k-h_{k'}|\geq 50 C_* \delta,
\end{equation}
whence proving \eqref{eq:disjoint}.
 \end{proof}

\medskip
\noindent
$\circ$ \textbf{Bounds close to level sets:} with the sets $\mathcal{E}_{\lambda,\delta}$ \eqref{def:Elambda} and Lemma \ref{lemma:subset} at hand, we are ready to estimate the second term in \eqref{eq:split}. In particular, we claim that there exists a universal constant $C_l>0$ such that 
 \begin{equation}
 \label{bd:claimE}
 \int_{\mathcal{E}_{\lambda,\delta} } |f|^2 \, dx   \leq  C_l\frac{\delta^2}{\nu}\| \cL_{\lambda} f \|\norm{f}+ \frac{1}{2} \| f \|^2 \,.
\end{equation} 
 To prove this claim, appealing to \eqref{eq:ThirdPoincare} and Lemma \ref{lem:Hamiltonian}, we first notice that
 \begin{equation}
 \int_{\mathcal{Q}_\delta } |f|^2 \, dx\leq \norm{T (\nabla H\circ \Phi)}_{L^\infty(Q_\delta)}^2\norm{\nabla f}_{L^2(\mathcal{Q_\delta})}^2\lesssim \delta^2\norm{\nabla f}^2.
 \end{equation}
 Using \eqref{bd:ReL0}, we get
 \begin{equation}
 \label{bd:Bdelta}
  \int_{\mathcal{Q}_\delta } |f|^2 \, dx\leq \tilde{C} \frac{\delta^2}{\nu}\norm{\cL_\lambda f}\norm{f} \end{equation}
for some universal constant $\tilde{C}$. This bound is consistent with the claim \eqref{bd:claimE}.
 
 Let us now consider the bound on the set $\boldsymbol{\cE}_{k,\lambda,\delta}\cap Q_{\delta}^c$, where we want to apply the Poincaré-type inequality across the streamlines \eqref{eq:SecondPoincare}. In particular, thanks to \eqref{eq:incl}, we observe that
 \begin{align}
 \int_{\boldsymbol{\cE}_{k,\lambda,\delta}\cap \mathcal{Q}_{\delta}^c}|f|^2\dd x&= \int_{\boldsymbol{E}_{k,\lambda,\delta}\cap Q_{\delta}^c}|f\circ\Phi|^2q'(h)T(h)\dd h\dd\theta\\
 &\leq \int_{ Q_{C_*\delta}(h_{k})}|f\circ\Phi|^2q'(h)T(h)\dd h\dd\theta=\norm{f}^2_{L^2(\mathcal{Q}_{C_*\delta}(h_{k}))},
 \end{align}
 where $h_k=h_{k,\lambda,\delta}$. Therefore, applying \eqref{eq:SecondPoincare} with $\eta=1$, $\gamma=C_*\delta$ and $K=32$,  we get
\begin{align}
 \int_{\boldsymbol{\cE}_{k,\lambda,\delta}\cap \mathcal{Q}_{\delta}^c}|f|^2\dd x\leq \, &\frac14 \norm{f}^2_{L^2(\mathcal{Q}_{32C_*\delta}(h_{k}))}\\
 \notag& +32C_*\delta \norm{f}_{L^2(\mathcal{Q}_{32C_*\delta}(h_{k}))}\norm{\nabla f}_{L^2(\mathcal{Q}_{32C_*\delta}(h_{k}))}\norm{\de_h\Phi}_{L^\infty(Q_{32C_*\delta}(h_{k}))}\\
\label{bd:Ek}
 & +16C_*\delta \norm{f}_{L^2(\mathcal{Q}_{32C_*\delta}(h_{k}))}^2\norm{[\log(q'T)]'}_{L^\infty(Q_{32C_*\delta}(h_{k}))}.
\end{align}
Thanks to the choice of the rescaled action angle variables we know that $\de_h \Phi \in L^\infty$.
To bound the last term in \eqref{bd:Ek}, since $h_k\in \boldsymbol{\mathcal{E}}_{k,\lambda,\delta} \cap Q_{\delta}^c$, we know that
\begin{align}
\norm{[\log(q'T)]'}_{L^\infty(Q_{32C_*\delta}(h_{k}))}&\leq \norm{\left|\frac{T'}{T}\right|+\frac{q''}{q'}}_{L^\infty(Q_{32C_*\delta}(h_{k}))}\\
\label{bd:log}&\lesssim \norm{h^{-1}}_{L^\infty(Q_{32C_*\delta}(h_{k}))}\leq \frac{\tilde{C}_0}{\mathsf{C}_0\delta},
\end{align}
where in the last inequality we used $q(h)=h^2$, that $T'/T$ bounded and $\mathsf{C}_0$ is the constant involved in the definition \eqref{def:Bdelta}. Hence, taking $\mathsf{C}_0$ sufficiently large we infer
\begin{align}
\label{bd:Ek1}
 \int_{\boldsymbol{\cE}_{k,\lambda,\delta}\cap \mathcal{Q}_{\delta}^c}|f|^2\dd x\leq \frac38 \norm{f}^2_{L^2(\mathcal{Q}_{32C_*\delta}(h_{k}))}+\tilde{C}_1 \delta \norm{f}_{L^2(\mathcal{Q}_{32C_*\delta}(h_{k}))}\norm{\nabla f}_{L^2(\mathcal{Q}_{32C_*\delta}(h_{k}))}.\end{align}
Using the bound above, in account of \eqref{eq:disjoint}, we can consider the union over $k$ and prove that  
\begin{align}
\int_{\cup_{k\in \ZZ\setminus \{0\}}\boldsymbol{\cE}_{k,\lambda,\delta}\cap \mathcal{Q}_{\delta}^c}|f|^2\dd x\leq \frac38\norm{f}^2+\tilde{C}_1\delta \norm{f}\norm{\nabla f}.
\end{align}
Using again \eqref{bd:ReL0}, we conclude that
\begin{align}
\int_{\cup_{k\in \ZZ\setminus \{0\}}\boldsymbol{\cE}_{k,\lambda,\delta}\cap \mathcal{Q}_{\delta}^c}|f|^2\dd x\leq \frac12\norm{f}^2+\tilde{C}_3\frac{\delta^2}{\nu} \norm{\cL_\lambda f}\norm{f}.
\end{align}
Combining the bound above with \eqref{bd:Bdelta}, the claim \eqref{bd:claimE} is proved.

\medskip 
\noindent
$\circ$ \textbf{Bounds away from the level sets:}  Let us introduce the key Fourier multiplier needed to exploit the effects of the transport operator. We define $\cX \colon L^2(M) \to L^2(M)$ to be an operator with Fourier symbol $\chi(k,h)$ (with the Fourier notation introduced in Section \ref{sec:Fourier}), namely 
 \begin{equation}
 \label{def:chi}
 \widehat{((\cX f) \circ \Phi)}(k, h) = \chi(k, h) \widehat{f \circ \Phi}(k, h).
 \end{equation}
 We choose $\chi \colon \Z \times I\to \RR$ to be a function (which can be explicitly constructed) satisfying the following properties: 
\begin{enumerate}
\item \label{chi:inequality} $\chi ( k, h) \operatorname{sign}\left( k\Omega(h) - {\lambda} \right) \geq 0$ for any $h$ and $k \neq 0$.
\item \label{chi-derivative} $\norm{\chi}_{L^\infty}\leq 1$ and $\norm{\de_h\chi}_{L^\infty} \lesssim \delta^{-1}$.
\item  \label{chi-sign} $ \chi(k, h) = \operatorname{sign}\left( k\Omega(h) - {\lambda} \right)$ for any $h \in \boldsymbol{E}_{k, \lambda, \delta}^c$ and for any $k \neq 0$.
\item \label{chi-zero}$\chi (0 , \cdot) \equiv 0$.
\end{enumerate}
The key properties of $\cX$ are encoded in its definition. In particular, we see that we can apply the case a) in Lemma \ref{lemma:operator-estimates} with $N=\delta^{-1}$. 

We can now proceed with the proof. Integrating by parts, we have
 \begin{align} \label{eq:proof-identity}
  \Im \langle \cL_{ \lambda} f , \cX f \rangle = \Im \langle u \cdot \nabla f - i \lambda f, \cX f \rangle+ \nu \Im \langle \nabla f, \nabla (\cX f)  \rangle:=\cI_1+\cI_2.
 \end{align}
 To control $\cI_1$, we exploit the change of variables introduced in Section \ref{sec:rescaled-action-angle} and the Fourier expansion in Section \ref{sec:Fourier}. In particular,  we have
 \begin{align}
 \cI_1 &=  \Im \int_I \sum_{k\neq 0} i(k\Omega(h)-\lambda)\chi(k,h)|\widehat{f\circ \Phi}|^2(k,h) q(h)T(h)\dd h\\
 &\geq\int_{E_{\lambda,\delta}^c} \sum_{k\neq 0} |k|\left|\Omega(h)-\frac{\lambda}{k}\right||\widehat{f\circ \Phi}|^2(k,h) q(h)T(h)\dd h\\
\label{bd:lowI10} &\geq \delta^m \int_{M\setminus{\cE}_{\lambda,\delta}} |f|^2\dd x,
 \end{align}
 where in the last inequality we used that $h\notin E_{k,\lambda,\delta}$ (see \eqref{def:Eklambda}) and we undid the change of variables.
 
 To control $\cI_2$, by Cauchy-Schwarz and case a) in Lemma \ref{lemma:operator-estimates}, we have 
 \begin{equation}
 \label{bd:I2}
 |\cI_2|\leq \nu \norm{\nabla f}\norm{\nabla (\cX f)}\lesssim \nu\norm{\nabla f}^2+\frac{\nu}{\delta}\norm{\nabla f}\norm{f}
 \end{equation}
Since $|\chi|\leq 1$, we also know that 
 \begin{equation}
 \label{bd:claimW}
\norm{\cX f}\leq \norm{f}.
 \end{equation}
Hence, combining \eqref{eq:proof-identity}, \eqref{bd:lowI10} and \eqref{bd:I2} we get 
 \begin{align}
  \int_{M\setminus \cE_{\lambda,\delta}} |f|^2\dd x&\leq \frac{1}{\delta^{m}} \cI_1\leq \frac{1}{\delta^{m}}(\norm{\cL_\lambda f}\norm{\cX f}+|\cI_2|)\\
  &\lesssim \frac{1}{\delta^{m}}\left(\norm{\cL_\lambda f}\norm{ f}+\nu \norm{\nabla f}^2+\frac{\nu}{\delta}\norm{\nabla f} \norm{f}\right),\\
   &\leq \frac{1}{\delta^{m}}\left(\tilde{C}_4\norm{\cL_\lambda f}\norm{ f}+\tilde{C}_5\delta_0^{-(2+m)}\nu \norm{\nabla f}^2+\frac{\delta_0^{2+m}}{4}\frac{\nu}{\delta^2} \norm{f}^2\right)\\
   \label{bd:M-E}&\leq \frac{1}{\delta^m}\left(\tilde{C}_6\norm{\cL_\lambda f}\norm{ f}+\frac{\delta_0^{2+m}}{4}\frac{\nu}{\delta^2}\norm{f}^2\right),
 \end{align}
 where in the last inequality we used \eqref{bd:ReL0} and $\delta_0$ is a small fixed constant (with $\tilde{C}_6$ depending on $\delta_0$). 
 
 \medskip
 \noindent
 $\circ$ \textbf{Concluding the proof.} It remains to optimize the choice of $\delta$ in order to conclude the proof. From \eqref{bd:M-E}, we can only hope to absorb the last term in that inequality if we choose $\delta$ such that 
 \begin{equation}
 \label{def:choicedelta}
 \frac{\nu}{\delta^2}\simeq \delta^m\qquad \Longrightarrow \qquad \delta:=\delta_* \nu^{\frac{1}{m+2}}.
 \end{equation}
 Therefore, combining \eqref{bd:M-E} with \eqref{bd:claimE} and the choice of $\delta$, we finally arrive at 
 \begin{align}
 \norm{f}^2&=\int_{M\setminus{\cE}_{\lambda,\delta}} |f|^2\dd x+\int_{{\cE}_{\lambda,\delta}} |f|^2\dd x\leq \tilde{C}_7 \nu^{-\frac{m}{m+2}}\norm{\cL_\lambda f}\norm{ f}+\frac34\norm{f}^2,
 \end{align}
 which implies 
 \begin{equation}
 \norm{\cL_\lambda f}\geq \frac{1}{4\tilde{C}_7} \nu^{\frac{m}{m+2}} \norm{f}
 \end{equation}
 whence proving \eqref{bd:pseudolambda} for  $0<\delta_*:=1/(4\tilde{C}_7)\ll1 $ which can be explicitly computed.

 \subsubsection{\textbf{Case}  $\boldsymbol{| \lambda | > \gamma_0\delta^{-1}}$} In this regime, we are heuristically studying the high-$k$-frequency case separately. Indeed, whenever $|\Omega(h)-\lambda/k|\ll 1$ and $\Omega(h)\leq \overline{C}$, we necessarily have $|\lambda/k|\lesssim 1$ and therefore $|k|\gtrsim \delta^{-1}$. With our choice of $\delta$ in \eqref{def:choicedelta}, we would then have $ |k|\gg \nu^{-\frac{1}{m+2}}$, which is a situation where the dissipation  is expected to be much stronger. 
 \begin{remark}
 For shear or radial flows, the $k$-by-$k$ pseudospectral estimate is of order $\nu^{\frac{m}{m+2}}|k|^{\frac{2}{m+2}}$, whereas the standard heat equation estimate would give $\nu k^2$. These two quantities are equal when $|k|=\nu^{-\frac{1}{m+1}}\gg \nu^{-\frac{1}{m+2}}$. Therefore, we do not expect that in this regime the dissipation is already the dominant effect.
 \end{remark}
From a practical point of view, we simply change the definition of \eqref{def:Eklambda} as follows: 
\begin{equation}
\label{def:Eklambdahigh}
E_{k, \lambda, \delta} \coloneqq \left\{ (\theta, h) \in \T \times I : \left|\Omega(h) - \frac{\lambda}{k} \right| < \frac{\delta^{m-1}}{\mathsf{C}_1|k|} \right\}, \qquad \cE_{k,\lambda,\delta}=\Phi( E_{k, \lambda, \delta}),
\end{equation}
where $\mathsf{C}_1>0$ is a sufficiently large constant specified below. Then, instead of using a $\delta$-neighbourhood of the set above, we use a $|k|^{-1}$-neighbourhood, namely 
   \begin{equation}
   \label{def:Eklambdathickhigh}
 \boldsymbol{E}_{k, \lambda, \delta} =\{h \in I \, : \, \mathrm{dist}(h,E_{k,\lambda,\delta})<|k|^{-1}\}, \qquad \boldsymbol{\mathcal{E}}_{k, \lambda, \delta}=\Phi( \boldsymbol{E}_{k, \lambda, \delta}).
 \end{equation}
The set $Q_{\mathsf{C}_0\delta}$ in \eqref{def:Bdelta} remains the same whereas \eqref{def:Elambda} is redefined accordingly. We have the following.
 \begin{lemma} \label{lemma:subsethigh}
Let $|\lambda|> \gamma_0\delta^{-1}$, $k\neq 0$ be such that $E_{k,\lambda,\delta}\neq \emptyset$. Then, there exists a constant $\mathsf{C}_*>0$, independent of $\lambda,k$, such that for any point $(h_{k,\lambda,\delta},\bar{\theta})\in E_{k,\lambda,\delta}\cap Q_{\delta}^c$, one has
\begin{equation}
\label{eq:incl1}
\boldsymbol{E}_{k,\lambda,\delta}\cap Q^c_{\delta}\subset Q_{\mathsf{C}_*/|k|}(h_{k,\lambda,\delta}).
\end{equation}
Moreover, for any $k\neq k'$, 
\begin{equation}
\label{eq:disjoint1}
 Q_{32\mathsf{C}_*/|k|}(h_{k,\lambda,\delta})\cap Q_{32\mathsf{C}_*/|k|}(h_{k',\lambda,\delta})=\emptyset.
\end{equation}
\end{lemma}
\begin{proof}
The proof follows by rescaling the bounds in Lemma \ref{lemma:subset}, namely replace all the occurences of $\delta^m$ with $\delta^{m-1}/(\mathsf{C}_1|k|)$ and change accordingly the universal constants involved. In fact, the proof of \eqref{eq:incl1} is the same of \eqref{eq:incl} with a possibly different $\mathsf{C}_*$ compared to $C_*$ in Lemma \ref{lemma:subset}. Then, we only use $|\lambda|\leq \gamma_0\delta^{-1}$ to prove the disjointness property \eqref{eq:disjoint}. In this case, using the same notation in the proof of Lemma \ref{lemma:subset}, the inequality \eqref{bd:lk0} becomes 
\begin{equation} \label{eq:period-bdd-used}
|\Omega(h_k)-\Omega(h_{k'})|\gtrsim \frac{|\lambda|}{|k|}\frac{|k-k'|}{|k|}-2\frac{\delta^{m-1}}{\mathsf{C}_1|k|}\geq \frac{1}{2|k|}\left(\underline{c}-\frac{4\delta^{m-1}}{\mathsf{C}_1}\right).
\end{equation} 
Thus, choosing $\mathsf{C}_1$ sufficiently large we argue as done to prove \eqref{eq:disjoint} and prove \eqref{eq:disjoint1}.
\end{proof}
For the Fourier multiplier, we define $\cX$ as in \eqref{def:chi}, but we choose $\chi$ satisfying the properties \eqref{chi:inequality}, \eqref{chi-sign}, \eqref{chi-zero} and we replace \eqref{chi-derivative} with 
\begin{enumerate}
\item[(2)] \label{chi-derivative1} $\norm{\chi}_{L^\infty}\leq 1$ and $\norm{\de_h\chi}_{L^\infty} \lesssim |k|$.
\end{enumerate}
This is because now we are working with a $|k|^{-1}$-neighbourhood of $E_{k,\lambda,\delta}$. Hence, we are in case b) of Lemma \ref{lemma:operator-estimates}, meaning that we will control $\|\nabla (\chi f)\|$ with $\|\nabla f\|$. This apparent loss of optimality is compensated by the fact that now the thin sets are of size $\delta^{m-1}/|k|\ll \delta^{m}$ for all $|k|$ sufficiently large.
 
 We are now in the position of concluding the proof. Thanks to the properties in Lemma \ref{lemma:subsethigh} and the fact that $|k|\simeq |\lambda|\gtrsim \delta^{-1}$, we can repeat the proof of \eqref{bd:claimE} and show that there exists a universal constant $C_{l_1}>0$ such that 
 \begin{equation}
 \label{bd:claimE1}
 \int_{\mathcal{E}_{\lambda,\delta} } |f|^2 \, dx   \leq  C_{l_1}\frac{\delta^2}{\nu}\| \cL_{\lambda} f \|\norm{f}+ \frac{1}{2} \| f \|^2 \,.
\end{equation} 
For the bounds away from the level sets, we consider the splitting in \eqref{eq:proof-identity} but we now control $\cI_1$ as follows:
 \begin{align}
 \cI_1 &=\Im\jap{(u\cdot \nabla-i\lambda)f,\cX f}=  \Im \int_I \sum_{k\neq 0} i(k\Omega(h)-\lambda)\chi(k,h)|\widehat{f\circ \Phi}|^2(k,h) q(h)T(h)\dd h\\
 &\geq\int_{E_{\lambda,\delta}^c} \sum_{k\neq 0} |k|\left|\Omega(h)-\frac{\lambda}{k}\right||\widehat{f\circ \Phi}|^2(k,h) q(h)T(h)\dd h\\
\label{bd:lowI1} &\geq \delta^{m-1} \int_{M\setminus{\cE}_{\lambda,\delta}} |f|^2\dd x.
 \end{align}
 Therefore, similarly as done before in the case $| \lambda | \leq \gamma_0 \delta^{-1}$, using 
 the identity \eqref{eq:proof-identity}, the bound of the Fourier multiplier $\cX$ \eqref{bd:nablaWf2} in Lemma \ref{lemma:operator-estimates} and the identity \eqref{bd:ReL0}, we get 
  \begin{align}
  \label{bd:M-E1}  \int_{M\setminus \cE_{\lambda,\delta}} |f|^2\dd x&\leq \frac{1}{\delta^{m-1}} \tilde{C}_{6_1}\norm{\cL_\lambda f}\norm{ f}.
 \end{align}
Combining together \eqref{bd:claimE} and \eqref{bd:M-E1} with the choice made before $\delta = \nu^{\frac{1}{m+2}}$, we get 
 \begin{align}
 \norm{f}^2\leq \tilde{C}_{7_1}(\delta \nu^{-\frac{m}{m+2}}  + \nu^{-\frac{m}{m+2}})\norm{\cL_\lambda f}\norm{ f}+\frac34\norm{f}^2,
 \end{align}
 which  implies 
 \begin{equation}
 \norm{\cL_\lambda f}\geq \frac{1}{8\tilde{C}_{7_1}} \nu^{\frac{m}{m+2}} \norm{f}.
 \end{equation}
Therefore  the bound \eqref{bd:pseudolambda} holds true with  $0<\delta_*:=1/\max\{4\tilde{C}_7,8\tilde{C}_{7_1})\}\ll1 $, whence concluding the proof of Proposition \ref{prop:main}. \qed

\section{Examples}
\label{sec:examples}
In this section, we firstly prove a proposition that allows to consider a general class of Hamiltonians satisfying our assumptions of Definition \ref{def:ClassHamiltoniansA} and Definition \ref{def:ClassHamiltoniansP}. Then, we provide two examples where the average along the streamline has a non-trivial evolution, implying in particular that the commutator $[\Delta, P_0]$ is non-trivial.

\subsection{Perturbation of non-degenerate radial Hamiltonians}
\label{sec:Radial}
\begin{proposition}\label{prop:QuantitativePerturbationAreInClassP}
 Let $B_R(0) \subset \R^2$ be the ball of radius $R > 0$ centered in the origin. Let $H \colon B_R(0) \to \R$ be a smooth radial Hamiltonian with a non-degenerate elliptic point in $0$ and 
 $$\lambda I \leq D^2 H(x) \leq \Lambda I \quad \forall x \in B_R(0).$$ 
 Let $f \in C^{\infty}(B_R(0))$ such that $\nabla f(0) = 0$ and $f$ is constant on $\partial B_R(0)$.
 Then there exists a constant $C$ depending only on $\lambda$, $\Lambda$ and $\| f \|_{C^{3}}$ such that for all $\repsilon < \frac{1}{2C}$ the Hamiltonian $H_{\repsilon} \colon B_R(0) \to \R$ defined by
 \[
  H_{\repsilon} = H + \repsilon f
 \]
 belongs to the class $(P_{\repsilon})$ with parameter $\eps = C \repsilon$.
\end{proposition}

\begin{proof}[Proof of Proposition~\ref{prop:QuantitativePerturbationAreInClassP}]
 Without loss of generality we may assume that $H(0) = 0$ and $f(0) = 0$.
We use $r, \varphi$ to denote the standard polar coordinates in $B_R(0)$.
Recall that $H_{\repsilon} \colon B_R(0) \to \R$ is defined as
\[
 H_{\repsilon}(r, \varphi) = H(r) + \repsilon f (r, \varphi).
\]
The proof is divided into 4 steps. In the first step, we define a useful function by means of the implicit function theorem. Then in the second step, we adapt this mapping to obtain a so-called rescaled action-angle change of variables. The third step is dedicated to estimate two second order derivatives of the mapping constructed in Step 2. Finally, we prove the Proposition in Step 4.

\medskip
\noindent
\textbf{Step 1: Definition of function $\pmb{r}$.} 
Let
\[
 \repsilon_0 = \frac{\lambda}{2 \| f \|_{C^3}}.
\]
and define 
$$X = \Big\{ (\repsilon, h, \varphi) \in (-\repsilon_0, \repsilon_0) \times \R \times \T_{2 \pi} : h \in H \big( B_R(0) \setminus \{ 0 \} \big) \Big\}.$$
We claim that there exists a map $r \colon X \to B_R(0)$ such that $H_{\repsilon}(r(\repsilon, h, \varphi), \varphi) = h^2$ for all $(\repsilon, h, \varphi) \in X$ and
\begin{align}
 \partial_{\repsilon} r(\repsilon, h, \varphi) &= - \dfrac{f(r(\repsilon, h, \varphi), \varphi)}{\partial_{r} H_{\repsilon}(r(\repsilon, h, \varphi), \varphi)}; \label{eq:RDifferentiatedByEps} \\
 \partial_{h} r(\repsilon, h, \varphi) &= \dfrac{2h}{\partial_{r} H_{\repsilon}(r(\repsilon, h, \varphi), \varphi)}; \label{eq:RDifferentiatedByH} \\
 \partial_{\varphi} r(\repsilon, h, \varphi) &= - \dfrac{\partial_{\varphi}H_{\repsilon}(r(\repsilon, h, \varphi), \varphi)}{\partial_{r} H_{\repsilon}(r(\repsilon, h, \varphi), \varphi)} \label{eq:RDifferentiatedByVarphi}
\end{align}
for all  $(\repsilon, h, \varphi) \in X$.
In order to build such a map, we define $G \colon X \times (0, R] \to \R$ by
\[
 G(\repsilon, h, \varphi, r) = H_{\repsilon}(r, \varphi) - h^2 = H(r) + \repsilon f (r, \varphi) - h^2.
\]
Note that 
$$\partial_r G(\repsilon, h, \varphi, r) = \partial_r H(r) + \repsilon \partial_r f (r, \varphi) \geq \frac{\lambda}{2} r$$
for all $(\repsilon, h, \varphi, r) \in X \times (0, R]$.
Let $h_0 \in I$ be arbitrary and let $r_0 \in (0,R]$ such that $H(r_0) = h_0^2$. Then $G(0, h_0, 0, r_0) = 0$.
Hence by the implicit function theorem there exists a relatively open neighbourhood $U \subset X$ of $(0, h_0, 0)$ and a function $r \colon U \to (0,R]$ such that
\[
 G(\repsilon, h, \varphi, r(\repsilon, h, \varphi) ) = 0
\]
and \eqref{eq:RDifferentiatedByEps}, \eqref{eq:RDifferentiatedByH} and \eqref{eq:RDifferentiatedByVarphi} hold for $(\repsilon, h, \varphi) \in U$.
This function $r$ only proves the desired result locally. 
At this stage, we notice that whenever $r$ exists we have
\begin{equation}\label{eq:TheFunctionrIsSimilarToh}
 \frac{\lambda}{2} h \leq |r(\eps, h, \varphi)| \leq 2 \Lambda h.
\end{equation}
As a consequence, thanks to the fact that \eqref{eq:RDifferentiatedByEps}, \eqref{eq:RDifferentiatedByH} and \eqref{eq:RDifferentiatedByVarphi} hold wherever $r$ is defined we obtain that
\begin{equation}\label{eq:TheFunctionrHasBoundedDerivatives}
 |\partial_{\eps} r(\eps, h, \varphi)|, \, |\partial_{h} r(\eps, h, \varphi)|, \, |\partial_{\varphi} r(\eps, h, \varphi)| \leq C
\end{equation}
for some constant $C$ depending only on $\lambda$, $\Lambda$ and $\| f \|_{C^3}$.
The fact that $r$ can be defined on the whole set $X$ follows from the fact that whenever $r(\repsilon, h, \varphi)$ is defined it is unique combined with the fact that whenever $r$ is defined on a relatively open set $U \subset X$, then it can be extended to it closure $\overline{U}$ in $X$. Let us prove this last fact. Let $(\repsilon, h, \varphi) \in \overline{U}$. Then there is a sequence $\{ (\repsilon, h, \varphi) \}_{\ell = 1}^{\infty} \subset U$ such that
\[
 (\repsilon_{\ell}, h_{\ell}, \varphi_{\ell}) \to (\repsilon, h, \varphi) \text{ as $\ell \to \infty$.}
\]
We set $r_{\ell} = r(\repsilon_{\ell}, h_{\ell}, \varphi_{\ell})$. 
Thanks to \eqref{eq:TheFunctionrHasBoundedDerivatives}, we deduce that $r_{\ell}$ converges to a unique limit $r$ and $G(\eps, h, \varphi, r) = 0$. Hence, we may define $r(\eps, h, \varphi) = r$.
This is enough to conclude.

\medskip
\noindent
\textbf{Step 2: Definition of $\Phi_{\repsilon}$.}
We define the mapping $\MAP \colon [- \repsilon_0, \repsilon_0] \times \T \times I \to B_R(0)$ by
\[
 \MAP(\repsilon, \gamma, h) = (r(\repsilon, h, 2 \pi \gamma) \cos (2 \pi \gamma), r(\repsilon, h, 2 \pi \gamma) \sin (2 \pi \gamma)).
\]
Note that for each $\repsilon \in [- \repsilon_0, \repsilon_0]$, the mapping
\[
 (\gamma, h) \mapsto \MAP(\repsilon, \gamma, h)
\]
is a parametrisation of $B_R(0)$.
From \eqref{eq:RDifferentiatedByEps}, \eqref{eq:RDifferentiatedByH} and \eqref{eq:RDifferentiatedByVarphi} it follows that
\begin{align*}
\partial_{\repsilon} \MAP(\repsilon, \gamma, h) &= \partial_{\repsilon} r(\repsilon, h, 2 \pi \gamma) (\cos(2 \pi \gamma), \sin(2 \pi \gamma)) =  - \dfrac{f(r(\repsilon, h, \varphi), \varphi)}{\partial_{r} H_{\repsilon}(r(\repsilon, h, \varphi), \varphi)} (\cos(2 \pi \gamma), \sin(2 \pi \gamma)); \\
\partial_{h} \MAP(\repsilon, \gamma, h) &= \partial_{h} r(\repsilon, h, 2 \pi \gamma) (\cos(2 \pi \gamma), \sin(2 \pi \gamma)) = \dfrac{2h}{\partial_{r} H_{\repsilon}(r(\repsilon, h, \varphi), \varphi)} (\cos(2 \pi \gamma), \sin(2 \pi \gamma)); \\
 \partial_{\gamma} \MAP(\repsilon, \gamma, h) &= 2 \pi \dfrac{r(\repsilon, h, 2 \pi \gamma)}{\partial_{r} H_{\repsilon} (r(\repsilon, h, 2 \pi \gamma), 2 \pi \gamma)} \nabla^{\perp} H_{\repsilon} (\MAP(\repsilon, \gamma, h)).
\end{align*}
Moreover the period function of the Hamiltonian $H_{\repsilon}$ is given by
\begin{align}
\begin{split}\label{eq:PeriodFormula}
 T_{\repsilon}(h) &= \int_{\{ H_{\repsilon} = h^2 \}} \dfrac{1}{|\nabla H_{\repsilon}|} \, d \ell = \int_0^1 \dfrac{1}{|\nabla H_{\repsilon}(\MAP(\repsilon, \gamma, h))|} |\partial_{\varphi}\MAP(\repsilon, \gamma, h)| \, d \gamma \\
 &= 2 \pi \int_0^1 \dfrac{r(\repsilon, h, 2 \pi \gamma)}{\partial_{r} H_{\repsilon} (r(\repsilon, h, 2 \pi \gamma), 2 \pi \gamma)} \, d \gamma.
\end{split}
\end{align}
For each $\repsilon \in [- \repsilon_0, \repsilon_0]$ and $h \in I$, let $s_{\repsilon, h} \colon \T \to \T$ be the solution to the ordinary differential equation
\[
 \begin{cases}
  \partial_{\theta} s_{\repsilon, h}(\theta) = \dfrac{1}{2 \pi} T_{\repsilon}(h) \dfrac{\partial_{r} H_{\repsilon} (r(\repsilon, h, 2 \pi s_{\repsilon, h}(\theta)), 2 \pi s_{\repsilon, h}(\theta))}{r(\repsilon, h, 2 \pi s_{\repsilon, h}(\theta))}; \\
  s_{\repsilon, h}(0) = 0.
 \end{cases}
\]
Indeed, the fact that a unique solution exists follows from the Cauchy-Lipschitz theory. 
Then, for all $\repsilon \in [- \repsilon_0, \repsilon_0]$ define $\Phi_{\repsilon} \colon \T \times I \to B_R(0)$ as
\[
 \Phi_{\repsilon}(\theta, h) = \MAP(\repsilon, s_{\repsilon, h}(\theta), h).
\]
Standard computations imply
\begin{align}
\begin{split}\label{eq:PhiEpsDifferentiatedWRTTheta}
 \partial_{\theta} \Phi_{\repsilon}(\theta, h) &= T_{\repsilon}(h) \nabla^{\perp} H_{\repsilon} (\Phi_{\repsilon}(\theta, h)) \\
\end{split}
\end{align}
and
\begin{align}
  \partial_{\repsilon} \Phi_{\repsilon}(\theta, h) &= \partial_{\repsilon} \MAP(\repsilon, s_{\repsilon, h}(\theta), h) + \partial_{\gamma} \MAP(\repsilon, s_{\repsilon, h}(\theta), h) \partial_{\repsilon} s_{\repsilon, h}(\theta); \\
 \partial_{h} \Phi_{\repsilon}(\theta, h) &= \partial_{\gamma} \MAP(\repsilon, s_{\repsilon, h}(\theta), h) \partial_{h} s_{\repsilon, h}(\theta) + \partial_{h} \MAP(\repsilon, s_{\repsilon, h}(\theta), h).
\end{align}
Now observe that for all $\repsilon \in [- \repsilon_0, \repsilon_0]$, the mapping
\[
 (\theta, h) \mapsto \Phi_{\repsilon}(\theta, h)
\]
satisfies all the assumptions of a so-called rescaled action-angle change of variables with respect to the Hamiltonian $H_{\repsilon}$.

\medskip
\noindent
\textbf{Step 3: Derivation of bounds on  $\partial_{\repsilon \theta} \Phi_{\repsilon}$ and $\partial_{\repsilon h} \Phi_{\repsilon}$.}
Our goal is now to derive bounds on the following two quantities:
\[
 \partial_{\repsilon \theta} \Phi_{\repsilon} (\theta, h) \quad \text{and} \quad \partial_{\repsilon h} \Phi_{\repsilon} (\theta, h).
\]
Using the fact that
\[
 |f(r, \varphi)| = O (r^2), \, |\partial_{r} f(r, \varphi)| = O (r), \, |\partial_{rr} f(r, \varphi)| = O (1)
\]
and \eqref{eq:TheFunctionrHasBoundedDerivatives} we can prove that there exists a constant $\tilde{C}$ depending only on $\lambda$, $\Lambda$ and $\| f \|_{C^3}$ such that
\begin{equation}
 \dfrac{1}{\tilde{C}} \leq |T_{\repsilon}(h)| \leq \tilde{C}, \quad |\partial_{\repsilon} T_{\repsilon}(h)| \leq \tilde{C}, \quad |\partial_{h} T_{\repsilon}(h)| \leq \tilde{C}, \quad |\partial_{\repsilon h} T_{\repsilon}(h)| \leq \dfrac{\tilde{C}}{h}
\end{equation}
and
\begin{equation}
 \dfrac{1}{\tilde{C}} h \leq |\Phi_{\repsilon}(\theta, h)| \leq \tilde{C} h, \quad |\partial_{\repsilon} \Phi_{\repsilon}(\theta, h)| \leq \tilde{C} h, \quad |\partial_{h} \Phi_{\repsilon}(\theta, h)| \leq \tilde{C}.
\end{equation}

\bigskip
With these estimates at hand, we now derive bounds on $\partial_{\repsilon \theta} \Phi_{\repsilon} (\theta, h)$ and $\partial_{\repsilon h} \Phi_{\repsilon} (\theta, h)$. We start with $\partial_{\repsilon \theta} \Phi_{\repsilon} (\theta, h)$. From \eqref{eq:PhiEpsDifferentiatedWRTTheta}, we note that
\begin{equation}\label{eq:DerivativeEpsTheta}
\begin{split}
 \partial_{\repsilon \theta} \Phi_{\repsilon} (\theta, h) &= \partial_{\repsilon} T_{\repsilon}(h) \nabla^{\perp} H_{\repsilon}(\Phi_{\repsilon}(\theta, h)) + T_{\repsilon}(h) D \nabla^{\perp} H_{\repsilon}(\Phi_{\repsilon}(\theta, h)) \partial_{\repsilon} \Phi_{\repsilon}(\theta, h) \\
 &\qquad + T_{\repsilon}(h) \nabla^{\perp} f(\Phi_{\repsilon}(\theta, h)).
\end{split}
\end{equation}
Therefore, thanks to the Claim there exists a constant $C_1$ depending only on $\lambda$, $\Lambda$ and $\| f \|_{C^3}$ such that
\begin{align*}
 | \partial_{\repsilon \theta} \Phi_{\repsilon} (\theta, h) | &\leq C_1 h.
\end{align*}
Now we treat $\partial_{\repsilon h} \Phi_{\repsilon} (\theta, h)$.
Differentiating \eqref{eq:DerivativeEpsTheta} with respect to $h$ yields
\begin{align*}
 \partial_{\repsilon \theta h} \Phi_{\repsilon} (\theta, h) &= \partial_{\repsilon h} T_{\repsilon}(h) \nabla^{\perp} H_{\repsilon}(\Phi_{\repsilon}(\theta, h)) + \partial_{\repsilon} T_{\repsilon}(h) D \nabla^{\perp} H_{\repsilon}(\Phi_{\repsilon}(\theta, h)) \partial_{h} \Phi_{\repsilon}(\theta, h)\\
 &\qquad + \partial_{h} T_{\repsilon}(h) D \nabla^{\perp} H_{\repsilon}(\Phi_{\repsilon}(\theta, h)) \partial_{\repsilon} \Phi_{\repsilon}(\theta, h) + T_{\repsilon}(h) D^2 \nabla^{\perp} H_{\repsilon}(\Phi_{\repsilon}(\theta, h)) \partial_{h} \Phi_{\repsilon}(\theta, h) \partial_{\repsilon} \Phi_{\repsilon}(\theta, h) \\
 &\qquad + T_{\repsilon}(h) D \nabla^{\perp} H_{\repsilon}(\Phi_{\repsilon}(\theta, h)) \partial_{\repsilon h} \Phi_{\repsilon}(\theta, h) \\
 &\qquad + \partial_{h} T_{\repsilon}(h) \nabla^{\perp} f(\Phi_{\repsilon}(\theta, h)) + T_{\repsilon}(h) D \nabla^{\perp} f(\Phi_{\repsilon}(\theta, h)) \partial_{h} \Phi_{\repsilon}(\theta, h) \\
\end{align*}
Therefore, there exists a constant $C_2$ depending only on $\lambda$, $\Lambda$ and $\| f \|_{C^3}$ such that
\begin{align*}
 |\partial_{\repsilon \theta h} \Phi_{\repsilon} (\theta, h)| \leq C_2 + C_2 |\partial_{\repsilon h} \Phi_{\repsilon} (\theta, h)|
\end{align*}
and by Gr\"onwall's inequality we find
\[
 |\partial_{\repsilon h} \Phi_{\repsilon}(\theta, h)| \leq C_3 \e^{C_3}
\]
for some $C_3$ depending only on $\lambda$, $\Lambda$ and $\| f \|_{C^3}$.
We conclude that there exists a constant $C$ depending only on $\lambda$, $\Lambda$ and $\| f \|_{C^3}$ such that
\[
 |\partial_{\repsilon \theta} \Phi_{\repsilon}(\theta, h)| \leq C h \quad \text{and} \quad |\partial_{\repsilon h} \Phi_{\repsilon}(\theta, h)| \leq C.
\]

\medskip
\noindent
\textbf{Step 4: Final step}
The goal of this final step is to prove that there exists a constant $C$ depending on $\lambda$, $\Lambda$ and $\| f \|_{C^3}$ such that for all $\repsilon < \frac{1}{2C}$, the Hamiltonian $H_{\repsilon}$ belongs to $(P_{\eps})$ with parameter $\eps = C \repsilon$. First of all, it is immediately clear that $H_{\repsilon}$ satisfies points \ref{item:ClassACond1} and \ref{item:ClassACond2} of Definition~\ref{def:ClassHamiltoniansA} for all $\repsilon < \repsilon_0$. Moreover, it is clear from \eqref{eq:PeriodFormula} that $\Omega \colon I \to \R^+$ belongs to $C^1(I)$ and is bounded from above and below by strictly positive constants. Thus $H_{\repsilon}$ satisfies point \ref{item:ClassACond3} of Definition~\ref{def:ClassHamiltoniansA}. Thus it only remains to prove the points \ref{item:ClassPCond4} and \ref{item:ClassPCond5} of Definition~\ref{def:ClassHamiltoniansP}.
We start by obtaining estimates and properties of the functions $\partial_{\theta} \Phi_{0}(\theta, h)$ and $\partial_{h} \Phi_{0}(\theta, h)$. Since $H_0 = H$ is radial, it is not difficult to see that $\partial_{\theta} \Phi_{0}(\theta, h)$ and $\partial_{h} \Phi_{0}(\theta, h)$ are orthogonal. Since $H_0$ is radial, it is also not difficult to prove that there exists a constant $c_0$ depending only on $\lambda$, $\Lambda$ and $\| f \|_{C^3}$ such that
\begin{equation}\label{eq:BoundsForEpsZeroOnDerivatives}
 \dfrac{1}{c_0} \leq |\partial_h \Phi_0(\theta, h)| \leq c_0 \quad \text{and} \quad \dfrac{1}{c_0} h \leq |\partial_{\theta} \Phi_0(\theta, h)| \leq c_0 h.
\end{equation}
Now, our aim is to establish transversality and uniformity for the Hamiltonian $H_{\repsilon}$.
From Step 3, we know that there exists a constant $C^{\prime}$ depending only on $\lambda$, $\Lambda$ and $\| f \|_{C^3}$ such that
\begin{equation}\label{eq:RecallingWhatComesFromStep3}
 | \partial_{\theta} \Phi_{\repsilon}(\theta, h) - \partial_{\theta} \Phi_{0}(\theta, h) | \leq C^{\prime} h \repsilon \quad \text{and} \quad | \partial_{h} \Phi_{\repsilon}(\theta, h) - \partial_{h} \Phi_{0}(\theta, h) | \leq C^{\prime} \repsilon.
\end{equation}
We start with transversality. The inequalities above imply that there exists a constant $C_T$ depending only on $\lambda$, $\Lambda$ and $\| f \|_{C^3}$ such that for any $\repsilon$ such that $|\repsilon| < \frac{1}{2 C_T}$ we have
\[
 \partial_{\theta} \Phi_{\repsilon}(\theta, h) \cdot \partial_{h} \Phi_{\repsilon}(\theta, h) \leq C_T \repsilon | \partial_{\theta} \Phi_{\repsilon}(\theta, h) | | \partial_{h} \Phi_{\repsilon}(\theta, h) | \quad \text{for all } (\theta, h).
\]
Now, we prove uniformity. It follows from a straightforward computation that
\begin{align*}
 \partial_{\repsilon} \left( |\partial_{\theta} \Phi_{\repsilon}(\theta, h)|^2 \right) &\leq C_4 h^2
\end{align*}
for some constant $C_4$ depending only on $\lambda$, $\Lambda$ and $\| f \|_{C^3}$.
As a consequence,
\begin{align}
\begin{split}\label{eq:TransversalityEstimateINQuantitativeProof}
 \left| |\partial_{\theta} \Phi_{\repsilon}(\theta, h)|^2  - \int_{0}^{1} |\partial_{\theta} \Phi_{\repsilon}(\tilde{\theta}, h)|^2 \, d \tilde{\theta} \right| &\leq \left| |\partial_{\theta} \Phi_{\repsilon}(\theta, h)|^2 - |\partial_{\theta} \Phi_{0}(\theta, h)|^2\right| \\
 &\qquad + \left| \int_{0}^{1} |\partial_{\theta} \Phi_{\repsilon}(\tilde{\theta}, h)|^2 \, d \tilde{\theta} - \int_{0}^{1} |\partial_{\theta} \Phi_{0}(\tilde{\theta}, h)|^2 \, d \tilde{\theta}\right| \\
 &\qquad + \underbrace{\left| |\partial_{\theta} \Phi_{0}(\theta, h)|^2  - \int_{0}^{1} |\partial_{\theta} \Phi_{0}(\tilde{\theta}, h)|^2 \, d \tilde{\theta} \right|}_{= 0} \\
 &\leq \left| |\partial_{\theta} \Phi_{\repsilon}(\theta, h)|^2 - |\partial_{\theta} \Phi_{0}(\theta, h)|^2\right| \\
 &\qquad + \int_{0}^{1} \left| |\partial_{\theta} \Phi_{\repsilon}(\tilde{\theta}, h)|^2 \, d \tilde{\theta} - |\partial_{\theta} \Phi_{0}(\tilde{\theta}, h)|^2 \right| \, d \tilde{\theta} \\
 &\leq 2 C_4 h^2 \repsilon.
\end{split}
\end{align}
Moreover, from \eqref{eq:BoundsForEpsZeroOnDerivatives} and \eqref{eq:RecallingWhatComesFromStep3}, we have
\begin{equation*}
 |\partial_{\theta} \Phi_{\repsilon}(\theta, h)| \geq \frac{1}{c_0} h - C^{\prime} h \repsilon \geq \frac{1}{2 c_0} h
\end{equation*}
whenever $|\repsilon| \leq \frac{1}{2 c_0 C^{\prime}}$.
It follows that
\begin{equation}\label{eq:BoundFromBelowOnThetaDerivative}
 |\partial_{\theta} \Phi_{\repsilon}(\theta, h)|^2 \geq \frac{1}{4 c_0^2} h^2.
\end{equation}
Combining \eqref{eq:TransversalityEstimateINQuantitativeProof} and \eqref{eq:BoundFromBelowOnThetaDerivative} we see that there exists a constant $C_U$ depending only on $\lambda$, $\Lambda$ and $\| f \|_{C^3}$ such that for any $\repsilon$ such that $|\repsilon| < \frac{1}{2 C_U}$ we have
\[
 \left| |\partial_{\theta} \Phi_{\repsilon}(\theta, h)|^2  - \int_{0}^{1} |\partial_{\theta} \Phi_{\repsilon}(\tilde{\theta}, h)|^2 \, d \tilde{\theta} \right| \leq C_U \repsilon |\partial_{\theta} \Phi_{\repsilon}(\theta, h)|^2.
\]
This proves \ref{item:ClassPCond4} and \ref{item:ClassPCond5} in Definition~\ref{def:ClassHamiltoniansP} and hence concludes the proof.
\end{proof}

\begin{remark}
 In Proposition~\ref{prop:QuantitativePerturbationAreInClassP} above, we only considered perturbations $f$ which are constant on the boundary of a ball so that the resulting perturbed Hamiltonians are constant on the boundary of the ball. Hence the resulting velocity field is tangent to the boundary of the ball and enters into our framework. One could state a similar result dropping the assumption that that $f$ is constant on a ball. The price to pay for this is that the domain on which perturbations are defined will depend on $\repsilon$ since the domain has to be selected so that the Hamiltonian is constant on the boundary. To prove a statement of this kind, one can extend the perturbation $f$ to a larger ball in such a way that it is constant on the boundary of the larger ball and then apply Proposition~\ref{prop:QuantitativePerturbationAreInClassP}.
\end{remark}

\subsection{On the evolution of $\rho_0$ and the commutator $[P_0,\Delta]$}

We firstly provide an example where the solution of the advection-diffusion equation starts with an initial data with zero average along the streamline, and, in finite time, the average along the streamlines nearly matches the mass of the initial data as measured in $L^1$. This example exploits the hyperbolic point of the cellular flow and highlights the observation that, in a short time span, non-trivial phenomena can occur due to differences in width between two streamlines (this phenomenon is ruled out by our assumptions \ref{item:ClassACond4}, \ref{item:ClassPCond4}). 

\begin{proposition}\label{thm:example1}
Let $H(x_1, x_2) = \sin \left( \frac{x_1}{2 \pi} \right) \sin \left( \frac{x_2}{2 \pi} \right) \in C^\infty (\T^2)$, then for any  $\epsilon >0$ there exist $0 < \nu (\epsilon) \ll 1$ and  $\rho^{\initial} \in L^\infty$ such that $P_0 \rho^{\initial} =0$, $\| \rho^{\initial} \|_{L^\infty} \leq 1$ and the solutions $\rho$ of \eqref{eq:advdiff} with $u = \nabla^\perp H$ is such that 
$$ \| P_0 \rho (1, \cdot) \|_{L^1} \geq \left ( 1 - \epsilon \right) \| \rho^{\initial} \|_{L^1} $$
\end{proposition}

\begin{proof}
Let $\delta >0$ be a free parameter depeding on $\epsilon >0$.
We consider the rescaled action angle variables $(h, \theta)$ introduced in Section \ref{sec:rescaled-action-angle} with $\{ H = h^2 \}$ and let $\Phi (\theta , h)$ be the change of variable $\Phi : \T \times [0,1] \to A \subset \T^2$ with $\Phi (0,0) = (0,0)$, the hyperbolic point and $A$ a neighbourhood containing the hyperbolic point. We now define 
$$\rho^{\text{in}} = \rho^{\text{in}}_+ + \rho^{\text{in} }_- $$
where, for $\delta$ sufficiently small, $\rho^{\text{in}}_- \in C^\infty$ satisfies:
\begin{itemize}
\item $\| \rho^{\text{in}}_- \|_{L^\infty} \leq 1 $ and $\rho^{\text{in}}_- \leq 0$;
\item $ \text{supp} (\rho^{\text{in}}_- ) \subset \Phi ([\theta_0, \theta_0 + 1 / T(\delta )] \times [\delta, 2 \delta]) $, $\int_{\T^2} \rho^{\text{in}}_- = - \frac{ \delta^2}{4}$, where $T$ is the period and $\theta_0 = \text{argmin} \{ \dist (\Phi (\theta , \delta) , (0,0)) \}$. Since $H (x_1, x_2) \sim x_1 x_2$ near the hyperbolic point $(0,0)$ we have 
$$ \left | \Phi (\theta_0, h) - \Phi \left  (\theta_0 + \frac{1}{T(h)}, h \right ) \right | =    \int_{\theta_0}^{\theta_0 + 1}  | \nabla H (X_t)| dt \sim h $$
where the last $\sim$ follows from the fact that $|\nabla H (x)| \leq |x|$ for any $x \in \T^2$ and 
$$|X_t (x)| \leq |x| + \int_0^t | \nabla H (X_s)|  ds \leq |x | + \int_0^t |X_s| ds.$$ 
Therefore, by Gr\"onwall we have $|X_1 (x)| \leq e |x| $ which implies the property we need because we choose $\theta_0$ so that  $| \nabla H ( X_{\theta_0} (x))| \leq |X_{\theta_0} (x)| \sim h$ (we are  close to the hyperbolic point).
 Therefore, since $h \in [\delta , 2 \delta]$,
 it follows that there exists $x_0 \in \T^2$ and $C>0$ such that  $ \text{supp} \rho^{\text{in}}_- \subset B (x_0, C \delta)$, meaning that $$| \text{supp} (\rho^{\text{in}}_-)| \sim \| \rho^{\text{in}}_- \|_{L^1}.$$ 
 We will use this fact later. Finally, up to mollification, we can suppose that $\| \nabla \rho^{\text{in}}_- \|_{L^\infty} \leq C \delta^{-1} $. 
\end{itemize}

We define $\rho^{\text{in}}_+$ satisfying the following:
\begin{itemize}
\item In rescaled action-angle variables we define 
$$ \rho^{\text{in}}_+ (\theta , h) = - \varphi ( \theta)  \int_0^1 \rho_-^{\initial} (\Phi (\theta, h)) d \theta$$
where $\varphi = \mathbbm{1}_{[\overline{\theta}_0 , \overline{\theta}_0 + \frac{1}{T (h)}]}$  and
 $\overline{\theta}_0 (h) = \text{argmin} \{ \dist (\Phi (\theta , h) , (0,1/2)) \}$. It can be checked that $\text{supp} (\rho^{\text{in}}_{+}) \subset [0, C \delta^2]\times [1/6, 5/6]$. Finally we consider $\overline{\rho}^{\text{in}}_{+} = \mathbbm{1}_{[0, 2 C \delta^2] \times [0,1] } \star \psi $ independent on $y$, for some convolution kernel $\psi = \frac{2 C}{\delta^{2}} \psi(x/ 2C \delta^{2})$. Using the support  of the function, it is possible to  check that $\rho^{\text{in}}_{+} \leq \overline{\rho}^{\text{in}}_{+}$ pointwise and we also have that  $\|  \nabla \overline{\rho}^{\text{in}}_{+} \|_{L^\infty} \lesssim \delta^{-2}$ and $ \| \overline{\rho}^{\text{in}}_{+} \|_{L^1} \sim   \| {\rho}^{\text{in}}_{+} \|_{L^1} \sim \delta^2$.
\end{itemize}

It is clear by definition that  
$$ P_0 \rho^{\text{in}} = P_0 (\rho^{\text{in}}_+ + \rho^{\text{in}}_-) =0 \,. $$
and $\| \rho^{\text{in}}\|_{L^\infty} \leq 1$.

We introduce the following functions:
\begin{itemize}
\item we denote with $\rho_-^{\adf}, \rho_+^{\adf}$ the solutions of the advection diffusion equation \eqref{eq:advdiff} with initial data $  \rho^{\text{in}}_-$ and $  \rho^{\text{in}}_+$ respectively;
\item we denote $\rho_{-}^{\adv}$ the solution of the advection equation, that is \eqref{eq:advdiff} with $\nu=0$, with initial datum $  \rho^{\text{in}}_- $. We will compare $\rho_-^{\adf}$ with $\rho_-^{\adv}$;
\item  we denote with $\overline{\rho}_{+}^{\heat}$ the solution of the heat equation, namely \eqref{eq:advdiff} with $u\equiv 0$, with initial datum $\overline{\rho}_+^{\initial}$ and  $\overline{\rho}_{+}^{\adf}$ the solution of the advection diffusion with the initial datum $\overline{\rho}_+^{\initial}$. We will compare $\overline{\rho}_+^{\adf}$ with $\overline{\rho}_{+}^{\heat}$ and use that $ {\rho}_+^{\adf} \leq \overline{\rho}_+^{\adf}$ pointwise.
\end{itemize}

We now have the following inequality  at a fixed time $t=1$
\begin{align*}
\| P_0 \rho \|_{L^1} & = \| P_0 \rho_-^{\adf} + P_0  \rho_+^{\adf} \|_{L^1 (\Phi ((0,1) \times [0, 2 \delta]))} +  \| P_0 \rho_-^{\adf} + P_0  \rho_+^{\adf} \|_{L^1 (\Phi ((0,1) \times [ 2 \delta , 1]))} 
\\
&  \geq \| P_0 \rho_-^{\adf} \|_{L^1 (\Phi ((0,1) \times [0, 2 \delta]))} - \| P_0  \rho_+^{\adf} \|_{L^1 (\Phi ((0,1) \times [0, 2 \delta]))}
\\
& \quad + \| P_0  \rho_+^{\adf} \|_{L^1 (\Phi ((0,1) \times [2 \delta, 1]))} -  \| P_0  \rho_-^{\adf} \|_{L^1 (\Phi ((0,1) \times [2 \delta, 1]))} \,.
\end{align*}
Regarding these four last quantities, we claim the following estimates which we shall prove later. We have
\begin{align} \label{comm:claim1}
&\| P_0 \rho_{-}^{\adf} (1, \cdot ) \|_{L^1 (\Phi ((0,1) \times [0, 2 \delta]))} \geq (1 - \sfrac{\epsilon}{4}) \|  \rho^{\text{in}}_- \|_{L^1} \,, \\
& \| P_0 \rho_{-}^{\adf} (1, \cdot ) \|_{L^1 (\Phi ((0,1) \times [2 \delta, 1]))} \leq   \sfrac{\epsilon}{4}  \|  \rho^{\text{in}}_- \|_{L^1} \,,
\end{align}
 and 
 \begin{align} \label{comm:claim2}
 & \| P_0 \rho_{+}^{\adf} (1, \cdot ) \|_{L^1 (\Phi ((0,1) \times [0, 2 \delta]))} \leq   \sfrac{\epsilon}{4} \|  \rho^{\text{in}}_+ \|_{L^1}  \,, \\
  & \| P_0 \rho_{+}^{\adf} (1, \cdot ) \|_{L^1 (\Phi ((0,1) \times [2 \delta, 1]))} \geq (1-    \sfrac{\epsilon}{4})  \|  \rho^{\text{in}}_+ \|_{L^1}.
 \end{align}
Therefore, using that $\| \rho_{-}^{\text{in}} \|_{L^1} + \| \rho_{+}^{\text{in}} \|_{L^1} = \| \rho^{\text{in}} \|_{L^1}$, thanks to the disjointness of the supports of $\rho_{-}^{\text{in}}$ and $ \rho_{+}^{\text{in}}  $,
we conclude the proof.

We now prove \eqref{comm:claim1} and \eqref{comm:claim2}.
We firstly prove \eqref{comm:claim1}. Using the representation of the transport equation with the backward flow map associated to $u = \nabla^{\perp} H$, still denoted with $X_\cdot$, by the Feynman-Kac formula we know that  
$$\rho_{-} (t, x ) = \mathbb{E} \rho_-^{\initial} (X^{\nu}_{t, 0} (x)). $$ 
Moreover, we have  $\mathbb{E} \| X_{1,0} - X^\nu_{1,0} (\cdot)  \|_{L^\infty ( \T^2)} \leq C \sqrt{\nu}$, where $X^\nu$ is the backward stochastic flow. Then
$$\| \rho_{-}^{\adv} (1, \cdot ) - \rho_{-}^{\adf} (1, \cdot ) \|_{L^1_x}  \lesssim  \sqrt{\nu} \| \nabla  \rho^{\text{in}}_- \|_{L^\infty} |\text{supp}  \rho^{\text{in}}_-  | \lesssim \sqrt{\nu} \delta^{-1} \|  \rho^{\text{in}}_- \|_{L^1} \,. $$
Choosing $\nu = \delta^{2 + \gamma}$ with $\gamma \in (0,1/4)$ to be a fixed constant and $\delta$ sufficiently small so that  $\delta^{\gamma/2} \leq \frac{\epsilon}{4C}$,  observing that $\|P_0 f \|_{L^1} \leq \| f \|_{L^1}$, we conclude  
\begin{align*}
\| P_0 \rho_{-}^{\adf} (1, \cdot ) \|_{L^1 (\Phi ( (0,1) \times [0, 2 \delta] ) )} & \geq \| P_0 \rho_{-}^{\adv} (1, \cdot ) \|_{L^1 (\Phi ((0,1) \times [0, 2 \delta]))} - \| \rho_{-}^{\adv} (1, \cdot ) - \rho_{-}^{\adf} (1, \cdot ) \|_{L^1}  
\\
& \geq 
(1 - \epsilon/4) \|  \rho^{\text{in}}_- \|_{L^1}. 
\end{align*}
Adding to both sides of the previous inequality the term $\| P_0 \rho_{-}^{\adf} (1, \cdot ) \|_{L^1 (\Phi (0,1) \times [2 \delta, 1])} $, and using that $ \| \rho^{\text{in}}_- \|_{L^1} = \| P_0 \rho^{\text{in}}_- \|_{L^1} = \| P_0 \rho_-  (1, \cdot ) \|_{L^1}$ (thanks to the non-negativity of the functions and the conservation of the average along the evolution), we have
\begin{align*}
\| P_0 \rho_{-}^{\adf} (1, \cdot ) \|_{L^1 (\Phi ( (0,1) \times [2 \delta, 1] ) )} \leq \epsilon/4 \| \rho^{\text{in}}_- \|_{L^1} 
\end{align*}
and therefore the first part of the claim \eqref{comm:claim1} holds. 

Now we consider  the second part of the claim \eqref{comm:claim2}.
From the explicit formula of the heat equation solution, for  $t \geq  \delta^{ \gamma }$ it is not hard to deduce that
$$| \overline{\rho}_{+}^{\heat} (t,  x)| \leq 2 \left |  \frac{1}{\sqrt{4 \pi \nu t}} \int_{\T} \exp \left ( - \frac{| x - y_1|^2}{4 \nu t } \right ) \mathbbm{1}_{[0, 2 C \delta^2]} (y_1) d y_1 \right | \lesssim \frac{\delta^2}{\sqrt{\nu t }} \lesssim \delta^{1 - \sfrac{\gamma}{2}} \,,$$
for all $x \in \T^2$, where the last holds since  $\nu = \delta^{2+ \gamma}$.

Now, let  $X^{1}$ be the first component of the backward flow of $u$ and  $X^{\nu, 1}$ be the first component of the stochastic flow. Observing that $\overline{\rho}_+^{\initial}$ does not depend on the second variable, we use   the Feynman-Kac formula to conclude that for any $t$ the following holds
\begin{align*}
\|  \overline{\rho}_{+}^{\heat}  (t, \cdot) - \overline{\rho}_{+}^{\adf} (t, \cdot ) \|_{L^1} & \leq \int_{\mathbb{T}^2 } \| \partial_x \overline{\rho}_+^{\initial} \|_{L^\infty} \mathbbm{1}_{\text{supp} (\overline{\rho}_+^{\initial})} (x) \mathbb{E} | X^{\nu,1}_{1,0} (x, \cdot ) - \sqrt{2 \nu} W_t | dx \,.
\end{align*}
 Using that $| \partial_y H | \leq | \sin (x) | \leq |x| $ we can prove that 
for any $t \leq 1$
$$\mathbb{E} | X^\nu_{t,0} (x, \cdot ) - x - \sqrt{2 \nu} W_t | \leq \mathbb{E} \int_0^t  | \partial_y H (X^\nu_{t, s} (x, \cdot ))| ds \leq \mathbb{E} \int_0^t  | X^\nu_{t, s} (x, \cdot )| ds \leq 2 t |x|$$
where  in the last we crucially used that $t \leq 1$ to have that $\mathbb{E}   | X^\nu_{t, s} (x, \cdot )| \leq 2 |x|$ for any $t, s \leq 1$. Therefore, using also $\|\mathbbm{1}_{\text{supp} (\overline{\rho}_+^{\initial})} \|_{L^1} \lesssim \|  \overline{\rho}_+^{\initial} \| _{L^1}$,  $\| \partial_x \overline{\rho}_+^{\initial} \|_{L^\infty} \lesssim \delta^{-2}$ and $|x| \lesssim \delta^2$ for any $x \in \text{supp} (\overline{\rho}_+^{\initial})$ we conclude  
$$ \|  \overline{\rho}_{+}^{\heat}  (t, \cdot) - \overline{\rho}_{+}^{\adf} (t, \cdot ) \|_{L^1} \lesssim t    \|  \overline{\rho}_+^{\initial} \| _{L^1} \lesssim \delta^{\gamma}  \|  {\rho}_+^{\initial} \| _{L^1}$$
for any $t \leq \delta^{\gamma}$.

It can be shown that $\mathcal{L}^2 (\Phi ( (0,1) \times [0, 2 \delta])) \leq \delta^{3/2}$ and that for any non negative function $f$ we have $\| P_0 f \|_{L^1} = \| f \|_{L^1}$.
Using these, the non negativity of $\rho_+^{\initial}$, the fact that the average of the solutions to the advection diffusion equation  is conserved along the evolution and the maximum principle, we conclude 
\begin{align*}
\| P_0 \rho_{+}^{\adf} (1, \cdot ) &  \|_{L^1 (\Phi ( (0,1) \times [0, 2 \delta]))}  = \|  \rho_{+}^{\adf} (1, \cdot ) \|_{L^1 (\Phi ( (0,1) \times [0, 2 \delta] ) )}    = \|  \rho_{+}^{\adf} (\delta^{\gamma}, \cdot ) \|_{L^1 (\Phi ( (0,1) \times [0, 2 \delta] ) )}
\\
&  \leq  \|  \overline{\rho}_{+}^{\adf} (\delta^{\gamma} , \cdot ) \|_ {L^1 (\Phi ( (0,1) \times [0, 2 \delta] ) )} 
\\
& \leq  \|  \overline{\rho}_{+}^{\adf} (\delta^{\gamma}, \cdot ) -  \overline{\rho}_{+}^{\heat} (\delta^{\gamma}, \cdot) \|_{L^1} 
+  \|  \overline{\rho}_{+}^{\heat} (\delta^{\gamma}, \cdot ) \|_{L^1 (\Phi ( (0,1) \times [0, 2 \delta] ) )}
\\
& \lesssim \delta^{\gamma/8} \| \rho_{+}^{\text{in}} \|_{L^1} + \delta^{\sfrac{3}{2}} \delta^{ 1- \gamma}   \lesssim \delta^{\gamma}  \| \rho_{+}^{\text{in}} \|_{L^1} \,,
\end{align*}
where in the last inequality we used $ \| \rho_{+}^{\text{in}} \|_{L^1}  \sim \delta^2$ and $\gamma < 1/4$.
 Choosing $\delta$ even smaller, i.e. such that   $\delta^{\gamma} \ll \epsilon$ to reabsorb the universal constants we conclude $\| P_0 \rho_{+}^{\adf} (1, \cdot ) \|_{L^1 (\Phi ( (0,1) \times [0, 2 \delta] ) )} \leq \sfrac{\epsilon}{4} \|  \rho^{\text{in}}_+ \|_{L^1} $.

Finally, as done before, by adding   $\| P_0 \rho_{+}^{\adf} (1, \cdot ) \|_{L^1 (\Phi ( (0,1) \times [2 \delta, 1] ) )} $ to both sides of the previous inequality and using that $ \| \rho^{\text{in}}_+ \|_{L^1} = \| P_0 \rho^{\adf}_+ \|_{L^1} $, we conclude 
 $\| P_0 \rho_{+}^{\adf} (1, \cdot ) \|_{L^1 (\Phi ( (0,1) \times [2 \delta, 1] ) )} \geq (1-    \sfrac{\epsilon}{8})  \|  \rho^{\text{in}}_+ \|_{L^2} $. This, we prove \eqref{comm:claim2} and this concludes the proof of the proposition.\end{proof}

We now provide a second example where we show that  the the commutator $[P_0, \Delta]$ can really affect the $L^2$ dynamics in a non-trivial way. In fact, in the previous example we cannot have a lower bound in $L^2$ due to the unboundedness of the period near the hyperbolic point. Here, 
the analysis is performed near a non-degenerate elliptic point where the period is bounded. 

In the example we construct, we have an Hamiltonian not complying with the assumptions in Definition \ref{def:ClassHamiltoniansA}. In particular, there are stagnation regions where $\nabla H=0$ (the Hamiltonian is constant in a portion of the domain). The proof does not rely on these stagnation regions and can be adapted to avoid this issue. However, this requires some more technicalities, as the use of Feynman-Kac formula or the heat kernel in more general bounded domains instead of $\TT^2$. The main purpose of this example is investiganting the behavior of $\rho_0$ when we are localized around elliptic points, and we prefer to keep the proof lighter at the price of having an Hamiltonian not belonging to the class \ref{HamiltonianClassA}.
\begin{proposition} \label{thm:example2}
There exists $H: \T^2 \to \R$ such that the following holds. For any $\nu \in (0, 1/8)$, there exist $\rho^{\initial} \in L^\infty (\T^2) $ with $P_0 \rho^{\initial}  =0$ and $t \in (0, 1)$ such that 
$$\| P_0 \rho (t, \cdot) \|_{L^2} \geq \exp (- 16) \|  \rho^{\initial}  \|_{L^2}  \,,$$
where $\rho$ is the solution to \eqref{eq:advdiff} with initial datum $\rho^{\initial}$ and velocity field $u=\nabla^\perp H$.
\end{proposition}

\begin{proof}
Without loss of generality, we assume the torus to be $[-1/2 ,1/2]^2$ with the usual identifications of the boundary.
Let us fix $H(x,y) =  \varphi (x, y) \left (x^2 + \frac{y^2}{9} \right ) $ with $\varphi \in C^\infty_c (B(0,1/2))$, $\varphi \equiv 1$ on $B(0, 1/4)$ and $\| \varphi \|_{L^\infty} \leq 1$. We use the rescaled action-angle variables introduced in Section \ref{sec:rescaled-action-angle} with
$H(x,y) = h^2 $, angle $\Phi (\theta, h) = X(\theta T(h), z(h))$ as usual and $z(h) = (h, 0)$. An explicit computation shows that the period of the ellipse $x^2 + \frac{y^2}{9}$ is constant and equal to $6 \pi$.
Now, for any $\nu \in (0, 1/8)$ we define $\delta = \nu^2$ and $t = \nu^3$ and we are ready to define the initial datum as 
$$\rho^{\initial} = \rho^{\initial}_+ - \rho^{\initial}_-$$
where we define the function using the rescaled action angle variables $\rho^{\initial}_+ (\theta, h)= c \mathbbm{1}_{[5 \delta, 6 \delta]} (h)  \mathbbm{1}_{[0, \sfrac{1}{20}]} (\theta)$, where $c$ is a constant independent on $\delta$ so that $\int \rho^{\initial}_+ = \delta^2 $. Then 
$$\rho^{\initial}_- (\theta, h) = -  20c  \mathbbm{1}_{[\overline{\theta}, \overline{\theta} + \sfrac{1}{20}]} (\theta) \int_0^1 \rho^{\initial}_+ (\Phi (\theta, h)) d \theta  = - c \mathbbm{1}_{[5 \delta, 6 \delta]} (h)  \mathbbm{1}_{[\overline{\theta}, \overline{\theta} + \sfrac{1}{20}]} (\theta) \,,$$ where $\overline{\theta}$ is such that $\Phi (\overline{\theta}, \delta) = (0, 3 \delta)$. It is not hard to see that $P_0 \rho^{\initial} =0$. We remark that $c>0$ can be computed explicitly, namely $c^{-1}= \frac{1}{20 \delta^2} \int_{5 \delta}^{6 \delta} 2 h T(h) dh = \frac{12 \pi}{20 \delta^2} 11 \delta^2 \simeq 6.6$. Therefore, we also have  $$ \| \rho_{\initial} \|_{L^2} \leq \delta.$$
Considering the support of $\rho_{-}^{\initial},\rho_{+}^{\initial}$ in Cartesian coordinates, it is not difficult to see that
$$\text{supp} (\rho_{-}^{\initial}) \subset [- \delta/4 , \delta /4 ] \times [15 \delta , 18 \delta]  \qquad \text{supp} (\rho_{+}^{\initial}) \subset [ 4 \delta  , 6 \delta ] \times [0 ,  \delta]. $$
 Now, we introduce the notations
\begin{itemize}
\item $\rho^{\heat}_+$ and $\rho^{\heat}_-$ are the solutions to the heat equation, namely \eqref{eq:advdiff} with $u\equiv0$, and initial data $\rho^{\initial}_+$ and $\rho^{\initial}_-$, respectively. We set $\rho^{\heat} = \rho^{\heat}_+ + \rho^{\heat}_- $.
\item $\rho_+$ and  $\rho_-$ are the solutions to \eqref{eq:advdiff} with initial data $\rho^{\initial}_+$ and $\rho^{\initial}_-$, respectively.
\end{itemize}
From energy estimates and using $u = \nabla^\perp H$, for any $\iota \in \{ +, - \}$  we have 
\begin{align*}
 \| \rho_\iota (t, \cdot)  - \rho_\iota^{\heat}  (t, \cdot ) \|_{L^2}^2 & \leq 2 \left |  \int_0^t \int_{\T^2} u \cdot \nabla \rho_\iota \rho_\iota^{\heat} \right | 
 \\
 & \leq 2 \|u \|_{L^\infty}  \left ( \int_0^t \int_{\T^2} |  \nabla \rho_\iota |^2  \right )^{1/2} \left ( \int_0^t \int_{\T^2} |   \rho_\iota^{\heat} |^2  \right )^{1/2}
 \\
 & \leq  \frac{2 \|u \|_{L^\infty} \| \rho^{\initial} \|_{L^2}^2 \sqrt{t} }{\sqrt{\nu}}  \leq 2 \nu \| \rho^{\initial} \|_{L^2}^2 \,.
\end{align*}
 Now we study the properties of the heat equation solutions $\rho_+^{\heat}$ and $\rho_-^{\heat}$. Firstly, we study $P_0 \rho_\iota^{\heat} (h)$ for $h \in [ 10 \delta  , 11 \delta]$. 
We consider the domains in rescaled action angle variables
$$A =  \{ h \in [10 \delta , 11 \delta] \} \,. $$
Using the previous property on the supports of $ \rho^{\initial}_-$, $\rho^{\initial}_+$ and 
 the identity $H(x,y) = h^2$ near the elliptic point (where $\varphi =1$), where $x= h$ and $y = 3h$, we can prove  that 
 $$ \dist (A, \text{supp} \rho^{\initial}_-) \geq 9 \delta. $$
Defining $A_{100} = A \cap \{  \theta \leq \sfrac{1}{100} \}$
$$\diam (A_{100} \cup \text{supp} \rho^{\initial}_+ ) \leq (11 \delta - 5 \delta) +  \delta \leq 7 \delta \,,$$
where we used that $A_{100} \subset \T \times [0, \delta]$ and
  $\diam B = \sup_{x,y \in B} |x-y|$ for  $ B \subset \T^2$.
Therefore, using the heat equation kernel, the relations between $t$, $\nu$ and $\delta$ we have 
$$ \rho^{\heat}_+ (x,y)  \geq \frac{1}{4 \pi \nu t} \exp \left  (- \frac{ (7 \delta )^2}{4 \nu t} \right ) \delta^2  \geq \frac{1}{4 \pi} \exp \left ( - \frac{7^2}{4} \right ) \qquad \forall (x,y) \in A_ {100}.$$
Similarly
$$ |\rho^{\heat}_- (x,y) | \leq  2  \frac{1}{4 \pi} \exp \left ( - \frac{9^2}{4} \right )  \qquad \forall (x,y) \in A \,.$$
Therefore, by definition of $P_0$
$$ P_0 \rho^{\heat} (h) = \int_0^1\rho^{\heat}  (\Phi (\theta, h)) d \theta  \geq \frac{1}{100} \frac{1}{4 \pi} \exp \left ( - \frac{7^2}{4} \right ) - \frac{1}{4 \pi} \exp \left ( - \frac{9^2}{4} \right ) \geq \exp (- 16)$$
for any $h \in [10\delta, 11 \delta]$.
Finally, using that $\| \rho_{\initial} \|_{L^2}^2 \leq \delta^2$ we have 
$$ \|  P_0 \rho^{\heat} (t, \cdot) \|_{L^2}^2 \geq  \int_{10 \delta}^{11 \delta} \exp (- 32) 2 h T(h) dh = 126 \pi \exp (- 32) \delta^2 \geq  \exp (- 32) \| \rho_{\initial}\|_{L^2}^2 \,, $$
concluding the proof.
\end{proof}
\section{Concluding remarks}
\label{sec:conclusion}

In this section we list  some future directions based on our approach. The aim is to provide  a better understanding of the enhanced dissipation effect for general autonomous Hamiltonian flows.

\subsection*{On unbounded period}
In the proof of the pseudospectral bounds for the model problem in Section \ref{sec:pseudo}, we crucially used that the period is bounded. This is useful to prove that 
the sets $\{ E_{k, \lambda, \delta} \}_{k \in \N}$ introduced  in \eqref{def:Eklambda} and \eqref{def:Eklambdahigh} are disjoint (and well separated). 
In particular, the fact that $\Omega(h) \simeq 1$ is needed to conclude that  
$$  \frac{| \lambda |}{|k|} \sim 1,$$
as used in \eqref{bd:lk1}  and in \eqref{eq:period-bdd-used}. If the period is unbounded, meaning that $\Omega \to 0$ (which happens if for instance $H(r)=r^{2+m}$ with $m \geq 1$),  we can modify our proof and define 
$$E_{k, \lambda, \delta} \coloneqq \left\{ (\theta, h) \in \T \times I : \left|\Omega(h) - \frac{\lambda}{k} \right| < \varphi(\delta, k) \right\},  $$
with a $\varphi$ to be determined. Similarly to \eqref{bd:lk0}, it can be shown that  
$$| \Omega (h_k) - \Omega (h_{k'})| \sim \frac{| \lambda | }{|k|} \frac{|k - k'|}{|k '|} - 2 \varphi (\delta, k) \,.$$
Cutting a $\delta$-neighbourhood near the elliptic point as in \eqref{def:Bdelta}, and assuming the behaviour $\Omega (h) \sim h^m$, we get 
$$\frac{| \lambda | }{|k|}  \gtrsim \delta^m. $$
This implies that we need $\varphi(\delta, k) \sim \delta^m/|k|$ to ensure that the sets are disjoint. Essentially, instead of taking $\delta^m$ for small $|k|$ or $\delta^{m-1}/|k|$ for larger $|k|$, it is enough to always use $\delta^m/|k|$. However, repeating our  proof with this definition, one obtains the pseudospectral bound  
$$ \norm{\cL_\lambda f} \gtrsim \nu^{\frac{m+1}{m+3}} \norm{f},$$
where $\lambda_\nu=\nu^{\frac{m+1}{m+3}}$ is the rate expected for a critical point of order $m+1$ instead of the order $m$ we are assuming. 
Indeed, for the radial flow $H(r)=r^{2+m}$ with $m \geq 1$, this is not the optimal bound, which would be given by  $\nu^{\frac{m}{m+2}}$ \cite{Gallay:2021aa}. We believe that this mismatch is due to the fact that we are treating all the $k$-s at once, and a more careful analysis is needed to handle unbounded periods in an optimal way.

For instance, one might try to do $k$-by-$k$ estimates by  studying the operator $P_k\Delta P_k$. However, the commutator estimates and the comparison with the original equation \eqref{eq:advdiff} will be way more challenging. 

\subsection*{On hyperbolic points}
It may be possible to get a pseudospectral bound of the model problem \eqref{eq:model} when the Hamiltonian is allowed to have hyperbolic points (a similar statement as Theorem \ref{th:main2}, but it is not clear what would be the optimal $\lambda_\nu$ with this method). To treat this case there are two technical problems to take into account:
\begin{itemize}
\item Understand the optimal bound of $\| D_x X \|_{L^\infty ((0, T(h)) \times \Omega )}$, which is then used to bound $ \| \partial_h \Phi \|_{L^\infty}$ in the Poincar\'e inequality Lemma \ref{lem:Poincare} (for estimates away from hyperbolic points one would loose $|\log(\nu)|$ factors). Moreover, one would need to adapt the rescaled action angle variables in Section \ref{sec:rescaled-action-angle}, as to optimise the parameters in the Poincaré inequalities. A preliminary computation shows that this is possible  for the cellular flow, but with a rate that is not consistent with the $\nu^{1/2}$ expected from the behavior around the elliptic point, see \cite{Brue:2022aa} as well.
\item The gradient of the Fourier multiplier operator introduced in \eqref{def:chi} and bounded in Lemma \ref{lemma:operator-estimates} degenerates, since  \eqref{eq:DetermiantEstimateFromBelow} is not true near the hyperbolic point. Therefore, the proof of Corollary \ref{cormain2} fails in this context, since it relies on property \ref{item:ClassACond4}, which is not true near the hyperbolic point. Hence, it is not clear how to relate our model problem \eqref{eq:model} to the advection diffusion equation solution  near the hyperbolic point (i.e. the solution to \eqref{eq:advdiff} is close in $L^2$ to the solution of the model problem \eqref{eq:model} in subdiffusive time scales).  
\end{itemize}
 
 \subsection*{Averaged in time commutator estimate}
It would be interesting  to understand if  the estimate \eqref{bd:gcorr} in Theorem \ref{th:main} holds with $\eps$ replaced by $ a(t, \nu)$ depending on time and $\nu >0$ such that $a(t, \nu) \leq a(s, \nu)$ for any $\nu >0$, $t \geq s$ and
$$ \lim_{\nu \to 0} \,  \lim_{t \to T_\nu } \, a (t, \nu) =0  \,,$$
where $T_\nu$ is such that  
$$ \lim_{\nu \to 0}  \, \nu T_\nu=0 \,.$$
This corresponds to proving that the solution of the advection-diffusion is converging to the average along the streamlines on subdiffusive times. We remark that our example in Proposition \ref{thm:example2} provides a non-trivial evolution for $\rho_0$ for short times $t \leq 1$, which indicates a dissipation of $\rho_\perp$. However, their dynamics is coupled and it is not clear how to control their interactions directly.

\subsection*{On initial data concentrated near elliptic points}
Consider an Hamiltonian having at least one non-degenerate elliptic point and containing also hyperbolic points.
We believe it would be interesting to find extra assumptions on the initial data of \eqref{eq:advdiff}, such as being supported close to a non-degenerate elliptic point, so that we can still effectively compare our model problem to the original equation \eqref{eq:advdiff}. The goal would be to apply our pseudospectral bound in Theorem \ref{th:main2} in a suitable bounded domain contaning the non-degenerate elliptic point, with the aim of proving a result in the spirit of Theorem \ref{th:main}.

\subsection*{Propagating smallness of the streamline average}
As shown in Proposition \ref{prop:conditional}, quantitative bounds on $\rho_0$ will automatically improve the bounds on $\rho$. We believe that, starting with $\rho_0^{in}=0$, it would be interesting to obtain uniform bounds of the type $\|\rho_0(t)\|_{L^2}\lesssim \nu^q$ for some $q>0$ and $0<t<\nu^{-p}$ with $1/\lambda_\nu<p\leq 1$. Depending on the specific values of $q$ and $p$, we can run the proof Proposition \ref{prop:conditional} and eventually obtain the $\|\rho(t)\|_{L^2}\lesssim \max\{\nu^{q/2},\nu^{1/(m+2)},\e^{-\delta_*\lambda_\nu t/2}\}$. However, the only quantitative information we have of this type, is that for Hamiltonians in the class \ref{HamiltonianClassPEps-m} $\rho_0$ remains of size $\eps$. To improve this last bound even for perturbations of radial flows, it seems that a much deeper understanding of more general properties is needed. For instance, one would like to understand the averaging properties of the stochastic flow associated to a random perturbation of a given Hamiltonian flow.

\subsection*{Acknowledgements}   The research of MD, MS, CJ was supported by the Swiss State Secretariat for Education, Research and lnnovation (SERI) under contract number MB22.00034 through the project TENSE. MD was acknowledge support by the GNAMPA-INdAM. The authors thank Maria Colombo for useful discussion about the problem.


\subsection*{Conflict of interest} The authors have no competing interests to declare that are relevant to the content of this
article.

\bibliographystyle{siam}
\bibliography{bibendiss}
\end{document}

%% file: macros.tex
\def\dd{{\rm d}}

\def\eps{\varepsilon}
\def\e{{\rm e}} 
\def\de{{\partial}}

\def\RR {\mathbb{R}}
\def\TT {\mathbb{T}}
\def\ZZ {\mathbb{Z}}

\def\Re{{\rm Re}}
\def\Im{{\rm Im}}

\def \dist{{\rm dist}}
\def \diam{{\rm diam}}
\newcommand{\norm}[1]{\left\lVert #1 \right\rVert}

\newcommand{\jap}[1]{\left\langle #1 \right\rangle}

\newcommand{\cL}{\mathcal{L}}

\newcommand{\cE}{\mathcal{E}}

\newcommand{\cI}{\mathcal{I}}

\newcommand{\cA}{\mathcal{A}}

\newcommand{\cW}{\mathcal{W}}
\newcommand{\cX}{\mathcal{X}}

\newcommand{\InverseMap}{\Psi}

\newtheorem{proposition}{Proposition}[section]
\newtheorem{theorem}[proposition]{Theorem}
\newtheorem{corollary}[proposition]{Corollary}
\newtheorem{lemma}[proposition]{Lemma}
\theoremstyle{definition}
\newtheorem{definition}[proposition]{Definition}
\newtheorem{remark}[proposition]{Remark}

\theoremstyle{definition}

\numberwithin{equation}{section}

%% file: DJS_arXiv_final.bbl
\begin{bibdiv}
\begin{biblist}

\bib{Albritton22}{article}{
      author={Albritton, Dallas},
      author={Beekie, Rajendra},
      author={Novack, Matthew},
       title={Enhanced dissipation and {H}\"{o}rmander's hypoellipticity},
        date={2022},
        ISSN={0022-1236,1096-0783},
     journal={J. Funct. Anal.},
      volume={283},
      number={3},
       pages={Paper No. 109522, 38},
         url={https://doi.org/10.1016/j.jfa.2022.109522},
}

\bib{AV23}{article}{
      author={Armstrong, Scott},
      author={Vicol, Vlad},
       title={Anomalous diffusion by fractal homogenization},
        date={2023},
      eprint={2305.05048},
}

\bib{Bedrossina21Blum}{article}{
      author={Bedrossian, Jacob},
      author={Blumenthal, Alex},
      author={Punshon-Smith, Sam},
       title={Almost-sure enhanced dissipation and uniform-in-diffusivity
  exponential mixing for advection-diffusion by stochastic {N}avier-{S}tokes},
        date={2021},
        ISSN={0178-8051,1432-2064},
     journal={Probab. Theory Related Fields},
      volume={179},
      number={3-4},
       pages={777\ndash 834},
         url={https://doi.org/10.1007/s00440-020-01010-8},
}

\bib{Bedrossian17}{article}{
      author={Bedrossian, Jacob},
      author={Coti~Zelati, Michele},
       title={Enhanced dissipation, hypoellipticity, and anomalous small noise
  inviscid limits in shear flows},
        date={2017},
        ISSN={0003-9527,1432-0673},
     journal={Arch. Ration. Mech. Anal.},
      volume={224},
      number={3},
       pages={1161\ndash 1204},
         url={https://doi.org/10.1007/s00205-017-1099-y},
}

\bib{Brue:2022aa}{article}{
      author={Bru{\`e}, Elia},
      author={},
      author={Coti~Zelati, Michele},
      author={Marconi, Elio},
       title={Enhanced dissipation for two-dimensional {Hamiltonian} flows},
        date={2022},
     journal={arXiv preprint arXiv:2211.14057},
         url={https://arxiv.org/pdf/2211.14057.pdf},
}

\bib{BBS23}{article}{
      author={Burczak, Jan},
      author={Sz\'ekelyhidi, L\'aszl\'o},
      author={Wu, Bian},
       title={Anomalous dissipation and {Euler} flows},
        date={2023},
      eprint={2310.02934},
}

\bib{Colombo21}{article}{
      author={Colombo, Maria},
      author={Coti~Zelati, Michele},
      author={Widmayer, Klaus},
       title={Mixing and diffusion for rough shear flows},
        date={2021},
        ISSN={2769-8505},
     journal={Ars Inven. Anal.},
       pages={Paper No. 2, 22},
}

\bib{CCS22}{article}{
      author={Colombo, Maria},
      author={Crippa, Gianluca},
      author={Sorella, Massimo},
       title={Anomalous dissipation and lack of selection in the
  {O}bukhov-{C}orrsin theory of scalar turbulence},
        date={2023},
        ISSN={2524-5317,2199-2576},
     journal={Ann. PDE},
      volume={9},
      number={2},
       pages={Paper No. 21, 48},
         url={https://doi.org/10.1007/s40818-023-00162-9},
}

\bib{constantin2008diffusion}{article}{
      author={Constantin, Peter},
      author={Kiselev, Alexander},
      author={Ryzhik, Lenya},
      author={Zlato{\v{s}}, Andrej},
       title={Diffusion and mixing in fluid flow},
        date={2008},
     journal={Annals of Mathematics},
       pages={643\ndash 674},
}

\bib{CZDE20}{article}{
      author={Coti~Zelati, Michele},
      author={Delgadino, Matias~G.},
      author={Elgindi, Tarek~M.},
       title={On the relation between enhanced dissipation timescales and
  mixing rates},
        date={2020},
        ISSN={0010-3640},
     journal={Comm. Pure Appl. Math.},
      volume={73},
      number={6},
       pages={1205\ndash 1244},
         url={https://doi.org/10.1002/cpa.21831},
}

\bib{CZD20}{article}{
      author={Coti~Zelati, Michele},
      author={Dolce, Michele},
       title={Separation of time-scales in drift-diffusion equations on
  {$\Bbb{R}^2$}},
        date={2020},
        ISSN={0021-7824,1776-3371},
     journal={J. Math. Pures Appl. (9)},
      volume={142},
       pages={58\ndash 75},
         url={https://doi.org/10.1016/j.matpur.2020.08.001},
}

\bib{Gallay:2021aa}{article}{
      author={Coti~Zelati, Michele},
      author={Gallay, Thierry},
       title={Enhanced dissipation and {T}aylor dispersion in
  higher-dimensional parallel shear flows},
        date={2023},
        ISSN={0024-6107,1469-7750},
     journal={J. Lond. Math. Soc. (2)},
      volume={108},
      number={4},
       pages={1358\ndash 1392},
}

\bib{TDTEGIIJ22}{article}{
      author={Drivas, Theodore~D.},
      author={Elgindi, Tarek~M.},
      author={Iyer, Gautam},
      author={Jeong, In-Jee},
       title={Anomalous dissipation in passive scalar transport},
        date={2022},
        ISSN={0003-9527},
     journal={Arch. Ration. Mech. Anal.},
      volume={243},
      number={3},
       pages={1151\ndash 1180},
         url={https://doi.org/10.1007/s00205-021-01736-2},
}

\bib{EL23}{article}{
      author={Elgindi, Tarek~M.},
      author={Liss, Kyle},
       title={Norm growth, non-uniqueness, and anomalous dissipation in passive
  scalars},
        date={2023},
      eprint={2309.08576},
}

\bib{elgindi2023optimal}{article}{
      author={Elgindi, Tarek~M.},
      author={Liss, Kyle},
      author={Mattingly, Jonathan~C.},
       title={Optimal enhanced dissipation and mixing for a time-periodic,
  {Lipschitz} velocity field on $\mathbb{T}^2$},
        date={2023},
      eprint={2304.05374},
}

\bib{MR1265233}{article}{
      author={Fannjiang, Albert},
      author={Papanicolaou, George},
       title={Convection enhanced diffusion for periodic flows},
        date={1994},
        ISSN={0036-1399},
     journal={SIAM J. Appl. Math.},
      volume={54},
      number={2},
       pages={333\ndash 408},
         url={https://doi.org/10.1137/S0036139992236785},
}

\bib{Feng19}{article}{
      author={Feng, Yuanyuan},
      author={Iyer, Gautam},
       title={Dissipation enhancement by mixing},
        date={2019},
        ISSN={0951-7715,1361-6544},
     journal={Nonlinearity},
      volume={32},
      number={5},
       pages={1810\ndash 1851},
         url={https://doi.org/10.1088/1361-6544/ab0e56},
}

\bib{Feng23}{article}{
      author={Feng, Yuanyuan},
      author={Mazzucato, Anna~L.},
      author={Nobili, Camilla},
       title={Enhanced dissipation by circularly symmetric and parallel pipe
  flows},
        date={2023},
        ISSN={0167-2789,1872-8022},
     journal={Phys. D},
      volume={445},
       pages={Paper No. 133640, 13},
         url={https://doi.org/10.1016/j.physd.2022.133640},
}

\bib{Freidlin94}{article}{
      author={Freidlin, Mark~I.},
      author={Wentzell, Alexander~D.},
       title={Random perturbations of {H}amiltonian systems},
        date={1994},
        ISSN={0065-9266,1947-6221},
     journal={Mem. Amer. Math. Soc.},
      volume={109},
      number={523},
       pages={viii+82},
         url={https://doi.org/10.1090/memo/0523},
}

\bib{Gallagher09}{article}{
      author={Gallagher, Isabelle},
      author={Gallay, Thierry},
      author={Nier, Francis},
       title={Spectral asymptotics for large skew-symmetric perturbations of
  the harmonic oscillator},
        date={2009},
        ISSN={1073-7928,1687-0247},
     journal={Int. Math. Res. Not. IMRN},
      number={12},
       pages={2147\ndash 2199},
         url={https://doi.org/10.1093/imrn/rnp013},
}

\bib{gallay2019estimations}{article}{
      author={Gallay, Thierry},
       title={Estimations pseudo-spectrales et stabilit{\'e} des tourbillons
  plans},
        date={2019},
     journal={S{\'e}minaire BOURBAKI},
       pages={72e},
}

\bib{He22}{article}{
      author={He, Siming},
       title={Enhanced dissipation, hypoellipticity for passive scalar
  equations with fractional dissipation},
        date={2022},
        ISSN={0022-1236,1096-0783},
     journal={J. Funct. Anal.},
      volume={282},
      number={3},
       pages={Paper No. 109319, 28},
         url={https://doi.org/10.1016/j.jfa.2021.109319},
}

\bib{Helffer:2010aa}{article}{
      author={Helffer, B.},
      author={Sj\"{o}strand, J.},
       title={From resolvent bounds to semigroup bounds},
        date={2010},
     journal={arXiv preprint arXiv:1001.4171},
         url={https://arxiv.org/pdf/1001.4171.pdf},
}

\bib{HelfferImproving21}{article}{
      author={Helffer, B.},
      author={Sj\"{o}strand, J.},
       title={Improving semigroup bounds with resolvent estimates},
        date={2021},
        ISSN={0378-620X},
     journal={Integral Equations Operator Theory},
      volume={93},
      number={3},
       pages={Paper No. 36, 41},
         url={https://doi.org/10.1007/s00020-021-02652-6},
}

\bib{kumaresan2018advection}{book}{
      author={Kumaresan, Michael},
       title={The advection-diffusion equation and the enhanced dissipation
  effect for flows generated by {Hamiltonians}},
   publisher={PhD Thesis, City University of New York},
        date={2018},
}

\bib{LiPseudo20}{article}{
      author={Li, Te},
      author={Wei, Dongyi},
      author={Zhang, Zhifei},
       title={Pseudospectral and spectral bounds for the {O}seen vortices
  operator},
        date={2020},
        ISSN={0012-9593,1873-2151},
     journal={Ann. Sci. \'Ec. Norm. Sup\'er. (4)},
      volume={53},
      number={4},
       pages={993\ndash 1035},
         url={https://doi.org/10.24033/asens.2438},
}

\bib{rhines1983rapidly}{article}{
      author={Rhines, Peter~B},
      author={Young, William~R},
       title={How rapidly is a passive scalar mixed within closed
  streamlines?},
        date={1983},
     journal={Journal of Fluid Mechanics},
      volume={133},
       pages={133\ndash 145},
}

\bib{Vukadinovic15}{article}{
      author={Vukadinovic, J.},
      author={Dedits, E.},
      author={Poje, A.~C.},
      author={Sch\"{a}fer, T.},
       title={Averaging and spectral properties for the 2{D}
  advection-diffusion equation in the semi-classical limit for vanishing
  diffusivity},
        date={2015},
        ISSN={0167-2789,1872-8022},
     journal={Phys. D},
      volume={310},
       pages={1\ndash 18},
         url={https://doi.org/10.1016/j.physd.2015.07.011},
}

\bib{Vukadinovic21}{article}{
      author={Vukadinovic, Jesenko},
       title={The limit of vanishing diffusivity for passive scalars in
  {H}amiltonian flows},
        date={2021},
        ISSN={0003-9527},
     journal={Arch. Ration. Mech. Anal.},
      volume={242},
      number={3},
       pages={1395\ndash 1444},
         url={https://doi.org/10.1007/s00205-021-01707-7},
}

\bib{WeiDiffusion19}{article}{
      author={Wei, Dongyi},
       title={Diffusion and mixing in fluid flow via the resolvent estimate},
        date={2021},
        ISSN={1674-7283},
     journal={Sci. China Math.},
      volume={64},
      number={3},
       pages={507\ndash 518},
         url={https://doi.org/10.1007/s11425-018-9461-8},
}

\end{biblist}
\end{bibdiv}
